\documentclass[12pt]{article}

\usepackage[a4paper, top=3cm, left=3cm, right=2cm, bottom=2cm]{geometry}
\usepackage{amsmath, amsthm, pb-diagram}
\usepackage{amssymb}
\usepackage{multicol}
\usepackage{amsfonts}
\usepackage{graphics, color, enumerate, fancyhdr, dsfont}

    \newcommand{\dom}{\mbox{\rm dom}}

    \newcommand{\thzfc}{\mathrm{ZFC}}

    \newcommand{\Awf}{\mathcal{A}}
    \newcommand{\Bwf}{\mathcal{B}}
    \newcommand{\Cwf}{\mathcal{C}}
    \newcommand{\Dwf}{\mathcal{D}}
    \newcommand{\Ewf}{\mathcal{E}}
    \newcommand{\Fwf}{\mathcal{F}}
    \newcommand{\Gwf}{\mathcal{G}}
    \newcommand{\Hwf}{\mathcal{H}}
    \newcommand{\Iwf}{\mathcal{I}}
    \newcommand{\Jwf}{\mathcal{J}}
    \newcommand{\Mwf}{\mathcal{M}}
    \newcommand{\Nwf}{\mathcal{N}}
    
    \newcommand{\Pwf}{\mathcal{P}}

    \newcommand{\afrak}{\mathfrak{a}}
    \newcommand{\bfrak}{\mathfrak{b}}
    \newcommand{\cfrak}{\mathfrak{c}}
    \newcommand{\dfrak}{\mathfrak{d}}
    
    \newcommand{\gfrak}{\mathfrak{g}}
    
    \newcommand{\pfrak}{\mathfrak{p}}
    \newcommand{\rfrak}{\mathfrak{r}}
    \newcommand{\sfrak}{\mathfrak{s}}
    
    \newcommand{\ufrak}{\mathfrak{u}}

    \newcommand{\menos}{\smallsetminus}

    \newcommand{\pts}{\mathcal{P}}
    
    \newcommand{\frestr}{\!\!\upharpoonright\!\!}

    \newcommand{\add}{\mbox{\rm add}}
    \newcommand{\cov}{\mbox{\rm cov}}
    \newcommand{\non}{\mbox{\rm non}}
    \newcommand{\cof}{\mbox{\rm cof}}
    
    \newcommand{\limdir}{\mbox{\rm limdir}}

    \newcommand{\Bor}{\mathds{B}}
    \newcommand{\Cor}{\mathds{C}}
    \newcommand{\Dor}{\mathds{D}}
    \newcommand{\Eor}{\mathds{E}}
    \newcommand{\Loc}{\mathds{LOC}}
    \newcommand{\Mor}{\mathds{M}}
    \newcommand{\Por}{\mathds{P}}
    \newcommand{\Qor}{\mathds{Q}}
    \newcommand{\Ror}{\mathds{R}}
    \newcommand{\Sor}{\mathds{S}}
    
    \newcommand{\Pnm}{\dot{\mathds{P}}}
    \newcommand{\Qnm}{\dot{\mathds{Q}}}
    \newcommand{\Rnm}{\dot{\mathds{R}}}
    \newcommand{\Snm}{\dot{\mathds{S}}}

    \newcommand{\Bnm}{\dot{\mathds{B}}}
    \newcommand{\Cnm}{\dot{\mathds{C}}}
    \newcommand{\Dnm}{\dot{\mathds{D}}}
    
    \newcommand{\Locnm}{\dot{\mathds{LOC}}}

    \newcommand{\cf}{\mbox{\rm cf}}

    \newcommand{\sii}{{\ \mbox{$\Leftrightarrow$} \ }}

\title{Template iterations with non-definable ccc forcing notions}
\author{Diego Alejandro Mej\'ia\thanks{Supported by the Monbukagakusho (Ministry of Education, Culture, Sports, Science and Technology) Scholarship, Japan.}}
\date{\small Graduate School of System Informatics, Kobe University, Kobe, Japan.\\ \texttt{damejiag@kurt.scitec.kobe-u.ac.jp}}

\begin{document}

\makeatletter
\def\@roman#1{\romannumeral #1}
\makeatother

\theoremstyle{plain}
  \newtheorem{theorem}{Theorem}[section]
  \newtheorem{corollary}[theorem]{Corollary}
  \newtheorem{lemma}[theorem]{Lemma}
  \newtheorem{prop}[theorem]{Proposition}
  \newtheorem{claim}[theorem]{Claim}
  \newtheorem{conjecture}[theorem]{Conjecture}
  \newtheorem{exer}{Exercise}
\theoremstyle{definition}
  \newtheorem{definition}[theorem]{Definition}
  \newtheorem{example}[theorem]{Example}
  \newtheorem{remark}[theorem]{Remark}
  \newtheorem{context}[theorem]{Context}
  \newtheorem{question}[theorem]{Question}
  \newtheorem{problem}[theorem]{Problem}
  \newtheorem*{acknowledgements}{Acknowledgements}

\maketitle

\begin{abstract}
   We present a version with non-definable forcing notions of Shelah's theory of iterated forcing along a template. Our main result, as an application, is that, if $\kappa$ is a measurable cardinal and $\theta<\kappa<\mu<\lambda$ are uncountable regular cardinals, then there is a ccc poset forcing $\sfrak=\theta<\bfrak=\mu<\afrak=\lambda$. Another application is to get models with large continuum where the groupwise-density number $\gfrak$ assumes an arbitrary regular value.
\end{abstract}


\section{Introduction}\label{SecIntro}

The technique of template iterations was first introduced by Shelah in \cite{shelah} to prove the consistency
of $\dfrak<\afrak$ where $\dfrak$ is the dominating number and $\afrak$ is the almost disjointness number.
There are two approaches to construct the models for that consistency. Shelah first observed that, given a \emph{ccc (countable chain condition)} poset $\Por$ and measurable cardinal $\kappa$ witnessed by a $\kappa$-complete ultrafilter $\Dwf$, forcing with the ultrapower $\Por^\kappa/\Dwf$ destroys the maximality of any almost disjoint family of size $\geq\kappa$ in the $\Por$-extension (see Lemma \ref{DestrMad}), while it preserves all scales of length of cofinality different from $\kappa$ (this is an easy consequence of Lemma \ref{Sigma1-1andUltrapow}). Therefore, if $\Por$ is the \emph{fsi (finite support iteration)} of length $\mu>\kappa$ of Hechler forcing with $\mu$ regular then, by taking ultrapowers $\lambda$-many times for $\lambda>\mu$ regular with $\lambda^{\kappa}=\lambda$ (with special care in the limit steps), the obtained poset forces $\bfrak=\dfrak=\mu<\afrak=\cfrak=\lambda$, where $\bfrak$ is the bounding number and $\cfrak=2^{\aleph_0}$ is the size of the continuum. Although these ultrapowers can be represented by iterations along a template, to prove the consistency statement it is not necessary to look into their template structure, but it is enough to understand its forcing equivalence with a ccc fsi. This approach can be used to get the consistency of $\ufrak<\afrak$ modulo a measurable (see also \cite{brendle3}), where $\ufrak$ is the ultrafilter number, by starting with a fsi of Laver-Prikry type forcings with ultrafilters, but, as these forcing notions are not ($\boldsymbol{\Sigma}^1_1$) definable, it is not known whether the corresponding fsi's can be constructed by template iterations.

The second approach consists in defining a template iteration where the ultrapower argument to increase $\afrak$ is replaced by an isomorphism-of-names argument, so the consistency result can be obtained modulo $\thzfc$ alone.
Concretely, if $\aleph_1<\mu<\lambda$ are regular cardinals and $\lambda^{\aleph_0}=\lambda$, the statement $\bfrak=\dfrak=\mu<\afrak=\lambda$
is consistent by this method. However, it is still not known whether this approach can be applied to get the consistency of $\ufrak<\afrak$ on the basis of $\thzfc$, which is still an open problem. All details of this discussion can also be found in \cite{br} and \cite{brendle3}.

In this paper, we investigate to what extent it is possible to obtain, with the techniques discussed above, models where the values of $\bfrak$, $\afrak$ and the splitting number $\sfrak$ can be separated. Fix (only in this paragraph) $\theta\leq\mu<\lambda$ uncountable regular cardinals. The simplest of these results is the consistency of $\sfrak<\bfrak=\cfrak$ by a fsi of Hechler forcing (see \cite{baudor}), even more, using techniques from \cite{brendle}, there is a ccc fsi forcing $\sfrak=\theta\leq\bfrak=\cfrak=\mu$. Shelah \cite{sh84} proved, by countable support iteration techniques, the consistency of $\bfrak=\aleph_1<\afrak=\sfrak=\aleph_2$ and the consistency of $\bfrak=\afrak=\aleph_1<\sfrak=\aleph_2$. Extensions of these results are the consistency of $\bfrak=\mu<\afrak=\mu^+$ obtained by Brendle \cite{breMob} with fsi techniques and, using matrix iterations, Brendle and Fischer \cite{BF} proved the consistency of $\bfrak=\afrak=\theta\leq\sfrak=\mu$ with $\thzfc$ and the consistency of $\kappa<\bfrak=\mu<\afrak=\sfrak=\lambda$ where $\kappa$ is measurable in the ground model (here, the ultrapower technique explained above is also used). In Shelah's model for the consistency of $\ufrak<\afrak$ mentioned above, it is also true that $\kappa<\bfrak=\sfrak=\ufrak=\mu<\afrak=\lambda$ where $\kappa$ is measurable in the ground model. The consistency of $\bfrak=\sfrak=\aleph_1<\afrak=\aleph_2$ with $\thzfc$ is still and open problem (see \cite{BrRa}).

Concerning models where $\sfrak$, $\bfrak$ and $\afrak$ are different, these are the possibilities.
\begin{problem}\label{MainProb}
   Let $\theta<\mu<\lambda$ be uncountable regular cardinals. Is it consistent that
   \begin{enumerate}[(1)]
      \item $\bfrak=\theta<\afrak=\mu<\sfrak=\lambda$?
      \item $\bfrak=\theta<\sfrak=\mu<\afrak=\lambda$?
      \item $\sfrak=\theta<\bfrak=\mu<\afrak=\lambda$?
   \end{enumerate}
\end{problem}
As models for $\bfrak<\sfrak$ and $\bfrak<\afrak$ are hard to get, many difficulties arise to answer each question of this problem. In the case of (1) and (2), three dimensional iteration constructions may work, but it is not known how to guarantee complete embeddability between the intermediate stages. To answer (3), we may use the known techniques for obtaining posets that force $\bfrak<\afrak$ with large continuum and guarantee that these preserve splitting families of the ground model. This is not viable for the techniques of \cite{breMob}, so we are left with the elaborated technique of iterations along a template.

In the models constructed in both approaches explained at the beginning of this introduction, $\sfrak$ is preserved to be equal to $\aleph_1$ (see Remark \ref{ShTempMainThm} for details), so the consistency of (3) with $\thzfc$ is true for $\theta=\aleph_1$.

We obtain a partial answer to (3) for larger $\theta$, which is the main result of this text. By a forcing construction as in the first approach above, given a measurable cardinal $\kappa$ and regular uncountable cardinals $\theta<\kappa<\mu<\lambda$ with $\theta^{<\theta}=\theta$ and $\lambda^\kappa=\lambda$, we construct a ccc poset that forces $\sfrak=\theta<\bfrak=\mu<\afrak=\lambda$.  Here, it is needed that the resulting poset preserves $\sfrak\leq\theta$ and, moreover, we need to use posets with small filter subbases (of size $<\theta$) along the iteration, like  Mathias-Prikry type or Laver-Prikry type posets, to ensure that $\sfrak\geq\theta$ in the final extension. Although this construction can be done without using the template structure of the iterations, it seems that knowledge about the template is necessary to get an easy proof of the preservation of $\sfrak\leq\theta$ in the final extension.

As these posets with small filter subbases are non-definable and Shelah's theory of template iterations just applies for definable (Suslin) ccc posets, we need to expand this theory in order to be able to include non-definable posets. This is the main technical achievement of this paper and it is presented in such generality that it can be used for other purposes. For instance, we use this to obtain models where the groupwise-density number $\gfrak$ can assume an arbitrary regular value, even in models obtained by well known ccc fsi techniques. Concerning this, from results in \cite{Bl}, it is known how to force $\gfrak=\aleph_1$ by a fsi of Suslin ccc posets that adds new reals at many intermediate stages. Our application is an extension of this argument to force $\gfrak$ to be an arbitrary regular uncountable cardinal by quite arbitrary template iteration constructions.

This paper is structured as follows. Section \ref{SecPrel} contains definitions, notation and results that are considered preliminaries for this text (most of the concepts used in this introduction are defined there). Besides, special emphasis is made in \emph{correctness}, notion introduced by Brendle \cite{brendle2} which is essential to construct our template iterations with non-definable posets.  We introduce, in Section \ref{SecTemp}, the basic definitions of, and results about, the templates that are used as supports for the iterations in this paper. In Section \ref{SecItTemp} we present a version of Shelah's theory of iterated forcing along a template for non-definable forcing notions, plus some basic results about ccc-ness, regular contention and equivalence for posets constructed from template iterations. Most of the concepts and results of Sections \ref{SecTemp} and \ref{SecItTemp} are due to Shelah and many proofs of the extended results are not that different from the original proofs, which can be found in \cite{br} and \cite{brendle2}.

In Section \ref{SecPres} we extend some preservation results for fsi included in \cite[Sect. 6.4 and 6.5]{barju} to the context of template iterations. Sections \ref{SecApplg} and \ref{SecAppl} are devoted to our applications: in the former, we show how to obtain an arbitrary regular value for the groupwise-density number in models constructed by ccc fsi (but now in the context of templates), which are slight modifications of some models presented in \cite[Sect. 3]{mejia} and \cite[Sect. 4]{mejia02}; in the latter, we prove Theorem \ref{AppSplitting}, which is our main result. Section \ref{SecQ} contains questions and discussions about the material of this text.


\section{Preliminaries and elementary forcing}\label{SecPrel}

This section is divided intro three parts. In the first subsection, we define the classical cardinal invariants that we consider in this text. Next, we present basic facts about forcing theory with special emphasis in the notion of \emph{correctness} and its interplay with Suslin ccc posets, forcing quotients and iterations. Regarding this notion, we prove most of the stated results. In the last part, we introduce basic facts, due to Shelah \cite{shelah}, about forcing with ultrapowers.

Our notation is quite standard. Given a cardinal number $\mu$, $[X]^{<\mu}$ denotes the collection of all the subsets of $X$ of size $<\mu$. Likewise, define $[X]^{\leq\mu}$ and $[X]^\mu$, the latter being the collection of all the subsets of $X$ of size $\mu$. If we consider the product $\prod_{i\in I}X_i$, for a function $p\in\prod_{i\in J}X_i$ where $J\subseteq I$, denote by $[p]:=\{x\in\prod_{i\in I}X_i\ /\ p\subseteq x\}$. Also, if $k\in I\menos J$ and $z\in X_k$, let $p\widehat{\ \ }\langle z\rangle_{k}$ be the function that extends $p$ with domain $J\cup\{ k\}$ and whose $k$-th component is $z$. In the case where $I=\delta$ and $J=\alpha<\delta$ are ordinals, $p\widehat{\ \ }\langle z\rangle=p\widehat{\ \ }\langle z\rangle_\alpha$. Say that $\bar{J}=\langle J_n\rangle_{n<\omega}$ is an \emph{interval partition of $\omega$} if it is a partition of $\omega$ into non-empty finite intervals such that $\max(J_n)+1=\min(J_{n+1})$ for all $n<\omega$.

Given a formula $\varphi(x)$ of the language of $\thzfc$, $\forall^\infty_{n<\omega}\varphi(n)$ means that $\varphi(n)$ holds for all but finitely many $n<\omega$. $\exists^\infty_{n<\omega}\varphi(n)$ means that infinitely many $n<\omega$ satisfy $\varphi(n)$.

Throughout this text, we refer as a \emph{real} to any member of a fixed uncountable Polish space (like the Baire space $\omega^\omega$ or the Cantor space $2^\omega$).

\subsection{Cardinal invariants}\label{SubsecInv}

For proofs and further information about the classical cardinal invariants defined in this subsection, see \cite{barju}, \cite{bart} and \cite{blass}.

For $f,g\in\omega^\omega$, define $f\leq^* g$ as $\forall^\infty_{n<\omega}(f(n)\leq g(n))$, which is read \emph{$f$ is dominated by $g$}. $f\leq g$ means that $\forall_{n<\omega}(f(n)\leq g(n))$. $D\subseteq\omega^\omega$ is a \emph{dominating family} if any function in $\omega^\omega$ is dominated by some function in $D$. $\bfrak$, the \emph{(un)bounding number}, is the least size of a subset $Y\subseteq\omega^\omega$ such that there is no function in $\omega^\omega$ that dominates all the members of $Y$. Dually, $\dfrak$, the \emph{dominating number}, is the least size of a dominating family.

For $a,x\in[\omega]^\omega$, \emph{$a$ splits $x$} means that $a\cap x$ and $x\menos a$ are infinite. $S\subseteq[\omega]^\omega$ is a \emph{splitting family} if any infinite subset of $\omega$ is splitted by some member of $S$. For $x\in[\omega]^\omega$ and $F\subseteq[\omega]^\omega$, we say that \emph{$x$ reaps $F$} if $x$ splits all the sets in $F$. $\sfrak$, the \emph{splitting number}, is the least size of a splitting family. Dually, $\rfrak$, the \emph{reaping number}, is defined as the least size of a subset of $[\omega]^\omega$ that cannot be ripped by one infinite subset of $\omega$.

A family $\Awf\subseteq[\omega]^\omega$ is said to be \emph{almost disjoint (a.d.)} if the intersection of any two different members of $\Awf$ is finite. A maximal family of this kind is called \emph{maximal almost disjoint (mad)}. $\afrak$, the \emph{almost disjointness number}, is defined as the least size of an infinite mad family.

For $A$ and $B$ subsets of $\omega$, $A\subseteq^* B$ denotes that $A\menos B$ is finite. If $\Cwf\subseteq[\omega]^{\omega}$, say that $X\in[\omega]^\omega$ is a \emph{pseudo-intersection of $\Cwf$} if $X\subseteq^* A$ for any $A\in\Cwf$. Recall that $\Fwf\subseteq[\omega]^{\omega}$ is a \emph{filter subbase} if the intersection of any finite subfamily of $\Fwf$ is infinite. Denote $\hat{\Fwf}:=\{\bigcap F\ /\ F\in[\Fwf]^{<\omega}\}$ and $\mathrm{gen}(\Fwf)=\left\{X\in[\omega]^\omega\ /\ \exists_{A\in\hat{\Fwf}}(A\subseteq^* X)\right\}$ which is the filter generated by $\Fwf$. Say that $\Fwf$ is a \emph{filter base} if it is a filter subbase and $\hat{\Fwf}=\Fwf$. Define the following cardinal invariants.

\begin{description}
   \item[$\pfrak$,] the \emph{pseudo-intersection number}, which is the least size of a filter subbase that does not have a pseudo-intersection.
   \item[$\ufrak$,] the \emph{ultrafilter number}, which is the least size of a filter subbase that generates a non-principal ultrafilter on $\omega$
\end{description}

From now on, we only consider filter bases that contain all cofinite subsets of $\omega$. The cardinal numbers $\pfrak$ and $\ufrak$ are not altered when restricted to filter bases.

$\Gwf\subseteq[\omega]^\omega$ is \emph{groupwise-dense} if $\Gwf$ is downward closed under $\subseteq^*$ and, for any interval partition $\langle I_n\rangle_{n<\omega}$ of $\omega$, there exists an $A\in[\omega]^\omega$ such that $\bigcup_{n\in A}I_n\in\Gwf$. The \emph{groupwise-density number} $\gfrak$ is defined as the least size of a family of groupwise-dense sets whose intersection is empty.

For an uncountable Polish space with a Borel probability measure, let $\Mwf$ be the $\sigma$-ideal of meager sets and $\Nwf$ be the $\sigma$-ideal of null sets (from the context, it will be clear which Polish space corresponds to such an ideal). For $\Iwf$ being $\Mwf$ or $\Nwf$, the following cardinal invariants are defined, whose values do not depend on the underlying Polish space:
\begin{description}
   \item[$\add(\Iwf)$:] \emph{The additivity of $\Iwf$}, which is the least size of a family $\Fwf\subseteq\Iwf$ whose union is not in $\Iwf$.
   \item[$\cov(\Iwf)$:] \emph{The covering of $\Iwf$}, which is the least size of a family $\Fwf\subseteq\Iwf$ whose union covers all the reals.
   \item[$\non(\Iwf)$:] \emph{The uniformity of $\Iwf$}, which is the least size of a set of reals not in $\Iwf$.
   \item[$\cof(\Iwf)$:] \emph{The cofinality of $\Iwf$}, which is the least size of a cofinal subfamily of $\langle\Iwf,\subseteq\rangle$.
\end{description}
\begin{figure}
\begin{center}
  \includegraphics{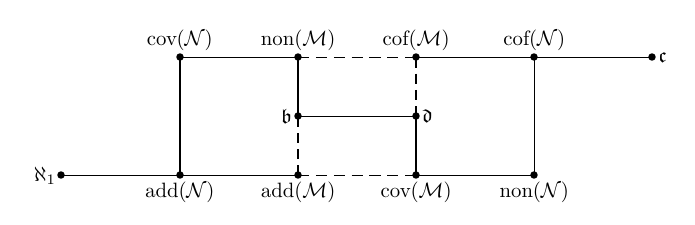}
\caption{Cicho\'n's diagram}
\label{fig:1}
\end{center}
\end{figure}

The cardinal invariants $\add(\Nwf)$, $\add(\Mwf)$, $\bfrak$, $\pfrak$ and $\gfrak$ are regular in $\thzfc$.

The following results correspond to characterizations of the additivity-cofinality of measure and the covering-uniformity of category. To fix some notation, functions $\psi:\omega\to[\omega]^{<\omega}$ are often called \emph{slaloms}. For $x\in\omega^\omega$ and an slalom $\psi$, define $x\in^*\psi$ by $\forall^\infty_{n<\omega}(x(n)\in\psi(n))$, which is read \emph{$\psi$ localizes $x$}. For a function $h:\omega\to\omega$ denote by $S(\omega,h)$ the set of all slaloms $\psi$ such that $\forall_{n<\omega}(|\psi(n)|\leq h(n))$.

\begin{theorem}[Bartoszy\'{n}ski's characterization {\cite[Thm. 2.3.9]{barju}}]\label{BartChar}
   Let $h\in\omega^\omega$ that converges to infinity. Then,
   \begin{enumerate}[(a)]
      \item $\add(\Nwf)$ is the least size of a set $Y\subseteq\omega^\omega$ such that no slalom in $S(\omega,h)$ localizes all the reals in $Y$.
      \item $\cof(\Nwf)$ is the least size of a family $\Fwf\subseteq S(\omega,h)$ with the property that every real in $\omega^\omega$ is localized by some slalom in $\Fwf$.
   \end{enumerate}
\end{theorem}

\begin{theorem}[{\cite[Thm. 2.4.1 and 2.4.7]{barju}}]\label{Catchar}
   \begin{enumerate}[(a)]
     \item $\non(\Mwf)$ is the least size of a family $\Fwf\subseteq\omega^\omega$ such that, for any $x\in\omega^\omega$, there is an $f\in\Fwf$ such that $\exists^\infty_{n<\omega}(f(n)=x(n))$.
     \item $\cov(\Mwf)$ is the least size size of a family $\Ewf\subseteq\omega^\omega$ such that, for any $x\in\omega^\omega$, there is a $y\in\Ewf$ such that $\forall^\infty_{n<\omega}(x(n)\neq y(n))$.
   \end{enumerate}
\end{theorem}

Recall the typical inequalities between these cardinal invariants that are true in $\thzfc$. Clearly, they are between $\aleph_1$ and $\cfrak$. The well known Cicho\'n's diagram (Figure \ref{fig:1}) illustrates the provable inequalities in $\thzfc$, where vertical lines from bottom to top and horizontal lines from left to right represent $\leq$. Also, the dotted lines mean $\add(\Mwf)=\min\{\bfrak,\cov(\Mwf)\}$ and $\cof(\Mwf)=\max\{\dfrak,\non(\Mwf)\}$. We also know that $\pfrak\leq\add(\Mwf)$, $\pfrak\leq\sfrak$, $\pfrak\leq\gfrak$, $\sfrak\leq\dfrak$, $\gfrak\leq\dfrak$, $\bfrak\leq\afrak$, $\bfrak\leq\rfrak$, $\sfrak\leq\non(\Iwf)$, $\cov(\Iwf)\leq\rfrak$ (where $\Iwf$ is $\Mwf$ or $\Nwf$) and $\rfrak\leq\ufrak$. No other inequalities can be proved in $\thzfc$ between these cardinal invariants.

\subsection{Forcing theory}\label{SubsecForcing}

Basic notation and knowledge about forcing can be found in \cite{jech} and \cite{kunen}. See also \cite[Ch. 3]{barju} for further information about Suslin ccc forcing.

Fix posets $\Por\subseteq\Qor$. Recall that $\Por$ is a \emph{regular subposet} of $\Qor$ if
\begin{enumerate}[(i)]
   \item for any $p,p'\in\Por$, $p\leq_{\Por} p'$ iff $p\leq_{\Qor}p'$,
   \item for any $p,p'\in\Por$, $p\perp_{\Por} p'$ implies $p\perp_{\Qor}p'$ and
   \item for every $q\in\Qor$ there is a $p\in\Por$ such that any condition in $\Por$ stronger than $p$ is compatible with $q$ in $\Qor$. Here, we call \emph{$p$ a reduction\footnote{Also known as \emph{pseudo-projection}.} of $q$ (with respect to $\Por,\Qor$)} (see, for example, \cite[Ch. VII Def. 7.1]{kunen}).
\end{enumerate}
Conditions (ii) and (iii) can be replaced by: any maximal antichain in $\Por$ is a maximal antichain in $\Qor$. When only (i) and (ii) hold and $\mathds{1}_\Qor$, \emph{the trivial condition in $\Qor$}, is in $\Por$, we say that \emph{$\Por$ is a subposet of $\Qor$}. Though notation $\Por\lessdot\Qor$ means, in general, that any completion of $\Por$ is completely embedded into any completion of $\Qor$, in this text we reserve this notation to mean that $\Por$ is a regular subposet of $\Qor$. On the other hand, $\Por\simeq\Qor$ means that $\Por$ and $\Qor$ are \emph{forcing equivalent}, that is, their completions are isomorphic.

We define a restricted version of regular subposet. Fix $M$ a transitive model of $\thzfc$ such that $\Por\in M$. Say that \emph{$\Por$ is a regular subposet of $\Qor$ with respect to $M$}, denoted by $\Por\lessdot_M\Qor$, if (i) holds and any maximal antichain in $\Por$ that belongs to $M$ is a maximal antichain in $\Qor$. One of the features of this notion is that, when $N\supseteq M$ is a transitive model of $\thzfc$ and $\Qor\in N$, if $H$ is $\Qor$-generic over $N$, then $\Por\cap H$ is $\Por$-generic over $M$ and $M[\Por\cap H]\subseteq N[H]$. Of course, any $\Por$-name $\dot{x}\in M$ is also a $\Qor$-name and $\mathrm{val}(\dot{x},\Por\cap H)=\mathrm{val}(\dot{x},H)\in M[\Por\cap H]$ (the generic sets $\Por\cap H$ and $H$ interpret the name $\dot{x}$ as the same object). Very often, we use $M^{\Por}$ to denote a generic extension of $M$ by $\Por$.

Recall the following stronger versions of the countable chain condition of a poset.

\begin{definition}\label{Defcenteredetc}
   Let $\mu$ be an infinite cardinal.
   \begin{enumerate}[(1)]
      \item For $n<\omega$, $B\subseteq\Por$ is \emph{$n$-linked} if, for every $F\subseteq B$ of size $\leq n$, $\exists_{p\in\Por}\forall_{q\in F}(p\leq q)$.
      \item $C\subseteq\Por$ is \emph{centered} if it is $n$-linked for every $n<\omega$.
      \item $\Por$ is \emph{$\mu$-linked} if it is the union of $\leq\mu$ many $2$-linked subsets of $\Por$. In the case $\mu=\aleph_0$, we say \emph{$\sigma$-linked}.
      \item $\Por$ is \emph{$\mu$-centered} if it is the union of $\leq\mu$ many centered subsets of $\Por$. In the case $\mu=\aleph_0$, we say \emph{$\sigma$-centered}.
      \item $\Por$ is \emph{$\mu$-Knaster} if, for every sequence $\{p_\alpha\}_{\alpha<\mu}$ of conditions in $\Por$, there is an $A\subseteq\mu$ of size $\mu$ such that $\{p_\alpha\ /\ \alpha\in A\}$ is $2$-linked. For $\mu=\aleph_1$, we just say \emph{Knaster}.
   \end{enumerate}
   Note that $\mu$-centered implies $\mu$-linked, and $\mu$-Knaster implies $\mu$-cc. Also, $\mu$-linked implies $\mu^+$-Knaster.
\end{definition}

\begin{definition}[Mathias-Prikry type forcing (see, for example, {\cite{blsh}})]\label{DefMatLavUf}
   Let $\Fwf$ be a filter base. \emph{Mathias-Prikry forcing with $\Fwf$} is the poset $\Mor_\Fwf=\{(s,A)\ /\ s\in[\omega]^{<\omega},\ A\in\Fwf,\ \sup(s+1)\leq\min(A)\}$ (where $s+1=\{k+1\ /\ k\in s\}$) ordered by $(t,B)\leq(s,A)$ iff $s\subseteq t$, $B\subseteq A$ and $t\menos s\subseteq A$.
\end{definition}

This forcing is $\sigma$-centered and it adds a pseudo-intersection of $\Fwf$, which is often referred as the \emph{Mathias-Prikry real added by $\Mor_{\Fwf}$}.

\begin{definition}[Suslin ccc poset]\label{DefSuslinposet}
   A \emph{Suslin ccc poset} $\Sor$ is a ccc poset, whose conditions are reals (in some fixed uncountable Polish space), such that the relations $\leq$ and $\perp$ are $\boldsymbol{\Sigma}_1^1$.
\end{definition}

Note that $\Sor$ itself is a $\boldsymbol{\Sigma}_1^1$-set because $x\in\Sor$ iff $x\leq x$. We even have `Suslin' definitions for $\sigma$-linked and $\sigma$-centered for Suslin ccc posets.

\begin{definition}[Brendle {\cite{brendle2}}]\label{DefSuslinLinked}
   Let $\Sor$ be a Suslin ccc poset.
   \begin{enumerate}[(1)]
      \item $\Sor$ is \emph{Suslin $\sigma$-linked} if there exists a sequence $\{S_n\}_{n<\omega}$ of 2-linked subsets of $\Sor$ such that the statement ``$x\in S_n$" is $\boldsymbol{\Sigma}^1_1$. Here, note that the statement ``$S_n$ is 2-linked" is $\boldsymbol{\Pi}_1^1$.
      \item $\Sor$ is \emph{Suslin $\sigma$-centered} if there exists a sequence $\{S_n\}_{n<\omega}$ of centered subsets of $\Sor$ such that the statement ``$x\in S_n$" is $\boldsymbol{\Sigma}^1_1$. Here, note that the statement ``$S_n$ is centered" is $\boldsymbol{\Pi}_2^1$, this because the statement ``$p_0,\ldots,p_l$ have a common stronger condition in $\Sor$" is $\boldsymbol{\Sigma}^1_1$.
   \end{enumerate}
\end{definition}

The following are well known Suslin ccc notions that are used in our applications. It is easy to note that, for each of them, the order relation and the incompatibility relation are Borel.
\begin{itemize}
   \item \emph{Cohen forcing $\Cor$}, which is equivalent to any atomless countable poset.
   \item \emph{Random forcing $\Bor$,} whose conditions are Borel non-null subsets\footnote{Any uncountable Polish space with a Borel probability measure that makes singletons null can be used instead.} of $2^\omega$ ordered by $\subseteq$.
   \item \emph{Hechler forcing $\Dor$}, whose conditions are of the form $(s,f)$ where $s\in\omega^{<\omega}$, $f\in\omega^\omega$ and $s\subseteq f$. The order is defined by $(t,g)\leq(s,f)$ iff $s\subseteq t$ and $f\leq g$.
   \item $\Eor=\omega^{<\omega}\times[\omega^\omega]^{<\omega}$, ordered by $(s',F')\leq(s,F)$ iff $s\subseteq s'$, $F\subseteq F'$ and $\forall_{i\in|s'|\menos|s|}(s'(i)\neq x(i))$ for any $x\in F$, is the \emph{standard ccc poset that adds an eventually different real}.
   \item Let $h:\omega\to\omega$ non-decreasing and converging to infinity. $\Loc^h$, the \emph{localization forcing at $h$}, consists of conditions of the form $(s,F)$ where $s\in\prod_{i<n}[\omega]^{\leq h(i)}$ and $F\in[\omega^\omega]^{\leq h(n)}$ for some $n<\omega$. The order is $(s',F')\leq(s,F)$ iff $s\subseteq s'$, $F\subseteq F'$ and $\{x(i)\ /\ x\in F\}\subseteq s'(i)$ for all $i\in|s'|\menos|s|$. $\Loc:=\Loc^{id}$ where $id:\omega\to\omega$ is the identity function.
\end{itemize}
Note that $\Cor$, $\Dor$ and $\Eor$ are Suslin $\sigma$-centered, while $\Loc^h$ and $\Bor$ are Suslin $\sigma$-linked. Moreover, for each of these posets, the statement ``$p_0,\ldots,p_l$ have a common stronger condition" is Borel. Then, ``$S$ is centered" is $\boldsymbol{\Pi}_1^1$ for any $\boldsymbol{\Sigma}^1_1$-subset $S$ of such a poset.

\begin{lemma}\label{Suslin2stepit}
Let $M\subseteq N$ be transitive models of $\thzfc$. If $\Sor$ is a Suslin ccc poset coded in $M$ then $\Sor^M\lessdot_M\Sor^N$.
\end{lemma}

Now, we introduce the notion of \emph{correctness}. This was originally defined by Brendle \cite{brendle2} in the context of complete Boolean algebras, but here we translate it in the context of posets. We use this notion to describe our general template iteration construction in Section \ref{SecItTemp}.

For the rest of this section, fix $M$ a transitive model of $\thzfc$.

\begin{definition}[Correct diagram of posets]\label{DefCorr}
   For $i=0,1$, let $\Por_i$ and $\Qor_i$ be posets. In (1) and (2), $\langle\Por_0,\Por_1,\Qor_0,\Qor_1\rangle$ represents the diagram of Figure \ref{fig:2}.
   \begin{figure}
     \begin{center}
         \includegraphics{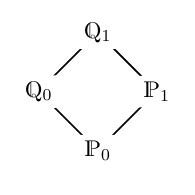}
         \caption{Diagram of posets}
         \label{fig:2}
     \end{center}
   \end{figure}
   \begin{enumerate}[(1)]
      \item When $\Por_i\lessdot\Qor_i$ for $i=0,1$, $\Por_0\lessdot\Por_1$ and $\Qor_0\lessdot\Qor_1$, say that the diagram $\langle\Por_0,\Por_1,\Qor_0,\Qor_1\rangle$  is \emph{correct} if, for each $q\in\Qor_0$ and $p\in\Por_1$, if both have a common reduction in $\Por_0$, then they are compatible in $\Qor_1$. An equivalent formulation is that, whenever $p_0\in\Por_0$ is a reduction of $p_1\in\Por_1$, then $p_0$ is a reduction of $p_1$ with respect to $\Qor_0,\Qor_1$.
      \item We consider a restriction of (1) to the model $M$. If $\Por_0,\Por_1\in M$, $\Por_0\lessdot\Por_1$, $\Qor_0\lessdot\Qor_1$ and $\Por_i\lessdot_M\Qor_i$ for $i=0,1$, the diagram $\langle\Por_0,\Por_1,\Qor_0,\Qor_1\rangle$ is \emph{correct with respect to $M$} iff, whenever $p_0\in\Por_0$ is a reduction of $p_1\in\Por_1$, then $p_0$ is a reduction of $p_1$ with respect to $\Qor_0,\Qor_1$.
   \end{enumerate}
\end{definition}

\begin{remark}
  (1) is a particular case of (2), in fact, in the context of (2), if $\langle\Por_0,\Por_1,\Qor_0,\Qor_1\rangle$ is correct with respect to $M$ and $\Por_i,\Qor_i\in M$ for $i=0,1$, then $\Por_i\lessdot\Qor_i$ and $\langle\Por_0,\Por_1,\Qor_0,\Qor_1\rangle$ is correct. Though the results in this section are stated for the notion (2), most of their applications are done in the context of (1).
\end{remark}

Note that $\langle\Por_0,\Por_1,\Qor_0,\Qor_1\rangle$ is correct iff $\langle\Por_0,\Qor_0,\Por_1,\Qor_1\rangle$ is too, but this symmetry is not true in general for the restricted notion of correctness.

In the context of (2), if $\Por_0=\Por_1$ then $\langle\Por_0,\Por_1,\Qor_0,\Qor_1\rangle$ is always correct with respect to $M$.

In the remaining of this subsection, we state and prove many general results about correctness in relation with forcing iterations and quotients. At the end, we present some results about preserving correct diagrams under two-step iterations, in particular, in relation to the notion of \emph{correctness preserving Suslin ccc poset} introduced by Brendle \cite{brendle2}. This material is important to guarantee that our examples of template iterations are well defined and to have tools to prove our preservation theorems in Section \ref{SecPres}.

\begin{lemma}\label{2stepitemb}
   Let $\Por\in M$, $\Por'$ be posets such that $\Por\lessdot_M\Por'$. If $\Qnm\in M$ is a $\Por$-name of a poset, $\Qnm'$ is a $\Por'$-name of a poset
   and $\Por'$ forces that $\Qnm\lessdot_{M^\Por}\Qnm'$, then $\Por\ast\Qnm\lessdot_M\Por'\ast\Qnm'$. Moreover, $\langle\Por,\Por\ast\Qnm,\Por',\Por'\ast\Qnm'\rangle$ is a correct diagram with respect to $M$.
\end{lemma}
\begin{proof}
   This is quite elementary. See, for example, \cite[Lemma 13]{BF}.
\end{proof}

Recall the notion of quotients of posets. Let $\Por$ and $\Qor$ be posets, $\Por\lessdot\Qor$. Define the quotient $\Qor/\Por:=\big\{q\in\Qor\ /\ \exists_{p\in\dot{G}}(p$ is a reduction of $q\big\}$, which is a $\Por$-name of a poset which inherits the same order as $\Qor$, where $\dot{G}$ is the canonical $\Por$-name for the $\Por$-generic set. Note that $p\in\Por$ is a reduction of $q\in\Qor$ iff $p\Vdash_{\Por}q\in\Qor/\Por$.

\begin{lemma}\label{QuotEqv}
   If $\Por\lessdot\Qor$ are posets then $\Qor\simeq\Por\ast(\Qor/\Por)$. Moreover, if $q\in\Qor$ and $\varphi$ is a formula in the forcing language, $q\Vdash_\Qor\varphi$ iff $(p,q)\Vdash_{\Por\ast(\Qor/\Por)}\varphi$ for any reduction $p\in\Por$ of $q$.
\end{lemma}


\begin{lemma}\label{2stepofQuot}
   Let $\Por\lessdot\Qor\lessdot\Ror$ be posets. Then, $\Por$ forces $\Ror/\Por\simeq(\Qor/\Por)\ast(\Ror/\Qor)$.
\end{lemma}


\begin{lemma}\label{QuotEmb}
   Let $\langle\Por_0,\Por_1,\Qor_0,\Qor_1\rangle$ be a correct diagram with respect to $M$ (where $\Por_0,\Por_1\in M$). Then $\Qor_0$ forces $\Por_1/\Por_0\lessdot_{M^{\Por_0}}\Qor_1/\Qor_0$.
\end{lemma}
\begin{proof}
   Correctness implies directly that $\Vdash_{\Qor_0}\Por_1/\Por_0\subseteq\Qor_1/\Qor_0$. We prove first that $\Qor_0$ forces that any pair of incompatible conditions in $\Por_1/\Por_0$ are incompatible in $\Qor_1/\Qor_0$. Let $q_0\in\Qor_0$, $p_1,p'_1\in\Por_1$ and $q_1\in\Qor_1$ be such that $q_0\Vdash_{\Qor_0}``p_1,p'_1\in\Por_1/\Por_0\textrm{,\ }q_1\in\Qor_1/\Qor_0\textrm{\ and\ }q_1\leq p_1,p'_1"$. We need to find a $q'_0\leq q_0$ in $\Qor_0$ which forces that $p_1$ and $p'_1$ are compatible in $\Por_1/\Por_0$. Work within $M$. Let $D\subseteq\Por_0$ be the set of conditions $p_0$ such that, either it is a reduction of some $p\leq p_1,p'_1$ in $\Por_1$, or $p_0$ is incompatible with all $p\leq p_1,p'_1$ in $\Por_1$. It is easy to see that $D$ is dense in $\Por_0$. Now (possibly outside $M$), $\Por_0\lessdot_M\Qor_0$ implies that $D$ is predense in $\Qor_0$, so there is a $p_0\in D$ compatible with $q_0$ in $\Qor_0$. Choose $q'_0\leq q_0,p_0$ in $\Qor_0$, which is compatible with $q_1$. This implies that there is a $q'_1\in\Qor_1$ which is stronger than $p_0$, $p_1$ and $p'_1$. As $\Por_1\lessdot_M\Qor_1$, $p_0$, $p_1$ and $p'_1$ have a common lower bound in $\Por_1$ (if not, $\{p\in\Por_1\ /\ \textrm{either }p\perp p_0\textrm{\ or }p\perp p_1\textrm{\ or }p\perp p'_1\}\in M$ would be dense in $\Por_1$, which contradicts $q'_1\leq_{\Qor_1} p_0,p_1,p'_1$), that is, $p_0$ is compatible with some condition in $\Por_1$ stronger than both $p_1$ and $p'_1$. Thus, $p_0\in D$ implies that there is a $p\leq p_1,p'_1$ such that $p_0$ is one of its reductions. As $q'_0\leq p_0$, it forces $p\in\Por_1/\Por_0$ and $p\leq p_1,p'_1$.

   Let $\dot{A}\in M$ be a $\Por_0$-name for a maximal antichain in $\Por_1/\Por_0$. Given $q_0\in\Qor_0$ and $q_1\in\Qor_1$ such that $q_0\Vdash_{\Qor_0}q_1\in\Qor_1/\Qor_0$, we need to find $q''_0\leq q_0$ in $\Qor_0$ and $p_1\in\Por_1$ such that $q''_0$ forces that $p_1\in\dot{A}$ and that $p_1$ is compatible with $q_1$ in $\Qor_1/\Qor_0$. First note that $D'=\{p\in\Por_1\ /\exists_{p_1\in\Por_1}\exists_{p_0\in\Por_0}(p\leq p_0,p_1\textrm{\ and }p_0\Vdash p_1\in\dot{A})\}\in M$ is dense in $\Por_1$, so it is predense in $\Qor_1$. Choose $q'_1\leq q_0,q_1$ in $\Qor_1$ and find $p'_1\in D$ which is compatible with $q'_1$, so let $q''_1\in\Qor_1$ be a common stronger condition. As $p'_1\in D$, there are $p_1\in\Por_1$ and $p_0\in\Por_0$ such that $p_0\Vdash p_1\in\dot{A}$ and $p'_1\leq p_0,p_1$. Then, there exists a reduction $q''_0\leq p_0,q_0$ in $\Qor_0$ of $q''_1$. $q''_0$ and $p_1$ are as desired.
\end{proof}

\begin{lemma}\label{CorrQuotEmb}
   Let $\langle\Por,\Qor,\Por',\Qor'\rangle$ and $\langle\Qor,\Ror,\Qor',\Ror'\rangle$ be correct diagrams with respect to $M$ where $\Por,\Qor,\Ror\in M$. Then,
   \begin{enumerate}[(a)]
    \item $\langle\Por,\Ror,\Por',\Ror'\rangle$ is a correct diagram with respect to $M$ and
    \item $\Por'$ forces that the diagram $\langle\Qor/\Por,\Ror/\Por,\Qor'/\Por',\Ror'/\Por'\rangle$ is correct with respect to $M^\Por$.
   \end{enumerate}
\end{lemma}
\begin{proof}
 \begin{enumerate}[(a)]
  \item Let $p_0\in\Por$ be a reduction of $r_0\in\Ror$. The set
   \[D=\{p\in\Por\ /\ \textrm{$\exists_{r\leq_{\Ror} r_0}\exists_{q\in\Qor}(p$ is a reduction of $q$ and
       $q$ is a reduction of $r$})\}\]
   is in $M$ and it is dense below $p_0$ in $\Por$, so it is predense below $p_0$ in $\Por'$. To see that $p_0$ is a reduction of $r_0$ with respect to $\Por',\Ror'$, if $p'\in\Por'$ is stronger than $p_0$, then it is compatible with some $p_1\in D$, so there are $r_1\leq_\Ror r_0$ and $q_1\in\Qor$ a reduction of $r_1$ such that $p_1$ is a reduction of $q_1$. By correctness, $p_1$ is a reduction of $q_1$ with respect to $\Por',\Qor'$ and $q_1$ is a reduction of $r_1$ with respect to $\Qor',\Ror'$. It follows directly that $p'$ is compatible with $r_0$ in $\Ror'$.
  \item By Lemma \ref{QuotEmb}, $\Por$ forces $\Qor/\Por\lessdot\Ror/\Por$ (because $\langle\Por,\Qor,\Por,\Ror\rangle$ is correct) and $\Por'$ forces $\Qor'/\Por'\lessdot\Ror'/\Por'$, $\Qor/\Por\lessdot_{M^\Por}\Qor'/\Por'$ and $\Ror/\Por\lessdot_{M^\Por}\Ror'/\Por'$. In any $\Por'$-extension, we know that $\Ror/\Por\simeq(\Qor/\Por)\ast(\Ror/\Qor)$ and $\Ror'/\Por'\simeq(\Qor'/\Por')\ast(\Ror'/\Qor')$ by Lemma \ref{2stepofQuot}. As $\Qor/\Por\lessdot_{M^\Por}\Qor'/\Por'$ and $\Qor'/\Por'$ forces that $\Ror/\Qor\lessdot_{M^\Qor}\Ror'/\Qor'$ by Lemma \ref{QuotEmb}, we get the correctness we are looking for from Lemma \ref{2stepitemb}.
 \end{enumerate}
\end{proof}

A partial order $\langle I,\leq\rangle$ is \emph{directed} iff any two elements of $I$ have an upper bound in $I$. A sequence of posets $\langle\Por_i\rangle_{i\in I}$ is a \emph{directed system of posets} if, for any $i\leq j$ in $I$, $\Por_i\lessdot\Por_j$. Here, the \emph{direct limit of $\langle\Por_i\rangle_{i\in I}$} is defined as the partial order $\limdir_{i\in I}\Por_i:=\bigcup_{i\in I}\Por_i$. It is clear that, for any $i\in I$, $\Por_i\lessdot\limdir_{i\in I}\Por_i$. Throughout this text, ``$\Por$ is a direct limit" means that it is a direct limit of a directed system of posets.

\begin{lemma}[{\cite[Lemma 1.2]{brendle2}}]\label{dirlimEmb}
   Let $I\in M$ be a directed set, $\langle\Por_i\rangle_{i\in I}\in M$ and $\langle\Qor_i\rangle_{i\in I}$ directed systems of posets such that
   \begin{enumerate}[(i)]
      \item for each $i\in I$, $\Por_i\lessdot_M\Qor_i$ and
      \item whenever $i\leq j$, $\langle\Por_i,\Por_j,\Qor_i,\Qor_j\rangle$ is a correct diagram with respect to $M$
   \end{enumerate}
   Then, $\Por:=\limdir_{i\in I}\Por_i$ is a regular subposet of $\Qor:=\limdir_{i\in I}\Qor_i$ with respect to $M$ and, for any $i\in I$, $\langle\Por_i,\Por,\Qor_i,\Qor\rangle$ is correct with respect to $M$.
\end{lemma}
\begin{proof}
   Let $A\in M$ be a maximal antichain of $\Por$. Let $q\in\Qor$, so there is some $i\in I$ such that $q\in\Qor_i$. Work within $M$. Enumerate $A:=\{p_\alpha\ /\ \alpha<\delta\}$ for some ordinal $\delta$ and, for each $\alpha<\delta$, choose $j_\alpha\geq i$ in $I$ such that $p_\alpha\in\Por_{j_\alpha}$. Now, if $p\in\Por_i$, there is some $\alpha<\delta$ such that $p$ is compatible with $p_\alpha$ in $\Por_{j_\alpha}$, so there exists $p'\leq p$ which is a reduction of $p_\alpha$ with respect to $\Por_i,\Por_{j_\alpha}$.

   The previous density argument implies that $q$ is compatible with some $p\in\Por_i$ which is a reduction of $p_\alpha$ for some $\alpha<\delta$. By (ii), $p$ is a reduction of $p_\alpha$ with respect to $\Qor_i,\Qor_{j_\alpha}$, which implies that $q$ is compatible with $p_\alpha$.

   Correctness follows straightforward.
\end{proof}

\begin{lemma}\label{dirlimquot}
   Let $\langle\Por_i\rangle_{i\in I}$ be a directed system of posets, $\Por$ its direct limit. Assume that $\Qor\lessdot\Por_i$ for all $i\in I$. Then, $\Qor$ forces that $\Por/\Qor=\limdir_{i\in I}\Por_i/\Qor$.
\end{lemma}
\begin{proof}
   For $i\in I$, as $\langle\Qor,\Por_i,\Qor,\Por\rangle$ is a correct diagram, by Lemma \ref{QuotEmb} $\Qor$ forces $\Por_i/\Qor\lessdot\Por/\Qor$. It is easy to see that $\Qor$ forces $\Por/\Qor=\bigcup_{i\in I}\Por_i/\Qor$.
\end{proof}

\begin{definition}[Brendle {\cite{brendle2}}]\label{DefSuslincorr}
 A Suslin ccc poset $\Sor$ is \emph{correctness preserving} if, given a correct diagram $\langle\Por_0,\Por_1,\Qor_0,\Qor_1\rangle$, the diagram $\langle\Por_0\ast\Snm^{V^{\Por_0}},\Por_1\ast\Snm^{V^{\Por_1}},\Qor_0\ast\Snm^{V^{\Qor_0}},
 \Qor_1\ast\Snm^{V^{\Qor_1}}\rangle$ is also correct.
\end{definition}

According to the rules of construction of template iterations, the Suslin ccc posets that are correctness preserving are the definable posets that can be used to perform such an iteration. We prove that the examples of Suslin ccc posets presented after Definition \ref{DefSuslinLinked} are correctness preserving. Nevertheless, it is not known an example of a Suslin ccc poset that is not correctness preserving.

\begin{conjecture}[Brendle]\label{ConjecSuslinCorr}
   Every Suslin ccc poset is correctness preserving.
\end{conjecture}

First, we consider the following facts about preserving correctness.

\begin{lemma}\label{lemmacorrpres}
  Let $\langle\Por_0,\Por_1,\Qor_0,\Qor_1\rangle$ be a correct diagram of posets with respect to $M$ (where $\Por_0,\Por_1\in M$).
  \begin{enumerate}[(a)]
   \item If $\Pnm_2\in M$ is a $\Por_1$-name for a poset, $\Qnm_2$ is a $\Qor_1$-name for a poset and $\Qor_1$ forces $\Pnm_2\lessdot_{M^{\Por_1}}\Qnm_2$, then $\langle\Por_0,\Por_1\ast\Pnm_2,\Qor_0,\Qor_1\ast\Qnm_2\rangle$ is correct with respect to $M$.
   \item Let $\Rnm\in M$ be a $\Por_0$-name of a poset\footnote{Unlike definable posets, this $\Rnm$ here is intended to have the same interpretation in any transitive model of $\thzfc$ that contains $M^{\Por_0}$}. Then, $\langle\Por_0\ast\Rnm,\Por_1\ast\Rnm,\Qor_0\ast\Rnm,\Qor_1\ast\Rnm\rangle$ is correct with respect to $M$.
  \end{enumerate}
\end{lemma}
\begin{proof}
  (a) follows directly from Lemma \ref{2stepitemb} and \ref{CorrQuotEmb}(a).

  We prove (b). Let $(p_0,\dot{r}_0)\in\Por_0\ast\Rnm$ be a reduction of $(p_1,\dot{r}_1)\in\Por_1\ast\Rnm$. Work within $M$. Define $D\subseteq\Por_0\ast\Rnm$ such that $(p,\dot{r})\in D$ iff it is a reduction of $(p'_1,\dot{r})\leq(p_1,\dot{r}_1)$ for some $p'_1\in\Por_1$. We show that $D$ is dense below $(p_0,\dot{r}_0)$. Indeed, let $(p',\dot{r}')\in\Por_0\ast\Rnm$ stronger than $(p_0,\dot{r}_0)$. Thus, $(p',\dot{r}')$ is compatible in $\Por_1\ast\Rnm$ with $(p_1,\dot{r}_1)$, so there is a common stronger condition $(p'_1,\dot{r})$. Without loss of generality, as $\Rnm$ is a $\Por_0$-name, we may assume that $\dot{r}$ is a $\Por_0$-name of a member of $\Rnm$. Also, $p'_1\leq p'$ implies that there is a reduction $p\in\Por_0$ of $p'_1$ which is stronger than $p'$. It is clear that $(p,\dot{r})$ is a reduction of $(p'_1,\dot{r})$.

  It remains to prove (possibly outside $M$) that $(p_0,\dot{r}_0)$ is a reduction of $(p_1,\dot{r}_1)$ with respect to $\Qor_0\ast\Rnm,\Qor_1\ast\Rnm$. Let $(q_0,\dot{s}_0)\in\Qor_0\ast\Rnm$ stronger than $(p_0,\dot{r}_0)$. As $D\in M$ is dense below $(p_0,\dot{r}_0)$ in $\Por_0\ast\Rnm$ and this poset is a regular subposet of $\Qor_0\ast\Rnm$ with respect to $M$, $D$ is predense below $(p_0,\dot{r}_0)$ in $\Qor_0\ast\Rnm$, so there exists a $(p'_0,\dot{r}'_0)\in D$ compatible with $(q_0,\dot{s}_0)$ in $\Qor_0\ast\Rnm$. Choose a common stronger condition $(q'_0,\dot{s}'_0)\in\Qor_0\ast\Rnm$ and also choose $p'_1\in\Por_1$ such that $(p'_0,\dot{r}'_0)$ is a reduction of $(p'_1,\dot{r}'_0)$ and the latter is stronger than $(p_1,\dot{r}_1)$. The correctness of $\langle\Por_0,\Por_1,\Qor_0,\Qor_1\rangle$ imply that $p'_0$ is a reduction of $p'_1$ with respect to $\Qor_0,\Qor_1$, so $q'_0$ is compatible with $p'_1$ in $\Qor_1$. Therefore, any common stronger condition $q_1\in\Qor_1$ forces $\dot{s}'_0\leq_{\Rnm}\dot{s}_0,\dot{r}'_0$, so $(q_0,\dot{s}_0)$ and $(p'_1,\dot{r}'_0)$ are compatible in $\Qor_1\ast\Rnm$.
\end{proof}

\begin{lemma}
  $\Cor$ is correctness preserving.
\end{lemma}
\begin{proof}
  Immediate from Lemma \ref{lemmacorrpres}(b) because the interpretation of $\Cor=2^{<\omega}$, ordered by end extension, is the same in any generic extension (in other words, $\Por\ast\Cnm\simeq\Por\times\Cor$ for any poset $\Por$).
\end{proof}

\begin{lemma}[Brendle {\cite{Br-shat}}]\label{randomCorrPres}
   $\Bor$ is correctness-preserving.
\end{lemma}

The author is very grateful to Brendle for letting include his proof in this paper, whose cited reference is not available at the moment (but, hopefully, it will be soon). We first prove a lemma that is implicit in the original proof.

\begin{lemma}\label{closenm}
   Let $\Por$ be a regular subposet of the poset $\Por'$. Let $\Qnm'$ be a $\Por'$-name of a poset and $\Qnm$ a $\Por$-name of a poset such that $\Vdash_{\Por'}\Qnm\lessdot_{V^\Por}\Qnm'$. If $(p,\dot{q})\in\Por\ast\Qnm$, $(p',\dot{q}')\in\Por'\ast\Qnm'$, $p$ is a reduction of $p'$ and $p'\Vdash\dot{q}'\parallel\dot{q}$, then there exists a $\Por$-name $\dot{q}_0$ of a condition in $\Qnm$ such that $p$ forces $\dot{q}_0\leq\dot{q}$ and that, for any $r\in\Qnm$ stronger than $\dot{q}_0$, there exists a $p_1\leq p'$ in $\Por'/\Por$ such that $p_1\Vdash_{\Por'/\Por}r\parallel \dot{q}'$.
\end{lemma}
\begin{proof}
   Let $G$ be $\Por$-generic over $V$ with $p\in G$. Work in $V[G]$. Let $D:=\{r\in\Qor\ /\ r\leq q\textrm{\ and\ } p'\Vdash_{\Por'/\Por}r\perp\dot{q}'\}$. Note that $D$ is not predense below $q$ (if so, $p'$ would force, with respect to $\Por'/\Por$, that $D$ is predense below $q$ in $\dot{\Qor'}$ and $\dot{q}'\perp\dot{q}$, which is impossible). Therefore, there exists a $q_0\leq q$ in $\Qor$ that is incompatible with all the members of $D$. Thus, any $r\leq q_0$ is not in $D$, so there is a $p_1\leq p'$ in $\Por'/\Por$ such that $p_1\Vdash_{\Por'/\Por}r\parallel \dot{q}'$.
\end{proof}

\begin{proof}[Proof of Lemma \ref{randomCorrPres}]
   Let $\langle\Por_\wedge,\Por_0,\Por_1,\Por_\vee\rangle$ be a correct diagram of posets. For $i\in I_4=\{\wedge,0,1,\vee\}$, let $\Bnm_i$ be a $\Por_i$-name for random forcing. To prove that $\langle\Por_\wedge\ast\Bnm_\wedge,\Por_0\ast\Bnm_0,\Por_1\ast\Bnm_1,\Por_\vee\ast
   \Bnm_\vee\rangle$ is correct, it is enough to show that, if $(p_0,\dot{b}_0)\in\Por_0\ast\Bnm_0$ and $(p_1,\dot{b}_1)\in\Por_1\ast\Bnm_1$ have a common reduction in $\Por_\wedge\ast\Bnm_\wedge$, then there are $(q_0,\dot{c}_0)\leq(p_0,\dot{b}_0)$ and $(q_1,\dot{c}_1)\leq(p_1,\dot{b}_1)$ in $\Por_0\ast\Bnm_0$ and $\Por_1\ast\Bnm_1$, respectively, such that $q_i\Vdash\lambda(\dot{c}_i\cap\dot{c}_\wedge)>\frac{3}{4}\lambda(\dot{c}_\wedge)$ for $i=0,1$ where $\lambda$ is the Lebesgue measure for $2^\omega$, $\dot{c}_\wedge$ is some $\Por_\wedge$-name for a condition in $\Bnm_\wedge$ and $q_0,q_1$ have a common reduction in $\Por_\wedge$. This is so because $q_0,q_1$ will be compatible in $\Por_\vee$ by correctness and $\dot{c}_\vee=\dot{c}_0\cap\dot{c}_1\cap\dot{c}_\wedge$ is forced to be, by any stronger condition than $q_0$ and $q_1$, a condition in $\Bnm_\vee$ of measure $>\frac{1}{2}\lambda(\dot{c}_\wedge)$.

   Let $(p,\dot{b})$ be a common reduction in $\Por_\wedge\ast\Bnm_\wedge$ of $(p_0,\dot{b}_0)$ and $(p_1,\dot{b}_1)$. Choose $(p'_1,\dot{b}'_1)\in\Por_1\ast\Bnm_1$ a condition stronger than $(p,\dot{b})$ and $(p_1,\dot{b}_1)$, so $p'_1$ forces that $\dot{b}'_1\subseteq\dot{b}\cap\dot{b}_1$. Choose $p'\leq p$ (in $\Por_\wedge$) a reduction of $p'_1$.
   \begin{claim}\label{closename}
      There is a $\Por_\wedge$-name $\dot{b}_\wedge$ of a condition in $\Bnm_\wedge$ such that $p'$ forces that $\dot{b}_\wedge\subseteq\dot{b}$ and that, for any $c\in\Bnm_\wedge$ stronger than $\dot{b}_\wedge$, there is a condition $q\leq p'_1$ in $\Por_1/\Por_\wedge$ such that $q\Vdash_{\Por_1/\Por_\wedge}\lambda(\dot{b}'_1\cap c)>0$.
   \end{claim}
   \begin{proof}
      Direct consequence of Lemma \ref{closenm}.
   \end{proof}
   As $(p',\dot{b}_\wedge)$ is a reduction of $(p_0,\dot{b}_0)$, there is a common stronger condition $(p'_0,\dot{c}_0)\in\Por_0\ast\Bnm_0$. Then, $p'_0\Vdash\dot{c}_0\subseteq\dot{b}_\wedge$.
   By the Lebesgue density Theorem, find $s\in 2^{<\omega}$ and $p''_0\leq_{\Por_0}p'_0$ that forces $\lambda(\dot{c}_0\cap[s])>\frac{3}{4}\lambda([s])$. Put $\dot{b}''_\wedge=\dot{b}_\wedge\cap[s]$, so $p_0''\Vdash\lambda(\dot{c}_0\cap\dot{b}''_\wedge)>\frac{3}{4}\lambda(\dot{b}''_\wedge)$.

   Let $p''\leq p'$ be a reduction of $p''_0$ in $\Por_\wedge$.
   \begin{claim}\label{lebantich}
      There are $\Por_\wedge$-names $\{\dot{c}^n\}_{n<\omega}$ for conditions in $\Bnm_\wedge$ and $\{\dot{p}^n_1\}_{n>\omega}$ of conditions in $\Por_1/\Por_\wedge$ such that $p''$ forces that $\{\dot{c}^n\}_{n<\omega}$ is a maximal antichain below $\dot{b}''_\wedge$ and, for each $n<\omega$, $\dot{p}^n_1\leq p'_1$ forces, with respect to $\Por_1/\Por_\wedge$, that $\lambda(\dot{b}'_1\cap\dot{c}^n)>\frac{3}{4}\lambda(\dot{c}^n)$.
   \end{claim}
   \begin{proof}
      Let $G$ be a $\Por_\wedge$-generic set over $V$ with $p''\in G$. Work in $V[G]$. Let $c\subseteq b''_\wedge$ in $\Bor$ arbitrary. By Claim \ref{closename}, there is a $q\leq p'_1$ that forces, with respect to $\Por_1/\Por_\wedge$, $\lambda(\dot{b}'_1\cap c)>0$. As in the paragraph preceding the present claim, use the Lebesgue density Theorem to get $c'\subseteq c$ and $q'\leq q$ that forces $\lambda(\dot{b}'_1\cap c')>\frac{3}{4}\lambda(c')$. This density argument implies that there exists $\{c^n\}_{n<\omega}$ maximal antichain below $b''_\wedge$ such that, for any $n<\omega$, there exists a $p_1^n\leq p'_1$ that forces $\lambda(\dot{b}'_1\cap c^n)>\frac{3}{4}\lambda(c^n)$.
   \end{proof}
   Note that there are $n<\omega$ and $q_0\in\Por_0$ a condition stronger than $p''$ and $p''_0$ such that $q_0\Vdash\lambda(\dot{c}_0\cap\dot{c}^n)>\frac{3}{4}\lambda(\dot{c}^n)$. If this were not the case, then any condition stronger than $p''$ and $p''_0$ in $\Por_0$ would force $\lambda(\dot{c}_0\cap\dot{c}^n)\leq\frac{3}{4}\lambda(\dot{c}^n)$ for all $n<\omega$, but this implies that $\lambda(\dot{c}_0\cap\dot{b}''_\wedge)\leq\frac{3}{4}\lambda(\dot{b}''_\wedge)$, which is false because $p''_0$ forces the contrary. Put $\dot{c}_\wedge=\dot{c}^n$.

   Let $q\leq p''$ be a reduction of $q_0$ in $\Por_\wedge$. As $q$ forces that $\dot{p}^n_1\leq p'_1$ in $\Por_1/\Por_\wedge$, there exists a $q'_1\leq p'_1$ in $\Por_1$ and $q_\wedge\leq q$ in $\Por_\wedge$ such that $q_\wedge\Vdash q'_1=\dot{p}^n_1\in\Por_1/\Por_\wedge$, so $q_\wedge$ is a reduction of $q'_1$. Let $q_1\in\Por_1$ be a condition stronger than $q_\wedge$ and $q'_1$, so Claim \ref{lebantich} implies that any reduction of $q_1$ in $\Por_\wedge$ forces that $q_1\Vdash_{\Por_1/\Por_\wedge}\lambda(\dot{b}'_1\cap\dot{c}_\wedge)>
   \frac{3}{4}\lambda(\dot{c}_\wedge)$. Therefore, by Lemma \ref{QuotEqv}, $q_1\Vdash_{\Por_1}\lambda(\dot{b}'_1\cap\dot{c}_\wedge)>\frac{3}{4}\lambda(\dot{c}_\wedge)$. Put $\dot{c}_1=\dot{b}'_1$. Note also that any reduction of $q_1$ in $\Por_\wedge$ is also a reduction of $q_0$, so the proof is complete.
\end{proof}

\begin{lemma}[Brendle {\cite[Lemma 1.3]{brendle2}}]\label{HechlerCorrPres}
   $\Dor$ is correctness preserving.
\end{lemma}
\begin{proof}
   Let $\langle\Por_\wedge,\Por_0,\Por_1,\Por_\vee\rangle$ be a correct diagram. Assume $(p_0,(s_0,\dot{f}_0))\in\Por_0\ast\Dnm_0$ and $(p_1,(s_1,\dot{f}_1))\in\Por_1\ast\Dnm_1$ with a common reduction in $\Por_\wedge\ast\Dnm_\wedge$. Show that they are compatible in $\Por_\vee\ast\Dnm_\vee$. Consider $(p,(s,\dot{f}))\in\Por_\wedge\ast\Dnm_\wedge$ such a common reduction with $|s|\geq|s_1|$. It is clear that $s_1\subseteq s$.
   \begin{claim}
      For any $t\supseteq s$ in $\omega^{<\omega}$, if $q\leq p$ in $\Por_\wedge$ forces that $\dot{f}\frestr|t|\leq t$, then there is a $p'_1\leq p_1,q$ in $\Por_1$ that forces $\dot{f}_1\frestr|t|\leq t$.
   \end{claim}
   As $(p,(s,\dot{f}))$ is a reduction of $(p_0,(s_0,\dot{f}_0))\in\Por_0\ast\Dnm_0$, let $(p'_0,(s'_0,\dot{f}'_0))\in\Por_0\ast\Dnm_0$ be a condition stronger than both, so $s\subseteq s'_0$. Let
   $(p',(t,\dot{f}'))\leq(p,(s,\dot{f}))$ be a reduction of $(p'_0,(s'_0,\dot{f}'_0))$ in $\Por_\wedge\ast\Dor_\wedge$ with $|s'_0|\leq|t|$, so $s'_0\subseteq t$. By the claim, there is a $p'_1\leq p_1,p'$ in $\Por_1$ that forces $\dot{f}_1\frestr|t|\leq t$. As any reduction of $p'_1$ in $\Por_\wedge$ is a reduction of $p'_0$, by correctness we get that $p'_0,p'_1$ are compatible in $\Por_\vee$. Note that any stronger condition than $p'_0$ and $p'_1$ in $\Por_\vee$ forces that $(s'_0,\dot{f}'_0)$ and $(s_1,\dot{f}_1)$ are compatible.
\end{proof}

\begin{lemma}\label{EvDiffCorrPres}
   $\Eor$ is correctness preserving.
\end{lemma}
\begin{proof}
   Imitate the proof of Lemma \ref{HechlerCorrPres} and just replace $\dot{f}\frestr|t|\leq t$ by $\forall_{i\in[|s|,|t|)}\forall_{x\in\dot{F}}(x(i)\neq t(i))$ and $\dot{f}_1\frestr|t|\leq t$ by $\forall_{i\in[|s_1|,|t|)}\forall_{x\in\dot{F}_1}(x(i)\neq t(i))$
\end{proof}

\begin{lemma}\label{LocCorrPres}
   For $h\in\omega^\omega$ non-decreasing that converges to infinite, $\Loc^h$ is correctness preserving.
\end{lemma}
\begin{proof}
   Same idea of the proof of Lemma \ref{HechlerCorrPres}.
\end{proof}

\subsection{Forcing with ultraproducts and ultrapowers}\label{SubsecUltraP}

We present some facts, introduced by Shelah \cite{shelah}  (see also \cite{br} and \cite{brendle3}) about forcing with the ultrapower of a ccc poset by a measurable cardinal. Fix a measurable cardinal $\kappa$ and a non-principal $\kappa$-complete ultrafilter $\Dwf$ on $\kappa$.

Say that a property $\varphi(\alpha)$ \emph{holds for $\Dwf$-many $\alpha$} iff $\left\{\alpha<\kappa\ /\ \varphi(\alpha)\right\}\in\Dwf$. If $\langle X_\alpha\rangle_{\alpha<\kappa}$ is a sequence of sets, $(\prod_{\alpha<\kappa}X_\alpha)/\Dwf=[\{X_\alpha\}_{\alpha<\kappa}]$ denotes the \emph{ultraproduct of $\langle X_\alpha\rangle_{\alpha<\kappa}$ modulo $\Dwf$}, which is the quotient of $\prod_{\alpha<\kappa}X_\alpha$ modulo the equivalence relation given by $x\sim_\Dwf y$ iff $x_\alpha=y_\alpha$ for $\Dwf$-many $\alpha$. If $x=\langle x_\alpha\rangle_{\alpha<\kappa}\in\prod_{\alpha<\kappa}X_\alpha$, denote its equivalence class under $\sim_\Dwf$ by $\bar{x}=\langle x_\alpha\rangle_{\alpha<\omega}/\Dwf$. An ultraproduct of the form $\prod_{\alpha<\kappa}X/\Dwf=X^\kappa/\Dwf$ is often known as an \emph{ultrapower}.

Fix a sequence $\langle\Por_\alpha\rangle_{\alpha<\kappa}$ of posets. For notation, if $p\in\bar{\Por}:=\prod_{\alpha<\kappa}\Por_\alpha$, denote $p_\alpha=p(\alpha)$. For $p,q\in\bar{\Por}$, say that $p\leq_\Dwf q$ iff $p_\alpha\leq q_\alpha$ for $\Dwf$-many $\alpha$. The poset $\bar{\Por}/\Dwf$, ordered by $\bar{p}\leq\bar{q}$ iff $p\leq_\Dwf q$, is the \emph{$\Dwf$-ultraproduct of $\langle\Por_\alpha\rangle_{\alpha<\kappa}$}. We are particularly interested in the \emph{$\Dwf$-ultrapower} $\Por^\kappa/\Dwf$ of a poset $\Por$.

\begin{lemma}[Shelah {\cite{shelah}}, see also {\cite[Lemma 0.1]{br}}]\label{UltraprodEmb}
   Consider $i:\Por\to\Por^\kappa/\Dwf$ defined by $i(r)=\bar{r}$ where $r_\alpha=r$ for all $\alpha<\kappa$. Then, $i$ is a regular embedding iff $\Por$ is $\kappa$-cc.
\end{lemma}

Throughout the text, when dealing with ultrapowers of posets, we identify $i(r)$ with $r$, so we can think of $\Por$ as a regular subposet of $\Por^\kappa/\Dwf$ when $\Por$ is $\kappa$-cc.

\begin{lemma}[Shelah {\cite{shelah}}, see also {\cite[Lemma 0.2]{br}}]\label{Ultraprodccc}
   If $\mu<\kappa$ and $\Por_\alpha$ is a $\mu$-cc poset for all $\alpha<\kappa$, then $\prod_{\alpha<\kappa}\Por_\alpha/\Dwf$ is also $\mu$-cc. The same holds for $\mu$-Knaster, $\mu$-centered and $\mu$-linked in place of $\mu$-cc.
\end{lemma}

The previous lemma was proved only for $\mu$-cc and ultrapowers in the mentioned references, but the cases for ultraproducts follow easily. The details are left to the reader.

Fix, until the end of this section, a sequence $\langle\Por_\alpha\rangle_{\alpha<\kappa}$ of ccc posets and put $\Qor:=\prod_{\alpha<\kappa}\Por_\alpha/\Dwf$, which is ccc. We analyze how $\Qor$-names for reals look like. For reference, consider $\omega^\omega$. First, we show how to construct a $\Qor$-name of a real from a sequence $\langle\dot{f}_\alpha\rangle_{\alpha<\kappa}$ where each $\dot{f}_\alpha$ is a $\Por_\alpha$-name of a real. For each $\alpha<\omega$ and $n<\omega$, let $\{p_\alpha^{n,j}\ /\ j<\omega\}$ be a maximal antichain in $\Por_\alpha$ and $k_\alpha^n:\omega\to\omega$ a function such that $p_\alpha^{n,j}\Vdash\dot{f}_\alpha(n)=k_\alpha^n(j)$ for all $j<\omega$. Put $p^{n,j}=\langle p^{n,j}_\alpha\rangle_{\alpha<\kappa}$ and note that, for $n<\omega$, $\{\bar{p}^{n,j}\ /\ j<\omega\}$ is a maximal antichain in $\Qor$ by $\kappa$-completeness of $\Dwf$. Also, as $\cfrak<\kappa$, there exists a $D\in\Dwf$ and, for each $n<\omega$, a function $k^n:\omega\to\omega$ such that $k^n_\alpha=k^n$ for all $\alpha\in D$. Define a $\Qor$-name $\dot{f}=\langle\dot{f}_\alpha\rangle_{\alpha<\kappa}/\Dwf$ for a real such that, for any $n,j<\omega$, $\bar{p}^{n,j}\Vdash\dot{f}(n)=k^n(j)$. Note that, if $\langle \dot{g}_\alpha\rangle_{\alpha<\kappa}$ is a sequence where each $\dot{g}_\alpha$ is a $\Por_\alpha$-name for a real and $\Vdash_\Por\dot{f}_\alpha=\dot{g}_\alpha$ for $\Dwf$-many $\alpha$, then $\Vdash_{\Por^\kappa/\Dwf}\dot{f}=\dot{g}$ where $\dot{g}=\langle\dot{g}_\alpha\rangle_{\alpha<\kappa}/\Dwf$.

We show that any $\Qor$-name $\dot{f}$ for a real can be described in this way. For each $n<\omega$, let $\{\bar{p}^{n,j}\ /\ j<\omega\}$ be a maximal antichain in $\Qor$ and $k^n:\omega\to\omega$ such that $\bar{p}^{n,j}\Vdash\dot{f}(n)=k^n(j)$. By $\kappa$-completeness of $\Dwf$, we can find $D\in\Dwf$ such that, for all $\alpha\in D$, $\{p^{n,j}_\alpha\ /\ j<\omega\}$ is a maximal antichain in $\Por_\alpha$ for any $n<\omega$. Let $\dot{f}_\alpha$ be the $\Por_\alpha$-name of a real such that $p^{n,j}_\alpha\Vdash_\Por\dot{f}_\alpha(n)=k^n(j)$. For $\alpha\in\kappa\menos D$ just choose any $\Por_\alpha$-name $\dot{f}_\alpha$ for a real, so we get that $\Vdash_{\Por^\kappa/\Dwf}\dot{f}=\langle\dot{f}_\alpha\rangle_{\alpha<\kappa}/\Dwf$.

\begin{lemma}\label{Sigma1-1andUltrapow}
   Fix $m<\omega$ and a $\boldsymbol{\Sigma}^1_m$ property $\varphi(x)$ of reals. For each $\alpha<\kappa$, let $\dot{f}_{\alpha}$ be a $\Por_\alpha$-name for a real and put $\dot{f}=\langle\dot{f}_{\alpha}\rangle_{\alpha<\kappa}/\Dwf$. Then, for $\bar{p}\in\prod_{\alpha<\kappa}\Por_\alpha/\Dwf$, $\bar{p}\Vdash\varphi(\dot{f})$ iff $p_\alpha\Vdash_{\Por_\alpha}\varphi(\dot{f}_{\alpha})$ for $\Dwf$-many $\alpha$.
\end{lemma}
\begin{proof}
   This is proved by induction on $m<\omega$. Recall that $\boldsymbol{\Sigma}^1_0=\boldsymbol{\Pi}^1_0$ corresponds to the pointclass of closed sets. Thus, if $\varphi(x)$ is a $\boldsymbol{\Sigma}^1_0$-property of reals, there exists a tree $T\subseteq\omega^\omega$ such that, for $x\in\omega^\omega$, $\varphi(x)$ iff $x\in[T]:=\{z\in\omega^\omega\ /\ \forall_{k<\omega}(z\frestr k\in T)\}$.

   As in the previous discussion choose, for each $n<\omega$, a maximal antichain $\{\bar{p}^{n,j}\ /\ j<\omega\}$ in $\prod_{\alpha<\kappa}\Por_\alpha/\Dwf$ and a function $k^n:\omega\to\omega$ such that $\bar{p}^{n,j}\Vdash\dot{f}(n)=k^n(j)$ and $p^{n,j}_\alpha\Vdash\dot{f}_\alpha(n)=k^n(j)$ for $\Dwf$-many $\alpha$. First, assume that $p_\alpha\Vdash f_\alpha\in[T]$ for $\Dwf$-many $\alpha$ and fix $k<\omega$. If $\bar{q}\leq\bar{p}$, we can find a decreasing sequence $\{\bar{q}^i\}_{i\leq k}$ and a $t\in\omega^k$ such that $\bar{q}^0=\bar{q}$ and $\bar{q}^{i+1}\leq\bar{p}^{i,t(i)}$ for any $i<k$. Therefore, $\bar{q}^k\Vdash\dot{f}\frestr k=k^n\circ t$ and, for $\Dwf$-many $\alpha$, $q^k_\alpha\Vdash\dot{f}_\alpha\frestr k=k^n\circ t$, so $k^n\circ t\in T$.

   Now, assume that $p_\alpha\not\Vdash f_\alpha\in[T]$ for $\Dwf$-many $\alpha$. Without loss of generality, we may assume that there is a $k<\omega$ such that $p_\alpha\Vdash\dot{f}_\alpha\frestr k\notin T$ for $\Dwf$-many $\alpha$. To prove $\bar{p}\Vdash\dot{f}\frestr k\notin T$ repeat the same argument as before, but note that this time we get $k^n\circ t\notin T$.

   For the induction step, assume that $\varphi(x)$ is $\boldsymbol{\Sigma}^1_{m+1}$, so $\varphi(x)\sii\exists_{y\in\omega^\omega}\psi(x,y)$ where $\psi(x,y)$ is some $\boldsymbol{\Pi}^1_m(\omega^\omega\times\omega^\omega)$-statement (notice that, if this theorem is valid for all $\boldsymbol{\Sigma}^1_m$-statements, then it is also valid for all $\boldsymbol{\Pi}^1_m$-statements). Assume that $p_\alpha\Vdash\exists_{z\in\omega^\omega}\psi(\dot{f}_\alpha,z)$ for $\Dwf$-many $\alpha$ and, for those $\alpha$, choose a $\Por_\alpha$-name $\dot{g}_\alpha$ such that $p_\alpha\Vdash\psi(\dot{f}_\alpha,\dot{g}_\alpha)$. By induction hypothesis, $\bar{p}\Vdash\psi(\dot{f},\dot{g})$ where $\dot{g}=\langle\dot{g}_\alpha\rangle_{\alpha<\kappa}/\Dwf$. The converse also follows easily.
\end{proof}

\begin{corollary}[Shelah {\cite{shelah}}, see also {\cite[Lemma 0.3]{br}}]\label{DestrMad}
   Let $\Por$ be a ccc poset and $\dot{\Awf}$ a $\Por$-name of an a.d. family such that $\Vdash_\Por|\dot{\Awf}|\geq\kappa$. Then, $\Vdash_{\Por^\kappa/\Dwf}\dot{\Awf}\textrm{ is not maximal}$.
\end{corollary}
\begin{proof}
   Let $r\in\Por$ and $\lambda\geq\kappa$ be a cardinal such that $r\Vdash_\Por\dot{\Awf}=\{\dot{A}_\xi\ /\ \xi<\lambda\}$. By Lemma \ref{Sigma1-1andUltrapow}, $\dot{A}=\langle\dot{A}_\alpha\rangle_{\alpha<\kappa}/\Dwf$ (this can be defined in a similar way by associating the characteristic function to each set) is a $\Por^\kappa/\Dwf$-name of an infinite subset of $\omega$ and $r\Vdash_{\Por^\kappa/\Dwf}|\dot{A}_\xi\cap\dot{A}|<\aleph_0$ for all $\xi<\lambda$.
\end{proof}


\section{Templates}\label{SecTemp}

We introduce Shelah's notion of a template (in a simpler way than in the original work \cite{shelah}), which is the index set of a forcing iteration as defined in Section \ref{SecItTemp}. Except for Lemmas \ref{RestUltrapowTemp} and \ref{RestChainTemp}, all definitions and results are, in essence, due to Shelah \cite{shelah}, but for proofs we refer to \cite{br}.

For a linear order $L:=\langle L,\leq_L\rangle$ and $x\in L$, denote $L_x:=\left\{z\in L\ /\ z<x\right\}$.

\begin{definition}[Indexed template]\label{DefIndxTemp}
  An \emph{indexed template} (or just a \emph{template}) is a pair $\langle L,\bar{\Iwf}:=\langle\Iwf_x\rangle_{x\in L}\rangle$ where $L$ is a linear order, $\Iwf_x\subseteq\Pwf(L_x)$ for all $x\in L$ and the following properties are satisfied.
  \begin{enumerate}[(1)]
     \item $\varnothing\in\Iwf_x$.
     \item $\Iwf_x$ is closed under finite unions and intersections.
     \item If $z<x$ then there is some $A\in\Iwf_x$ such that $z\in A$.
     \item $\Iwf_x\subseteq\Iwf_y$ if $x<y$.
     \item $\Iwf(L):=\bigcup_{x\in L}\Iwf_x\cup\{L\}$ is well-founded with the subset relation.
  \end{enumerate}
\end{definition}

An iteration along $L$ as in Section \ref{SecItTemp} can be constructed thanks to the well-foundedness of $\Iwf(L)$. Note that properties (2) and (4) imply that $\Iwf(L)$ is closed under finite unions and intersections.

If $A\subseteq L$ and $x\in L$, define $\Iwf_x\frestr A:=\left\{A\cap X\ /\ X\in\Iwf_x\right\}$ the \emph{trace of $\Iwf_x$ on $A$}. Put $\bar{\Iwf}\frestr A:=\langle\Iwf_x\frestr A\rangle_{x\in A}$ and $\Iwf(A):=\bigcup_{x\in A}\Iwf_x\frestr A\cup\{A\}$. Note that $\Iwf(A)\subseteq\Iwf(L)\frestr A=\{A\cap X\ /\ X\in\Iwf(L)\}$ but equality may not hold. For $Z\in\Iwf(L)\frestr A$, let $X_Z:=X_{\bar{\Iwf},Z}$ be a set in $\Iwf(L)$ of minimal rank such that $Z=A\cap X_Z$.

\begin{lemma}\label{tracetemp}
   $\langle A,\bar{\Iwf}\frestr A\rangle$ is an indexed template. Moreover, $\Iwf(L)\frestr A$ is well-founded and $\mathrm{rank}_{\Iwf(L)\upharpoonright A}(A\cap X)\leq\mathrm{rank}_{\Iwf(L)}(X)$ for all $X\in\Iwf(L)$.
\end{lemma}
\begin{proof}
    Put $\Jwf=\Iwf(L)\frestr A$. First note that, for $Z,Z'\in\Jwf$, $Z\subseteq Z'$ iff $X_Z\subseteq X_{Z'}$. Indeed, if $Z\subseteq Z'$, then $A\cap (X_Z\cap X_{Z'})=Z$ and $X_Z\cap X_{Z'}\in\Iwf(L)$. Therefore, by minimality of $X_Z$, $X_Z\cap X_{Z'}=X_Z$.

    From the previous argument, $\Jwf$ is well-founded, moreover, $\mathrm{rank}_\Jwf(Z)\leq\mathrm{rank}_{\Iwf(L)}(X_Z)$ for all $Z\in\Jwf$. If $Z=A\cap X$ for some $X\in\Iwf(L)$, then $\mathrm{rank}_\Jwf(Z)\leq\mathrm{rank}_{\Iwf(L)}(X_Z)\leq\mathrm{rank}_{\Iwf(L)}(X)$ because of the minimality of $X_Z$.
\end{proof}

Also note that, if $X\subseteq A\subseteq L$, then $(\Iwf_x\frestr A)\frestr X=\Iwf_x\frestr X$ for any $x\in L$, $(\bar{\Iwf}\frestr A)\frestr X=\bar{\Iwf}\frestr X$ and $(\Iwf(A))(X)=\Iwf(X)$.

For a template $\langle L,\bar{\Iwf}\rangle$ define $\mathrm{Dp}^{\bar{\Iwf}}:\Pwf(L)\to\mathbf{ON}$ by $\mathrm{Dp}^{\bar{\Iwf}}(X)=\mathrm{rank}_{\Iwf(X)}(X)$. Although this is not a rank function (that is, increasing with respect to $\subsetneq$), recursion on $\alpha=\mathrm{Dp}^{\bar{\Iwf}}(X)$ turns out to be useful to prove statements of the form $\forall_{X\subseteq L}\varphi(X)$ where $\varphi(X)$ is a property related to iterations along $\langle L,\bar{\Iwf}\rangle$. For such proofs, the following lemma is very useful. When the template is understood, we just denote $\mathrm{Dp}:=\mathrm{Dp}^{\bar{\Iwf}}$.

\begin{lemma}\label{UpsilonTemp}
   Fix $A\subseteq L$. $\mathrm{Dp}:=\mathrm{Dp}^{\bar{\Iwf}}$ has the following properties.
   \begin{enumerate}[(a)]
      \item If $Y\in\Iwf(A)$, then $\mathrm{Dp}(Y)\leq\mathrm{rank}_{\Iwf(A)}(Y)$.
      \item If $X\subseteq A$ then $\mathrm{Dp}(X)\leq\mathrm{Dp}(A)$.
      \item Let $x\in A$. If $Y\subsetneq A\cap(L_x\cup\{x\})$ and $Y\cap L_x\in\Iwf_x\frestr A$ then $\mathrm{Dp}(Y)<\mathrm{Dp}(A)$. In particular, $\mathrm{Dp}(X)<\mathrm{Dp}(A)$ for all $X\in\Iwf_x\frestr A$
      \item $\mathrm{Dp}^{\bar{\Iwf}\upharpoonright A}=\mathrm{Dp}\frestr\Pwf(A)$.
   \end{enumerate}
\end{lemma}
\begin{proof}
   For $X\subseteq A$, put $\Jwf_X:=\Iwf(A)\frestr X$.
   \begin{enumerate}[(a)]
      \item $\mathrm{Dp}(Y)\leq\mathrm{rank}_{\Jwf_Y}(Y)=\mathrm{rank}_{\Jwf_Y}(Y\cap Y)\leq\mathrm{rank}_{\Iwf(A)}(Y)$ by Lemma \ref{tracetemp}.
      \item By Lemma \ref{tracetemp}, $\mathrm{Dp}(X)\leq\mathrm{rank}_{\Jwf_X}(A\cap X)\leq\mathrm{rank}_{\Iwf(A)}(A)$.
      \item First assume that $x\in Y$, so there is a $z\in(A\cap L_x)\menos Y$. By Definition $\ref{DefIndxTemp}$(3), there is some $C\in\Iwf_x\frestr A$ that contains $z$. Put $E=Y\cup C$. Clearly, $E_x,Y_x\in\Iwf_x\frestr E$ where $E_x=E\cap L_x$ and $Y_x=Y\cap L_x$. Moreover, $\Iwf_x\frestr Y=\{D\subseteq Y\ /\ D\in\Iwf_x\frestr E\}$, so $\mathrm{rank}_{\Iwf(Y)}(Y_x)=\mathrm{rank}_{\Iwf(E)}(Y_x)$. Therefore, $\mathrm{Dp}(Y)=\mathrm{rank}_{\Iwf(Y)}(Y_x)+1=\mathrm{rank}_{\Iwf(E)}(Y_x)+1
          <\mathrm{rank}_{\Iwf(E)}(E_x)+1=\mathrm{Dp}(E)\leq\mathrm{Dp}(A)$.

          If $x\notin Y$ then $Y\in\Iwf_x\frestr A$. So, by (a), $\mathrm{Dp}(Y)\leq\mathrm{rank}_{\Iwf(A)}(Y)<\mathrm{rank}_{\Iwf(A)}(A)$.
      \item For $X\subseteq A$, $\mathrm{Dp}^{\bar{\Iwf}\upharpoonright A}(X)=\mathrm{rank}_{(\Iwf(A))(X)}(X)=\mathrm{rank}_{\Iwf(X)}(X)$.

   \end{enumerate}
\end{proof}

For $x\in L$, define $\hat{\Iwf}_x:=\left\{B\subseteq L_x\ /\ B\in\Iwf_x\frestr(B\cup\{x\})\right\}$. This family is important at the time of the construction of an iteration because the generic object added at stage $x$ is generic over all the intermediate extensions that come from any set in $\hat{\Iwf}_x$ (see the comment after Theorem \ref{TempIt}). Note that $B\in\hat{\Iwf}_x$ if and only if $B\subseteq H$ for some $H\in\Iwf_x$, that is, $\hat{\Iwf}_x$ is the ideal on $L_x$ generated by $\Iwf_x$ (which is equal to $\Pwf(L_x)$ iff $L_x\in\Iwf_x$). Also, (1), (2) and (3) imply that any finite subset of $L_x$ is in $\hat{\Iwf}_x$.

\begin{remark}\label{RemSimpleTemp}
   Lemma \ref{UpsilonTemp} implies that the relation $\lhd_{\bar{\Iwf}}$ is well-founded on $\Pwf(L)$ where $A\lhd_{\bar{\Iwf}} B$ iff $A\subseteq B$ and $A\in\hat{\Iwf}_x$ for some $x\in B$. The anonymous referee of this paper noted that, for the purposes of this text, the definition of template can be simplified by only looking at the ideals $\hat{\Iwf}_x$. Say that $\langle L,\bar{\Jwf}\rangle$ is a \emph{simple indexed template} if $L$ is a linear order and $\bar{\Jwf}=\langle\Jwf_x\rangle_{x\in L}$ where the following conditions are satisfied:
   \begin{enumerate}[(1)]
    \item For $x\in L$, $\Jwf_x\subseteq\Pwf(L_x)$ is an ideal ($L_x\in\Jwf_x$ is allowed) that contains $[L_x]^{<\omega}$.
    \item $x<y$ in $L$ implies $\Jwf_x\subseteq\Jwf_y$.
    \item The relation $\lhd_{\bar{\Jwf}}$ is well-founded on $\Pwf(L)$, where $A\lhd_{\bar{\Jwf}} B$ iff $A\subseteq B$ and $A\in\Jwf_x$ for some $x\in B$.
   \end{enumerate}
   It is clear that any indexed template induces a simple one. Though all the theory of template iterations can be reformulated in terms of simple indexed templates, we present it in the context of indexed templates (Definition \ref{DefIndxTemp}) because we are particularly interested in tracking the generators of the ideals in our applications, specifically in relation with Lemmas \ref{RestUltrapowTemp} and \ref{RestChainTemp}. The reader may note that statements and proofs in this paper can be translated directly to the case of simple indexed templates, but this is not relevant for our applications.
\end{remark}

\begin{example}\label{ExmpTemp}
  \begin{enumerate}[(1)]
     \item Given a linear order $L$, $\Iwf_x=[L_x]^{<\omega}$ for $x\in L$ form an indexed template on $L$. Note that $\hat{\Iwf}_x=\Iwf_x$ and, for $X\subseteq L$,
         \[\mathrm{Dp}(X)=\left\{\begin{array}{ll}
            |X| & \textrm{if $X$ is finite,}\\
            \omega & \textrm{otherwise.}
         \end{array}\right.\]
     \item (Template for a fsi) Let $\delta$ be an ordinal number. Then, $\Iwf_\alpha:=\alpha+1=\left\{\xi\ /\ \xi\leq\alpha\right\}$ for $\alpha<\delta$ form an indexed template on $\delta$. This is the template structure that corresponds to a fsi of length $\delta$. Note that $\hat{\Iwf}_\alpha=\pts(\alpha)$ and, for $X\subseteq\delta$, $\mathrm{Dp}(X)$ is the order type of $X$.
  \end{enumerate}
\end{example}

\begin{definition}[Innocuous extension]\label{DefInnoc}
   Let $\langle L,\bar{\Iwf}\rangle$ be an indexed template and $\theta$ an uncountable cardinal.
   \begin{enumerate}[(I)]
       \item An indexed template $\langle L,\bar{\Jwf}\rangle$ is a \emph{$\theta$-innocuous extension of $\langle L,\bar{\Iwf}\rangle$} if, for every $x\in L$,
           \begin{enumerate}[(1)]
              \item $\Iwf_x\subseteq\Jwf_x$ and
              \item if $A\in\Jwf_x$ and $X\subseteq A$ has size $<\theta$ then there exists a $C\in\Iwf_x$ containing $X$.
           \end{enumerate}
           If in (2) we can even find $C\subseteq A$, say that $\langle L,\bar{\Jwf}\rangle$ is a
           \emph{strongly $\theta$-innocuous extension of} $\langle L,\bar{\Iwf}\rangle$.
       \item Let $\langle L',\bar{\Iwf}'\rangle$ be an indexed template such that $L'$ is a linear order extending $L$. $\langle L',\bar{\Iwf}'\rangle$ is a \emph{(strongly) $\theta$-innocuous extension of $\langle L,\bar{\Iwf}\rangle$} if
           \begin{enumerate}[(1)]
              \item for every $x\in L$, $\Iwf'_x\frestr L\subseteq\Iwf'_x$ and
              \item $\langle L,\bar{\Iwf}'\frestr L\rangle$ is a (strongly) $\theta$-innocuous extension of $\langle L,\bar{\Iwf}\rangle$.
           \end{enumerate}
   \end{enumerate}
\end{definition}

The main point of this definition is that, when two iterations are defined along templates where one is an innocuous extension of the other and where some ``coherence'' is ensured in the construction of both iterations, we can get regular contention or even equivalence between the resulting posets. The results that express this are Corollary \ref{EmbCor} and Lemma \ref{InnEqv}.

In Lemmas \ref{AddSmallTemp}, \ref{UtrapowTemp} and \ref{ChainTemp}, the statements about innocuity are more general than in the cited reference \cite{br}. However, their proofs are either easy or very similar to those presented there.

\begin{lemma}[{\cite[Lemma 1.3]{br}}]\label{AddSmallTemp}
   Let $\langle L,\bar{\Iwf}\rangle$ be an indexed template, $L_0\subseteq L$. For $x\in L$, define
   $\Jwf_x:=\left\{A\cup(B\cap L_0)\ /\ A,B\in\Iwf_x\right\}$. Then, $\langle L,\bar{\Jwf}\rangle$ is an indexed template which is a $\theta$-innocuous extension of $\langle L,\bar{\Iwf}\rangle$ and a strongly $\theta$-innocuous extension of $\langle L_0,\bar{\Iwf}\frestr L_0\rangle$
   for any uncountable cardinal $\theta$. Moreover, $\bar{\Jwf}\frestr L_0=\bar{\Iwf}\frestr L_0$.
\end{lemma}


Fix a measurable cardinal $\kappa$ with a non-principal $\kappa$-complete ultrafilter $\Dwf$ and
let $\langle L,\bar{\Iwf}\rangle$ be an indexed template. Put $L^*:=L^\kappa/\Dwf$, which is a linear order. For $\bar{x}=\langle x_\alpha\rangle_{\alpha<\kappa}/\Dwf\in L^*$, let $\Iwf^*_{\bar{x}}$ be the family of sets of the form
$\bar{A}:=[\{A_\alpha\}_{\alpha<\kappa}]=\prod_{\alpha<\kappa}A_\alpha/\Dwf$ where $\{A_\alpha\}_{\alpha<\kappa}$ is a sequence of subsets of $L$ such that $A_\alpha\in\Iwf_{x_\alpha}$ for $\Dwf$-many $\alpha$. Identifying the members of $L$ with constant functions in $L^*$, $L^*$ extends the linear order $L$ and $\Iwf_x\subseteq\Iwf'_x:=\Iwf^*_x\frestr L$ for all $x\in L$. For $\bar{x}\in L^*$, let $\Iwf^\dagger_{\bar{x}}=\left\{A\cup(B\cap L)\ /\ A,B\in\Iwf^*_{\bar{x}}\right\}$. Notice that $\Iwf'_x=\Iwf^\dagger_x\frestr L\subseteq\Iwf^\dagger_x$ for all $x\in L$. From Lemma \ref{AddSmallTemp}, we get

\begin{lemma}[{\cite[Lemma 2.1]{br}}]\label{UtrapowTemp}
   \begin{enumerate}[(a)]
      \item $\langle L^*,\bar{\Iwf}^*\rangle$ is an indexed template.
      \item $\langle L,\bar{\Iwf}'\rangle$ is an indexed template which is a strongly $\kappa$-innocuous extension of $\langle L,\bar{\Iwf}\rangle$.
      \item $\langle L^*,\bar{\Iwf}^\dagger\rangle$ is an indexed template which is a
            $\theta$-innocuous extension of $\langle L^*,\bar{\Iwf}^*\rangle$ and a strongly
            $\theta$-innocuous extension of $\langle L,\bar{\Iwf}'\rangle$ for any uncountable cardinal $\theta$.
      \item $\langle L^*,\bar{\Iwf}^\dagger\rangle$ is a strongly $\kappa$-innocuous extension of $\langle L,\bar{\Iwf}\rangle$.
   \end{enumerate}
\end{lemma}

Typically, given a poset $\Qor$ that comes from an iteration along the template $\langle L,\bar{\Iwf}\rangle$, its ultrapower is (forcing equivalent to) an iteration along $\langle L^*,\bar{\Iwf}^*\rangle$. Also, $\bar{\Iwf}^\dagger$ is very close to $\bar{\Iwf}^*$, so there is an iteration along $\langle L^*,\bar{\Iwf}^\dagger\rangle$ that gives a poset which is forcing equivalent to the ultrapower of $\Qor$. This procedure is used for the inductive step of the construction of the chain of template iterations of the proof of Theorem \ref{AppSplitting}. Though $\bar{\Iwf}^*$ and $\bar{\Iwf}^\dagger$ may define the same iteration for $\Qor^\kappa/\Dwf$, $\bar{\Iwf}^\dagger$ is preferred because of (d) and Lemma \ref{RestChainTemp} (after constructing a chain of templates by ultrapowers, see also Remark \ref{RemChain}).

\begin{lemma}\label{RestUltrapowTemp}
   Fix $\theta<\kappa$ an infinite cardinal. Assume that $|\Iwf(X)|<\theta$ for all $X\in[L]^{<\theta}$. Then, for every $\bar{X}\in[L^*]^{<\theta}$, $|\Iwf^\dagger(\bar{X})|<\theta$.
\end{lemma}
\begin{proof}
   Let $\bar{X}=\left\{\bar{x}^\xi\ /\ \xi<\nu\right\}$ for some $\nu<\theta$. For $\alpha<\kappa$ let $X_\alpha:=\left\{x^\xi_\alpha\ /\ \xi<\nu\right\}$. Then, $\bar{X}=[\{X_\alpha\}_{\alpha<\kappa}]$, so any
   $Z\in\Iwf^\dagger(\bar{X})$ comes from two objects of the form $\bar{Y}=[\{Y_\alpha\}_{\alpha<\kappa}]$ where $Y_\alpha\in\Iwf(X_\alpha)$ for $\Dwf$-many $\alpha$. But, as $\theta<\kappa$ and each $|\Iwf(X_\alpha)|<\theta$, there exists
   $\nu'<\theta$ such that $|\Iwf(X_\alpha)|=\nu'$ for $\Dwf$-many $\alpha$. Therefore,
   $|\Iwf^\dagger(\bar{X})|\leq(\nu')^2<\theta$.
\end{proof}

Now we deal with the construction of a ``limit'' of templates, which is relevant for the limit step in the proof of Theorem \ref{AppSplitting}.
Fix, until the end of this section, an uncountable cardinal $\theta$ and consider a chain of indexed templates $\big\{\langle L^\alpha,\bar{\Iwf}^\alpha\rangle\big\}_{\alpha<\delta}$ such that, for $\alpha<\beta<\delta$, $\langle L^\beta,\bar{\Iwf}^\beta\rangle$ is a strongly $\theta$-innocuous extension of $\langle L^\alpha,\bar{\Iwf}^\alpha\rangle$. Moreover, assume that there is an ordinal $\mu\subseteq L^0$ such that, for all $\alpha<\delta$,
\begin{enumerate}[(i)]
   \item $\mu$ is cofinal in $L^\alpha$ and
   \item $L^\alpha_\xi\in\Iwf^\alpha_\xi$ for all $\xi\in\mu$.
\end{enumerate}
 Define $L^\delta:=\bigcup_{\alpha<\delta}L^\alpha$ and, for $x\in L^\delta$, let $\Iwf_x:=\bigcup_{\alpha\in[\alpha_x,\delta)}\Iwf^\alpha_x$ where $\alpha_x$ is the least $\alpha$ such that $x\in L^\alpha$. Also, put $\Jwf_x:=\Iwf_x\cup\left\{L^\delta_\xi\cup A\ /\ \xi\in\mu\right.$, $\xi\leq x$ and $\left.A\in\Iwf_x\right\}$.

 \begin{lemma}[{\cite[Lemma 1.8]{br}}]\label{ChainTemp}
   \begin{enumerate}[(a)]
    \item $\langle L^\delta,\bar{\Iwf}\rangle$ is an indexed template which is a strongly $\theta$-innocuous extension of $\langle L^\alpha,\bar{\Iwf}^\alpha\rangle$ for all $\alpha<\delta$.
    \item $\langle L^\delta,\bar{\Jwf}\rangle$ is an indexed template which is a strongly $\theta$-innocuous extension of $\langle L^\alpha,\bar{\Iwf}^\alpha\rangle$ for all $\alpha<\delta$. Moreover, if $\cf(\delta)\geq\theta$, then $\langle L^\delta,\bar{\Jwf}\rangle$ is a strongly $\theta$-innocuous extension of $\langle L^\delta,\bar{\Iwf}\rangle$.
   \end{enumerate}
 \end{lemma}

 Note that properties (i) and (ii) also hold for the template $\langle L^\delta,\bar{\Jwf}\rangle$, but (ii) may not hold for $\langle L^\delta,\bar{\Iwf}\rangle$. Although, in many cases, both templates lead to the same template iteration construction when $\cf(\delta)\geq\theta$, $\bar{\Jwf}$ is preferred over $\bar{\Iwf}$ because of property (ii).

 Like Lemma \ref{RestUltrapowTemp}, the following result states that, in the resulting template, the property of having small templates when restricting to a small set is preserved. This is needed to use Theorem \ref{PresTemp} in Section \ref{SecAppl}.

\begin{lemma}\label{RestChainTemp}
   Assume that $\nu\leq\theta$ is a regular cardinal and that, for each $\alpha<\delta$ and $X\in[L^\alpha]^{<\nu}$, $|\Iwf^\alpha(X)|<\nu$. Then, $|\Iwf(X)|<\nu$ and $|\Jwf(X)|<\nu$ for any $X\in[L^\delta]^{<\nu}$.
\end{lemma}
\begin{proof}
   If $\cf(\delta)<\nu$, choose an increasing cofinal sequence $\{\alpha_\eta\}_{\eta<\cf(\delta)}$ for $\delta$
   and note that $\Iwf(X)\subseteq\{X\}\cup\bigcup_{\eta<\cf(\delta)}\Iwf^{\alpha_\eta}(X\cap L^{\alpha_\eta})$ for any $X\subseteq L^\delta$, so it has size $<\nu$ when $X$ does. In the case that $\cf(\delta)\geq\nu$, if $X\in[L^\delta]^{<\nu}$, there exists an $\alpha<\delta$ such that $X\subseteq L^\alpha$. We claim that $\Iwf(X)=\Iwf^\alpha(X)$. If $Z\in\Iwf(X)$, then either $Z=X\in\Iwf^\alpha(X)$ or $Z=X\cap H$ for some
   $H\in\Iwf^\beta_x$ with $x\in X$ and $\alpha<\beta<\delta$. As $|Z|<\nu$, by strong $\theta$-innocuity, we can find a $C\in\Iwf^\alpha_x$ such that $Z\subseteq C\subseteq H$, so $Z=C\cap X\in\Iwf^\alpha_x\frestr X$.

   For the case of $\Jwf$, note that $\left\{L^\delta_\xi\cap X\ /\ \xi\leq\mu\right\}$ has size $\leq|X|$. As, for any $X\subseteq L^\delta$, $\Jwf(X)=\{(L^\delta_\xi\cap X)\cup Z\ /$ $\xi\leq\mu$ and $Z\in\Iwf(X)\}$, then it has size $<\nu$ when $X$ does.
\end{proof}

\begin{remark}\label{RemChain}
   In the chain of templates, if we just assume that $\langle L^\beta,\bar{\Iwf}^{\beta}\rangle$ is a $\theta$-innocuous extension of $\langle L^\alpha,\bar{\Iwf}^{\alpha}\rangle$ (and also assume (i) and (ii) as well), then Lemma \ref{ChainTemp} is still valid for $\theta$-innocuity (not strongly). Although Lemma \ref{RestChainTemp} may not be valid, it can be reformulated: if $\hat{\Iwf}^\alpha(X)$ is generated by $<\nu$ objects for all $X\in[L^\alpha]^{<\nu}$ and $\alpha<\delta$, then $\hat{\Iwf}(X)$ and $\hat{\Jwf}(X)$ are generated by $<\nu$ objects for all $X\in[L^\delta]^{<\nu}$, where we denote $\hat{\Iwf}(X):=\bigcup_{z\in X}\hat{\Iwf}_z\frestr X$ (which is clearly an ideal on $X$). This reformulation is also valid in the context of simple indexed templates (see Remark \ref{RemSimpleTemp}) and Lemma \ref{RestUltrapowTemp} can be reformulated in a similar way (note that we cannot define strong innocuity in that context).
\end{remark}


\section{Iterations along templates}\label{SecItTemp}

We present the theory of template iterations for non-definable posets. Although this approach is general, proofs are not different from those in \cite{brendle2}, actually, our presentation is based on this reference. We can say that the original version of template iterations with definable forcings (in \cite{shelah}) corresponds to Example \ref{ExpTempItFund} with $L_C=\varnothing$.

\begin{theorem}[Iteration along a template]\label{TempIt}
   Given a template $\langle L,\bar{\Iwf}\rangle$, a partial order $\Por\frestr A$ is defined by recursion on $\alpha=\mathrm{Dp}(A)$ for all $A\subseteq L$ with the following conditions.
   \begin{enumerate}[(1)]
      \item For $x\in L$ and $B\in\hat{\Iwf}_x$, $\Qnm^B_x$ is a $\Por\frestr B$-name of a poset. The following conditions should hold.
            \begin{enumerate}[(i)]
               \item If $E\subseteq B$ and $\Por\frestr E\lessdot\Por\frestr B$, then $\Vdash_{\Por\upharpoonright B}\Qnm_x^E\lessdot_{V^{\Por\upharpoonright E}}\Qnm_x^B$.
               \item If $E\in\hat{\Iwf}_x$, $\Por\frestr(B\cap E)$ is a regular subposet of both $\Por\frestr B$ and $\Por\frestr E$, and $\dot{q}$ is a $\Por\frestr(B\cap E)$-name such that $\Vdash_{\Por\upharpoonright E}\dot{q}\in\Qnm_x^E$                     and $\Vdash_{\Por\upharpoonright B}\dot{q}\in\Qnm_x^B$, then $\Vdash_{\Por\upharpoonright(B\cap E)}\dot{q}\in\Qnm_x^{B\cap E}$.
               \item If $B',D\subseteq B$ and $\langle\Por\frestr(B'\cap D),\Por\frestr B',\Por\frestr D,\Por\frestr B\rangle$ is a correct diagram, then the diagram ${\langle\Por\frestr(B'\cap D)\ast\Qnm_x^{B'\cap D}},\Por\frestr B'\ast\Qnm_x^{B'},\Por\frestr D\ast\Qnm_x^D,\Por\frestr B\ast\Qnm_x^B\rangle$ is correct.
            \end{enumerate}
      \item The partial order $\Por\frestr A$ is defined by:
            \begin{enumerate}[(i)]
               \item $\Por\frestr A$ consists of all finite partial functions $p$ with domain contained in $A$ such that $p=\varnothing$ or, if $|p|>0$ and $x=\max(\dom p)$, then there exists a $B\in\Iwf_x\frestr A$ such that $p\frestr L_x\in\Por\frestr B$ and $p(x)$ is a $\Por\frestr B$-name for a condition in $\Qnm_x^B$.
               \item The order on $\Por\frestr A$ is given by: $q\leq_A p$ if $\dom p\subseteq\dom q$ and either $p=\varnothing$ or, when $p\neq0$ and $x=\max(\dom q)$, there is a $B\in\Iwf_x\frestr A$ such that $q\frestr L_x\in\Por\frestr B$ and, either $x\notin\dom p$, $p\in\Por\frestr B$ and $q\frestr L_x\leq_B p$, or $x\in\dom p$, $p\frestr L_x\in\Por\frestr B$, $q\frestr L_x\leq_B p\frestr L_x$ and $p(x),q(x)$ are $\Por\frestr B$-names for conditions in $\Qnm_x^B$ such that $q\frestr L_x\Vdash_{\Por\upharpoonright B}q(x)\leq p(x)$.
            \end{enumerate}
   \end{enumerate}
   Within this recursion, the following properties are proved.
   \begin{enumerate}[(a)]
      \item If $p\in\Por\frestr A$, $x\in A$ and $\max(\dom p)<x$, then there exists a $B\in\Iwf_x\frestr A$ such that $p\in\Por\frestr B$.
      \item For $D\subseteq A$, $\Por\frestr D\subseteq\Por\frestr A$ and, for $p,q\in\Por\frestr D$, $q\leq_D p$ iff $q\leq_A p$.
      \item $\Por\frestr A$ is a poset.
      \item $\Por\frestr A$ is obtained from some posets of the form $\Por\frestr B$ with $B\subsetneq A$ in the following way:
            \begin{enumerate}[(i)]
               \item If $x=\max(A)$ exists and $A_x:=A\cap L_x\in\hat{\Iwf}_x$, then $\Por\frestr A=\Por\frestr A_x\ast\Qnm_x^{A_x}$.
               \item If $x=\max(A)$ but $A_x\notin\hat{\Iwf}_x$, then $\Por\frestr A$ is the direct limit of the $\Por\frestr B$ where
                     $B\subseteq A$ and $B\cap L_x\in\Iwf_x\frestr A$.
               \item If $A$ does not have a maximum element, then $\Por\frestr A$ is the direct limit of the $\Por\frestr B$ where $B\in\Iwf_x\frestr A$ for some $x\in A$ (in the case $A=\varnothing$, it is clear that $\Por\frestr A=\mathds{1}$).
            \end{enumerate}
            Note that, by Lemma \ref{UpsilonTemp}(c), $\mathrm{Dp}(A_x)<\mathrm{Dp}(A)$ in (i) and, in (ii) and (iii), $\mathrm{Dp}(B)<\mathrm{Dp}(A)$ for each corresponding $B$.
      \item If $D\subseteq A$, then $\Por\frestr D$ is a regular subposet of $\Por\frestr A$.
      \item If $D\subseteq L$ then $\Por\frestr(A\cap D)=\Por\frestr A\cap\Por\frestr D$.
      \item If $D,A'\subseteq A$ then $\langle\Por\frestr(A'\cap D),\Por\frestr A',\Por\frestr D,\Por\frestr A\rangle$ is a correct diagram.
   \end{enumerate}
\end{theorem}
\begin{proof}
   By just changing certain notation, the proof follows the same ideas as \cite[Thm. 2.2]{brendle2}. Lemma \ref{UpsilonTemp} guaranties that (2) can be defined recursively by the function $\mathrm{Dp}$.
   \begin{enumerate}[(a)]
      \item Denote $z:=\max(\dom p)$. By (2)(i), there is an $E\in\Iwf_z\frestr A$ such that $p\frestr L_z\in\Por\frestr E$ and $p(z)$ is a $\Por\frestr E$-name for a condition in $\Qnm^E_z$. By Definition \ref{DefIndxTemp} and Lemma \ref{tracetemp}, there is some $B\in\Iwf_x\frestr A$ such that $z\in B$. We may assume that $E\in\Iwf_z\frestr B$ (as $E=A\cap H$ and $B=A\cap H'$ for some $H\in\Iwf_z$ and $H'\in\Iwf_x$, just redefine $B$ as $A\cap(H\cup H')$). Thus, $p\in\Por\frestr B$.
      \item Let $p\in\Por\frestr D$ and assume that $p\neq\varnothing$, so let $x=\max(\dom p)$. By (2), there is an $E\in\Iwf_x\frestr D$ such that $p\frestr L_x\in\Por\frestr E$ and $p(x)$ is a $\Por\frestr E$-name for a condition in $\Qnm^E_x$. Also, there exists an $H\in\Iwf_x$ such that $E=D\cap H$. Put $B:=A\cap H\in\Iwf_x\frestr A$. As $E\subseteq B$ and $\mathrm{Dp}(B)<\mathrm{Dp}(A)$ (see Lemma \ref{UpsilonTemp}), by induction hypothesis and (e), $\Por\frestr E\lessdot\Por\frestr B$, so $p\frestr L_x\in\Por\frestr B$. Moreover, by (1)(i), $p(x)$ is a $\Por\frestr B$-name for a condition in $\Qnm^B_x$, so $p\in\Por\frestr A$.

      Now, fix $p,q\in\Por\frestr D$. Assume that $q\leq_D p$ and $x=\max(\dom p)=\max(\dom q)$. By (2)(ii), there exists an $E\in\Iwf_x\frestr D$ such that $p\frestr L_x,q\frestr L_x\in\Por\frestr E$, $q\frestr L_x\leq_E p\frestr L_x$ and $p(x)$ and $q(x)$ are $\Por\frestr E$-names for conditions in $\Qnm^E_x$ such that $q\frestr L_x\Vdash_{\Por\upharpoonright E}q(x)\leq_{\Qnm^E_x} p(x)$. Also, there is an $H\in\Iwf_x$ such that $E= D\cap H$. Put $B=A\cap H$ so, by induction hypothesis, $q\frestr L_x\leq_B p\frestr L_x$, $p(x)$ and $q(x)$ are $\Por\frestr B$-names for conditions in $\Qnm^B_x$ and $q\frestr L_x\Vdash_{\Por\upharpoonright B}q(x)\leq_{\Qnm^B_x} p(x)$. Clearly, $q\leq_A p$. The case $\max(\dom p)<\max(\dom q)$ is treated similarly.

      To prove the converse, assume $q\leq_A p$ and $x=\max(\dom p)=\max(\dom q)$. $p,q\in\Por\frestr D$ implies that there is an $E\in\Iwf_x\frestr D$ such that $p\frestr L_x,q\frestr L_x\in\Por\frestr E$ and $p(x)$ and $q(x)$ are $\Por\frestr E$-names for conditions in $\Qnm^E_x$ (in fact, we find $E_1$ for $p$, $E_2$ for $q$ and put $E=E_1\cup E_2$, so induction hypothesis and (e) are used). On the other hand, $q\leq_A p$ implies that there is a $B\in\Iwf_x\frestr A$ such that the statement in (2)(ii) holds. We may assume that $E\subseteq B$ (there are $H,H'\in\Iwf_x$ such that $E=D\cap H$ and $B=A\cap H'$, so just redefine $B$ as $A\cap (H\cup H')$ and note that the induction hypothesis and (e) are used to see that the statement in (2)(ii) still holds). By induction hypothesis, $q\frestr L_x\leq_E p\frestr L_x$ and $q\frestr L_x\Vdash_{\Por\upharpoonright E}q(x)\leq_{\Qnm^E_x}p(x)$, so $q\leq_D p$. The case $\max(\dom p)<\max(\dom q)$ is treated similarly, but it requires (a).

      \item Reflexivity of $\leq_A$ is easy by the induction hypothesis, so we prove transitivity. Let $p,q,r\in\Por\frestr A$ be such that $r\leq_A q$ and $q\leq_A p$. Assume that $x=\max(\dom p)=\max(\dom q)=\max(\dom r)$ (the other cases are treated similarly). We can find a $B\in\Iwf_x\frestr A$ such that $p\frestr L_x,q\frestr L_x,r\frestr L_x\in\Por\frestr B$, $r\frestr L_x\leq_B q\frestr L_x$, $q\frestr L_x\leq_B p\frestr L_x$ and $p(x),q(x),r(x)$ are $\Por\frestr B$-names for conditions in $\Qnm^B_x$ such that $r\frestr L_x\Vdash_{\Por\upharpoonright B}r(x)\leq q(x)$ and $q\frestr L_x\Vdash_{\Por\upharpoonright B}q(x)\leq p(x)$. By induction hypothesis, it is clear that $r\frestr L_x\leq_B p\frestr L_x$ and $r\frestr L_x\Vdash_{\Por\upharpoonright B}r(x)\leq p(x)$.
      \item \begin{enumerate}[(i)]
         \item It is enough to show that the set $\{p\in\Por\frestr A\ /\ x=\max(\dom p)\textrm{\ and }p\frestr L_x\in\Por\frestr A_x\}$ is dense in $\Por\frestr A$. Let $p\in\Por\frestr A$. If either $p=\varnothing$ or $\max(\dom p)<x$ then, by (a), $p\in\Por\frestr B$ for some $B\in\Iwf_x\frestr A$, so $p\in\Por\frestr A_x$ and $p\widehat{\ \ }\langle\dot{q}\rangle_x\leq_A p$ for some $\Por\frestr A_x$-name $\dot{q}$ for a condition in $\Qnm^{A_x}_x$. On the other hand, if $\max(\dom p)=x$, then it is clear that $p\frestr L_x\in\Por\frestr A_x$.
         \item Let $p\in\Por\frestr A$. If either $p=\varnothing$ or $\max(\dom p)<x$ then there is a $B\in\Iwf_x\frestr A$ such that $p\in\Por\frestr B$ (by (a)), so assume that $\max(\dom p)=x$. There is an $E\in\Iwf_x\frestr A$ such that $p\frestr L_x\in\Por\frestr E$ and $p(x)$ is a $\Por\frestr E$-name for a condition in $\Qnm^E_x$. Put $B:=E\cup\{x\}$. It is clear that $B\cap L_x=E\in\Iwf_x\frestr B$ and that $p\in\Por\frestr B$. On the other hand, by induction hypothesis and (e), $\{\Por\frestr B\ /\ B\subseteq A\textrm{\ and\ }B\cap L_x\in\Iwf_x\frestr A\}$ is a directed system of posets, so $\Por\frestr A$ is its direct limit.
         \item Let $p\in\Por\frestr A$ and $y=\max(\dom p)$. As there is some $x\in A$ above $y$, there is some $B\in\Iwf_x\frestr A$ such that $p\in\Por\frestr B$ by (a). On the other hand, by induction hypothesis and (e), $\{\Por\frestr B\ /\ \exists_{x\in A}(B\in\Iwf_x\frestr A)\}$ is a directed system of posets, so $\Por\frestr A$ is its direct limit.
         \end{enumerate}
      \item We argue by cases from (d).
         \begin{enumerate}[(i)]
            \item If $x\notin D$ then $D\subseteq A_x$. It is clear that $\Por\frestr A_x\lessdot\Por\frestr A$ and, by induction hypothesis (as $\mathrm{Dp}(A_x)<\mathrm{Dp}(A)$), $\Por\frestr D\lessdot\Por\frestr A_x$. Assume $x\in D$ otherwise. Note that $D_x:=D\cap L_x\in\hat{\Iwf}_x$, so $\Por\frestr D_x\lessdot\Por\frestr A_x$ by induction hypothesis. Then, by (1)(i), Lemma \ref{2stepitemb} and (d)(i), $\Por\frestr D\lessdot\Por\frestr A$.
            \item We proceed by the following cases.
                 \begin{itemize}
                    \item $D_x:=D\cap L_x\in\hat{\Iwf}_x$. Then, there is some $B_x\in\Iwf_x\frestr A$ such that $D_x=D\cap B_x$. Put $B:=B_x\cup\{x\}$. $\Por\frestr B\lessdot\Por\frestr A$ by (d)(ii) and, by induction hypothesis, $\Por\frestr D\lessdot\Por\frestr B$.
                    \item $D_x\notin\hat{\Iwf}_x$. We first assume that $x\in D$. Then, $\Por\frestr D=\limdir_{E\in\Dwf}\Por\frestr E$ where $\Dwf:=\{E\subseteq D\ /\ E\cap L_x\in\Iwf_x\frestr D\}$ (we can apply (d)(ii) here because $\mathrm{Dp}(D)\leq\mathrm{Dp}(A)$). Clearly, $\Dwf=\{B\cap D\ /\ B\in\Awf\}$ where $\Awf:=\{B\subseteq A\ /\ B\cap L_x\in\Iwf_x\frestr A\}$ and, for each $B,B'\in\Awf$, if $B\subseteq B'$, then $\langle\Por\frestr(B\cap D,\Por\frestr(B'\cap D),\Por\frestr B,\Por\frestr B')\rangle$ is correct by induction hypothesis and (g). Therefore, by Lemma \ref{dirlimEmb}, $\Por\frestr D=\limdir_{B\in\Awf}\Por\frestr(B\cap D)$ is a regular subposet of $\Por\frestr A$.

                    Now, we assume that $x\notin D$. $D_x\notin\hat{\Iwf}_x$ implies that, whenever $D$ has a maximum $z$, $D_z:=D\cap L_z\notin\hat{\Iwf}_z$, so $\Por\frestr D$ is described as a direct limit from (d)(ii) or (iii). In either case, $\Por\frestr D=\bigcup_{B\in\Awf}\Por\frestr(B\cap D)$, moreover, as $\{\Por\frestr(B\cap D)\ /\ B\in\Awf\}$ is a directed system of posets by induction hypothesis (because $\mathrm{Dp}(B\cap D)\leq\mathrm{Dp}(B)<\alpha$ for all $B\in\Awf$), $\Por\frestr D=\limdir_{B\in\Awf}\Por\frestr(B\cap D)$.
                    Hence, as in the previous argument, $\Por\frestr D\lessdot\Por\frestr A$.
                 \end{itemize}
                 \item If $D\in\hat{\Iwf}_x$ for some $x\in A$, we can find some $B\in\Iwf_x\frestr A$ such that $D\subseteq B$, so $\Por\frestr D\lessdot\Por\frestr B$ (by induction hypothesis) and, by (d)(iii), it is clear that the latter is a regular subposet of $\Por\frestr A$. So assume that $D\notin\hat{\Iwf}_x$ for any $x\in A$. Proceeding like in the previous paragraph, $\Por\frestr D=\limdir_{B\in\Awf}\Por\frestr(B\cap D)$ with $\Awf:=\{B\subseteq A\ /\ \exists_{x\in A}(B\in\Iwf_x\frestr A)\}$, so Lemma \ref{dirlimEmb} implies $\Por\frestr D\lessdot\Por\frestr A$.
         \end{enumerate}
      \item We prove the statement for all $D\subseteq L$ with $\mathrm{Dp}(D)\leq\alpha$. By (b), it is clear that $\Por\frestr(A\cap D)\subseteq\Por\frestr A\cap\Por\frestr D$. To prove the converse contention, assume $p\in\Por\frestr A\cap\Por\frestr D$ with $x=\max(\dom p)$. Then, there are $B\in\Iwf_x\frestr A$ and $E\in\Iwf_x\frestr D$ such that $p\frestr L_x\in\Por\frestr B\cap\Por\frestr E$ and $p(x)$ is a $\Por\frestr B$-name for a condition in $\Qnm^B_x$ as well as a $\Por\frestr E$-name for a condition in $\Qnm^E_x$. We may assume that $B\cap E\in\Iwf_x\frestr(A\cap D)$ (by increasing $B$ and $E$ so that there is an $H\in\Iwf_x$ such that $B=A\cap H$ and $E=D\cap H$). By induction hypothesis, as $\mathrm{Dp}(B),\mathrm{Dp}(E)<\alpha$, $\Por\frestr(B\cap E)=\Por\frestr B\cap \Por\frestr E$, so $p\frestr L_x\in\Por\frestr (B\cap E)$. Clearly, $p(x)$ is a $\Por\frestr(B\cap E)$-name\footnote{Considering the formal definition of a name (see, e.g., \cite{kunen}), if $\Por$ and $\Qor$ are posets, $\dot{x}$ is a $\Por$-name and, at the same time, a $\Qor$-name, then it is a $\Por\cap\Qor$-name.}. Thus, by (1)(ii), $p(x)$ is a $\Por\frestr(B\cap E)$-name for a condition in $\Qnm^{B\cap E}_x$.
      \item We split into cases according to (d).
         \begin{enumerate}[(i)]
            \item Here, $A'_x:=A'\cap L_x$ and $D_x:=D\cap L_x$ are subsets of $A_x$, so they are in $\hat{\Iwf}_x$. By induction hypothesis, $\langle\Por\frestr(A'\cap D\cap L_x),\Por\frestr A'_x,\Por\frestr D_x,\Por\frestr A_x\rangle$ is correct, so the result follows (do cases for $x$ being in $A'$ or in $D$, use (1)(iii) in the case $x\in A'\cap D$ and use Lemma \ref{lemmacorrpres}(a) in the other cases).
            \item Let $p\in\Por\frestr A'$ and $r\in\Por\frestr(A'\cap D)$ a reduction of $p$. We first assume that $D_x\in\hat{\Iwf}_x$. Find $B\in\Awf:=\{B\subseteq A\ /\ B\cap L_x\in\Iwf_x\frestr A\}$ such that $D\subseteq B$ and $p\in\Por\frestr B$ (by (d)(ii)). Put $B':=A'\cap B$, so $p\in\Por\frestr B'$ by (f) and, as $A'\cap D= B'\cap D$, $r$ is a reduction of $p$ with respect to $\Por\frestr(A'\cap D),\Por\frestr B'$. On the other hand, by induction hypothesis, $\langle\Por\frestr(A'\cap D),\Por\frestr B',\Por\frestr D,\Por\frestr B\rangle$ is correct, so $r$ is a reduction of $p$ with respect to $\Por\frestr D,\Por\frestr B$. Hence, this is so with respect to $\Por\frestr D,\Por\frestr A$.

                Now, assume that $D_x\notin\hat{\Iwf}$, so $\Por\frestr D=\limdir_{B\in\Awf}\Por\frestr(B\cap D)$ (see the second case of the proof of (e)). Choose $B\in\Awf$ such that $p\in\Por\frestr B$ and $r\in\Por\frestr(B\cap D)$. Put $B'=A'\cap B$. As before, $p\in\Por\frestr B'$ and $r\in\Por\frestr(B'\cap D)$ by (f) and $\langle\Por\frestr(B'\cap D),\Por\frestr B',\Por\frestr(B\cap D),\Por\frestr B\rangle$ is correct by induction hypothesis. Clearly, $r$ is a reduction of $p$ with respect to $\Por\frestr(B'\cap D),\Por\frestr B'$ and, by correctness, it is with respect to $\Por\frestr(B\cap D),\Por\frestr B$. We claim that $r$ is a reduction of $p$ with respect to $\Por\frestr D,\Por\frestr A$. Indeed, if $q\leq_D r$, find $B_1\in\Awf$ containing $B$ such that $q\in\Por\frestr(B_1\cap D)$. The diagram $\langle\Por\frestr(B\cap D),\Por\frestr B,\Por\frestr(B_1\cap D),\Por\frestr B_1\rangle$ is correct, which implies that $r$ is a reduction of $p$ with respect to $\Por\frestr(B_1\cap D),\Por\frestr B_1$, so $q$ is compatible with $p$ in $\Por\frestr B_1$ (and so in $\Por\frestr A$).

            \item By cases on whether $\exists_{x\in A}(D\in\Iwf_x\frestr A)$ or not, a similar argument as before (using facts from the proof of (e) as well) works.
         \end{enumerate}
   \end{enumerate}
\end{proof}

Condition (1), particularly item (i), implies that, when we step into the generic extension of $\Por\frestr L$, the generic object added at stage $x$ is generic over the intermediate extension by $\Por\frestr B$ for any $B\in\hat{\Iwf}_x$. In general, as $L_x$ may not belong to $\hat{\Iwf}_x$ (that is, to $\Iwf_x$), this object added at stage $x$ need \underline{not} be generic over the intermediate extension by $\Por\frestr L_x$ or over the extension for any subset of $L_x$ that is not in $\hat{\Iwf}_x$.

\begin{example}[Fsi in terms of a template iteration]\label{ExmpFSI}
   Let $\delta$ be an ordinal and consider the template $\bar{\Iwf}$ defined in Example \ref{ExmpTemp}(2). An iteration along $\langle\delta,\bar{\Iwf}\rangle$, defined as in Theorem \ref{TempIt}, is equivalent to the fsi $\langle\Por\frestr\alpha,\Qnm^\alpha_\alpha\rangle_{\alpha<\delta}$. Unlike a standard fsi, this iteration can be restricted to any subset of $\delta$. To be more precise, if $X\subseteq\delta$, then $\Por\frestr X$ is equivalent to the fsi $\langle\Por\frestr(X\cap\alpha),\Qnm^{X\cap\alpha}_\alpha\rangle_{\alpha\in X}$, which is a regular subposet of $\Por\frestr\delta$. Recall that, for any $\alpha<\delta$, $\hat{\Iwf}_\alpha=\pts(\alpha)$, so the generic object added at stage $\alpha$ is generic over the intermediate extension by $\Por\frestr X$ for any $X\subseteq\alpha$.

   Of course, the proof of Theorem \ref{TempIt} is much simpler for this template, for it is enough to have the conditions in (1) and prove, by induction on $\alpha\leq\delta$, that $\Por\frestr X$ is defined for any $X\subseteq\alpha$ and that properties (a)-(g) hold.
\end{example}

\begin{example}\label{ExpTempItFund}
   Let $\langle L,\Iwf\rangle$ be an indexed template, $L=L_S\cup L_C$ a disjoint union. For $x\in L$ define the orders $\Qnm_x^B$ for $B\in\hat{\Iwf}_x$ according to one of the following cases.
   \begin{enumerate}[(i)]
      \item If $x\in L_S$, $\Qnm_x^B$ is a $\Por\frestr B$-name for $\Sor_x^{V^{\Por\upharpoonright B}}$, where $\Sor_x $ is a fixed Suslin correctness-preserving ccc poset coded in the ground model.
      \item If $x\in L_C$, for a fixed $C_x\in\hat{\Iwf}_x$ and a $\Por\frestr C_x$-name $\Qnm_x$ for a poset,
            \[\Qnm_x^B=\left\{
                \begin{array}{ll}
                   \Qnm_x & \textrm{if $C_x\subseteq B$}\\
                   \mathds{1} & \textrm{otherwise.}
                \end{array}\right.\]
   \end{enumerate}
   In (ii) note that, if $B\in\hat{\Iwf}_x$ contains $C_x$, the interpretation of $\Qnm_x^B$ in $V^{\Por\upharpoonright B}$ is the same poset as $\Qnm_x$ interpreted in $V^{\Por\upharpoonright C_x}$ (which is not required to be ccc). Therefore, by Lemma \ref{lemmacorrpres} and other direct calculations, the properties stated in (1) of Theorem \ref{TempIt} hold, so the template iteration can be defined as stated in that Theorem.
\end{example}

The following result states sufficient conditions for regular contention between two template iterations. Although it is stated in a general way, we only use a particular case (Corollary \ref{EmbCor}) for our applications.

\begin{theorem}\label{EmbThm}
   Let $L$ be a linear order, $\bar{\Iwf}$ and $\bar{\Jwf}$ indexed templates on $L$ such that $\Iwf_x\subseteq\Jwf_x$ for all $x\in L$. Consider two template iterations $\Por\frestr\langle L,\bar{\Iwf}\rangle$
   and $\check{\Por}\frestr\langle L,\bar{\Jwf}\rangle$ with the following properties.
   \begin{enumerate}[(1)]
      \item For $x\in L$ and $B\in\hat{\Iwf}_x$, if $\Por\frestr B\lessdot\check{\Por}\frestr B$, then $\Vdash_{\check{\Por}\upharpoonright B}\Qnm_x^B\lessdot_{V^{\Por\upharpoonright B}}\dot{\check{\Qor}}_x^B$.
      \item Whenever $B\in\hat{\Iwf}_x$, $A\subseteq B$ and $\langle\Por\frestr A,\check{\Por}\frestr A,\Por\frestr B,\check{\Por}\frestr B\rangle$ is a correct diagram, then the diagram $\langle\Por\frestr A\ast\Qnm_x^A,\check{\Por}\frestr A\ast\dot{\check{\Qor}}_x^A,\Por\frestr B\ast\Qnm_x^B,\check{\Por}\frestr B\ast\dot{\check{\Qor}}_x^B
          \rangle$ is correct.
      \item For $B\subseteq L$, $x\in B$, if $C\in\Jwf_x\frestr B$ and $p\in\check{\Por}\frestr C$, then there exists an $A\in\Iwf_x\frestr B$ such that $p\in\check{\Por}\frestr A$.
      \item For $B\subseteq L$, $x\in B$, if $C\in\Jwf_x\frestr B$ and $\dot{q}$ is a $\check{\Por}\frestr C$-name for a condition in $\dot{\check{\Qor}}_x^C$, then there exists an $A\in\Iwf_x\frestr B$ such that $\dot{q}$ is a $\check{\Por}\frestr A$-name for a condition in $\dot{\check{\Qor}}_x^A$.
   \end{enumerate}
   Then, for each $B\subseteq L$,
   \begin{enumerate}[(a)]
      \item $\Por\frestr B\lessdot\check{\Por}\frestr B$ and
      \item if $A\subseteq B$, then $\langle\Por\frestr A,\check{\Por}\frestr A,\Por\frestr B,\check{\Por}\frestr B\rangle$ is correct.
   \end{enumerate}
\end{theorem}
\begin{proof}
   Proceed by induction on $\mathrm{Dp}^{\bar{\Iwf}}(B)$. The non-trivial case is when $B\neq\varnothing$. According to Theorem \ref{TempIt}, consider the following cases.
   \begin{enumerate}[(i)]
      \item \emph{Case $x=\max(B)$ and $B_x=B\cap L_x\in\hat{\Iwf}_x$.} Then, $\Por\frestr B=\Por\frestr B_x\ast\Qnm_x^{B_x}$ and $\check{\Por}\frestr B=\check{\Por}\frestr B_x\ast\dot{\check{\Qor}}_x^{B_x}$ so, by induction hypothesis, Lemma \ref{2stepitemb} and (1), $\Por\frestr B\lessdot\check{\Por}\frestr B$. This gives (a).

          For (b), if $x\in A$, note that $A_x=A\cap L_x\in\hat{\Iwf}_x$. By induction hypothesis,
          $\langle\Por\frestr A_x,\check{\Por}\frestr A_x,\Por\frestr B_x,\check{\Por}\frestr B_x\rangle$
          is a correct diagram, so $\langle\Por\frestr A,\check{\Por}\frestr A,\Por\frestr B,\check{\Por}\frestr B\rangle$ is correct by (2). The conclusion is simpler when $x\notin A$ (Lemma \ref{lemmacorrpres} is used here).
      \item \emph{Case $x=\max(B)$ and $B_x\notin\hat{\Iwf}_x$.} Then, with $\Bwf:=\left\{B'\subseteq B\ /\ B'\cap L_x\in\Iwf_x\frestr B\right\}$, $\Por\frestr B=\limdir_{B'\in\Bwf}\Por\frestr B'$.
          If $B'\subseteq B''$ are in $\Bwf$ then, by induction hypothesis, $\langle\Por\frestr B',\check{\Por}\frestr B',\Por\frestr B'',\check{\Por}\frestr B''\rangle$ is correct. By Lemma \ref{dirlimEmb}, it is enough to prove that $\check{\Por}\frestr B=\limdir_{B'\in\Bwf}\check{\Por}\frestr B'$ to get $\Por\frestr B\lessdot\check{\Por}\frestr B$. If $p\in\check{\Por}\frestr B$ then, in the case that $x=\max(\dom(p))$,
          there exists an $A'\in\Jwf_x\frestr B$ such that $p\frestr L_x\in\check{\Por}\frestr A'$ and $p(x)$
          is a $\check{\Por}\frestr A'$-name for a condition in $\dot{\check{\Qor}}_x^{A'}$. By (3) and (4), we can find $C\in\Iwf_x\frestr B$ such that $p\frestr L_x\in\check{\Por}\frestr C$ and $p(x)$
          is a $\check{\Por}\frestr C$-name for a condition in $\dot{\check{\Qor}}_x^{C}$, so $p\in\check{\Por}\frestr(C\cup\{x\})$ with $C\cup\{x\}\in\Bwf$. The case $\max(\dom(p))<x$ is treated in a similar way.

          For (b), let $A\subseteq B$ and $p\in\Por\frestr A$ a reduction of $q\in\Por\frestr B$. We prove that $p$ is a reduction of $q$ with respect to $\check{\Por}\frestr A,\check{\Por}\frestr B$. Find $B'\in\Bwf$ such that $p,q\in\Por\frestr B'$. Put $A'=A\cap B'$, so $p\in\Por\frestr A'$. It is easy to notice that $p$ is a reduction of $q$ with respect to $\Por\frestr A',\Por\frestr B'$ so, by induction hypothesis, $p$ is a reduction of $q$ with respect to $\check{\Por}\frestr A',\check{\Por}\frestr B'$. As $\langle\check{\Por}\frestr A',\check{\Por}\frestr A,\check{\Por}\frestr B',\check{\Por}\frestr B\rangle$ is correct, our claim is proved.
      \item \emph{Case $B$ does not have a maximum element.} Then, $\Por\frestr B=\limdir_{B'\in\Bwf'}\Por\frestr B'$ where $\Bwf':=\{B'\subseteq B\ /$ $\exists_{x\in B}(B'\in\Iwf_x\frestr B)\}$. Like in the previous case, (3) and (4) imply that $\check{\Por}\frestr B=\limdir_{B'\in\Bwf}\check{\Por}\frestr B'$. Then, by Lemma \ref{dirlimEmb}, $\Por\frestr B\lessdot\check{\Por}\frestr B$. The argument for (b) is very similar to the one of the previous case.
   \end{enumerate}
\end{proof}

For our applications, we are interested in template iterations that produce ccc posets. The following result presents some sufficient conditions for this.

\begin{lemma}\label{templateitccc}
   Consider a template iteration $\Por\frestr\langle L,\bar{\Iwf}\rangle$ with the following properties for all $x\in L$ and $B\in\hat{\Iwf}_x$:
   \begin{enumerate}[(i)]
      \item There are $\Por\frestr B$-names $\langle\dot{Q}^B_{n,x}\rangle_{n<\omega}$
            witnessing that $\Qnm^B_x$ is $\sigma$-linked.
      \item If $D\subseteq B$ then $\Vdash_{\Por\upharpoonright B}\dot{Q}^D_{n,x}\subseteq\dot{Q}^B_{n,x}$ for all $n<\omega$.
   \end{enumerate}
   Then $\Por\frestr L$ has the Knaster condition.
\end{lemma}
\begin{proof}
   The idea is the same as in the proof of \cite[Lemma 2.3]{brendle2}. By induction on $\mathrm{Dp}(A)$ with $A\subseteq L$ it is easy to prove that, for any $p\in\Por\frestr A$, there is a stronger condition $q\in\Por\frestr A$ and a function $f_q:\dom q\to\omega$ such that, for any $x\in\dom q$, there is a $B\in\Iwf_x\frestr A$ such that $q\frestr L_x\in\Por\frestr B$ and $q\frestr L_x\Vdash q(x)\in\dot{Q}^B_{f_q(x),x}$.

  Fix $A\subseteq L$. We prove that, whenever $p,q\in\Por\frestr A$ are as above and $f_p$ and $f_q$ are compatible functions, then $p$ and $q$ are compatible conditions. Enumerate $\dom p\cup\dom q=\{x_k\ /\ k<m\}$ in increasing order. Construct conditions $r_k$ and sets $B_k$ for $k\leq m$ such that
  \begin{itemize}
     \item $B_k\in\Iwf_{x_{k}}\frestr B_{k+1}$ for $k<m$, $B_m=A$,
     \item $\dom r_k=\{x_j\ /\ j<k\}\subseteq B_k$ and $p\frestr L_{x_k},q\frestr L_{x_k},r_k\in\Por\frestr B_k$, (for $k=m$ there is no $x_m$, so use $p\frestr L_{x_m}=p$, likewise for $q$),
     \item $r_k\leq p\frestr L_{x_k}, q\frestr L_{x_k}$ and
     \item for all $k<m$, $r_{k+1}\frestr L_{x_k}=r_k$, $r_k$ forces, in $\Por\frestr B_k$, that $r_{k+1}(x_k)$ is stronger than both $p(x_k)$ and $q(x_k)$ (or only one of these if the other is not defined). Also, $p\frestr L_{x_k}$ forces $p(x_k)\in\dot{Q}^{B_k}_{f_p(x_k),x_k}$ and $q\frestr L_{x_k}$ forces $q(x_k)\in\dot{Q}^{B_k}_{f_q(x_k),x_k}$ (again, ignore undefined cases).
  \end{itemize}
  $\langle B_k\rangle_{k\leq m}$ is constructed by regressive recursion on $k\leq m$ such that $p\frestr L_{x_k}\in\Por\frestr B_k$ forces $p(x_k)\in\dot{Q}^{B_k}_{f_p(x_k),x_k}$ and $q\frestr L_{x_k}\in\Por\frestr B_k$ forces $q(x_k)\in\dot{Q}^{B_k}_{f_q(x_k),x_k}$ for $k<m$ (when, for example, $q(x_k)$ is not defined, just ensure that $q\frestr L_{x_k}\in\Por\frestr B_k$). Construct $r_k$ by recursion on $k\leq m$.
  Put $r_0=\varnothing$. Assume we have constructed $r_k$ ($k<m$). If $x_k\in\dom p\menos\dom q$, put $r_{k+1}=r_k\widehat{\ \ }\langle p(x_k)\rangle_{x_k}$; if $x_k\in\dom q\menos\dom p$, put $r_{k+1}=r_k\widehat{\ \ }\langle q(x_k)\rangle_{x_k}$; otherwise, if $x_k\in\dom p\cap\dom q$, $p\frestr L_{x_k}, q\frestr L_{x_k}, r_k\in\Por\frestr B_k$, $p\frestr L_{x_k}$ forces $p(x)\in\dot{Q}^{B_k}_{n_k,x_k}$ and $q\frestr L_{x_k}$ forces $q(x)\in\dot{Q}^{B_k}_{n_k,x_k}$ where $n_k=f_p(x_k)=f_q(x_k)$. As $r_k\leq p\frestr L_{x_k},q\frestr L_{x_k}$, it forces that $p(x_k)$ and $q(x_k)$ are compatible in $\Qnm^B_{x_k}$, so let $r_{k+1}(x_k)$ be a $\Por\frestr B_k$-name of a common stronger condition.

  A typical delta-system argument and the previous facts imply that $\Por\frestr A$ has the Knaster condition.
\end{proof}

If the template $\langle L,\overline{\Iwf}\rangle$ is as in Example \ref{ExmpTemp}(2), to obtain that $\Por\frestr L$ has the ccc it is enough to assume that $\Vdash_{\Por\upharpoonright B}``\Qnm^B_x \textrm{\ has the ccc''}$ for any $x\in L$ and $B\in\hat{\Iwf}_x$. The reason of this, as explained in Example \ref{ExmpFSI}, is that $\Por\frestr X$ is a fsi for any $X\subseteq L$.

The following result is a generalization of \cite[Lemma 2.4]{brendle2}.

\begin{lemma}\label{CondSupp}
  Fix $\theta$ a cardinal with uncountable cofinality. Consider a template iteration $\Por\frestr\langle L,\bar{\Iwf}\rangle$ defined as in Example \ref{ExpTempItFund} where,
  \begin{itemize}
     \item for $x\in L_S$, $\Sor_x$ is a Suslin $\sigma$-linked correctness-preserving forcing notion and
     \item for $x\in L_C$, $\Qnm_x$ is a $\Por\frestr C_x$-name for a $\sigma$-linked poset whose conditions are reals\footnote{These reals belong to some fixed uncountable Polish space $R_x$ coded in the ground model.} and $|C_x|<\theta$.
  \end{itemize}
  Then, for each $A\subseteq L$,
  \begin{enumerate}[(a)]
     \item $\Por\frestr A$ has the Knaster condition,
     \item if $p\in\Por\frestr A$ there exists a $C\subseteq A$ of size $<\theta$ such that $p\in\Por\frestr C$, and
     \item if $\dot{x}$ is a $\Por\frestr A$-name for a real, then there exists a $C'\subseteq A$ of size $<\theta$ such that $\dot{x}$ is a $\Por\frestr C'$-name.
  \end{enumerate}
\end{lemma}
\begin{proof}
   (a) follows from Lemma \ref{templateitccc}. We prove (b) and (c) simultaneously by induction on $\mathrm{Dp}(A)$. Let $p\in\Por\frestr A$ and $x=\max(\dom p)$, so there exists a $B\in\Iwf_x\frestr A$ such that $p\frestr L_x\in\Por\frestr B$ and $p(x)$ is a $\Por\frestr B$-name for a condition in $\Qnm_x^B$. By induction hypothesis, there exists $D\subseteq B$ of size $<\theta$ such that $p\frestr L_x\in\Por\frestr D$ and $p(x)$ is a $\Por\frestr D$-name for a real. If $x\in L_S$, then clearly $p(x)$ is a name for a condition in $\Qnm^{D}_x=\Snm^{V^{\Por\upharpoonright D}}_x$, so $p\in\Por\frestr(D\cup\{x\})$. When $x\in L_C$, if $C_x\not\subseteq B$ then $p(x)$ will be the trivial condition and $p\in\Por\frestr(D\cup\{x\})$. Else, if $C_x\subseteq B$, we may assume $C_x\subseteq D$, so $p(x)$ is a $\Por\frestr D$-name for a condition in $\Qnm_x^D=\Qnm_x$. Then, $p\in\Por\frestr(D\cup\{x\})$.

   Now, if $\dot{x}$ is a $\Por\frestr A$-name for a real, note that it can be determined by countably many conditions $\langle r_n\rangle_{n<\omega}$ in $\Por\frestr A$. As each $r_n\in\Por\frestr E_n$ for some $E_n\subseteq A$ of size $<\theta$ and $\theta$ has uncountable cofinality, then $C':=\bigcup_{n<\omega}E_n\subseteq A$ has size $<\theta$ and $r_n\in\Por\frestr C'$ for all $n<\omega$. This implies that $\dot{x}$ is a $\Por\frestr C'$-name.
\end{proof}

In fact, \cite[Lemma 2.4]{brendle2} corresponds to the case $L_C=\varnothing$, so $C$ and $C'$ can be found countable in there. But this cannot be guaranteed in the presence of non-definable posets that come from $L_C\neq\varnothing$.

The following is a consequence of Theorem \ref{EmbThm} that fits better for our applications. Although this type of results was considered originally to get only forcing equivalence, we need to extend to cases where we can get regular contention, fact that is needed in order to deal with the limit steps of small cofinality in the proof of Theorem \ref{AppSplitting}.

\begin{corollary}[Regular contention between template iterations, particular case]\label{EmbCor}
   Let $\theta$ be a cardinal with uncountable cofinality, $L$ a linear order, $\bar{\Iwf}$ and $\bar{\Jwf}$ templates on $L$ such that $\langle L,\bar{\Jwf}\rangle$ is a $\theta$-innocuous extension of $\langle L,\bar{\Iwf}\rangle$. Consider two template iterations $\Por\frestr\langle L,\bar{\Iwf}\rangle$
   and $\check{\Por}\frestr\langle L,\bar{\Jwf}\rangle$ defined with the conditions of Lemma \ref{CondSupp}, such that
   \begin{enumerate}[(1')]
   \setcounter{enumi}{-1}
      \item The same $L_S$ and $L_C$ are considered for both iterations.
      \item For $x\in L_S$, the same Suslin forcing
            $\Sor_x$ is considered for both template iterations.
      \item For $x\in L_C$ either $\check{C}_x=C_x$ and $\dot{\check{\Qor}}_x=\Qnm_x$, or $C_x=\varnothing$ and $\Qnm_x$ is the trivial forcing.
   \end{enumerate}
   Then, the following hold for each $B\subseteq L$.
   \begin{enumerate}[(a)]
      \item $\Por\frestr B\lessdot\check{\Por}\frestr B$.
      \item If $A\subseteq B$, then $\langle\Por\frestr A,\check{\Por}\frestr A,\Por\frestr B,\check{\Por}\frestr B\rangle$ is a correct diagram.
   \end{enumerate}
\end{corollary}
\begin{proof}
   It is enough to prove conditions (1)-(4) of Theorem \ref{EmbThm}.
   \begin{enumerate}[(1)]
      \item Straightforward from (0'), (1') and (2').
      \item For $x\in L_S$, the result follows because $\Sor_x$ is a correctness-preserving Suslin ccc notion. For $x\in L_C$, it follows from (2') and Lemma \ref{lemmacorrpres}.
      \item Let $B\subseteq L$, $x\in B$, $C\in\Jwf_x\frestr B$ and $p\in\check{\Por}\frestr C$. By Lemma \ref{CondSupp}, there exists a $K\subseteq C$ such that $p\in\check{\Por}\frestr K$ and $|K|<\theta$. Then, by $\theta$-innocuity, there exists an $H\in\Iwf_x$ such that $K\subseteq H$, so $K\subseteq A$ and $p\in\check{\Por}\frestr A$ where $A:=B\cap H\in\Iwf_x\frestr B$.
      \item Let $B\subseteq L$, $x\in B$, $C\in\Jwf_x\frestr B$ and $\dot{q}$ a $\check{\Por}\frestr C$-name for a condition in $\dot{\check{\Qor}}_x^C$. A similar argument as before works using Lemma \ref{CondSupp}. We show this for $x\in L_C$ (the case for $x\in L_S$ is even simpler). If $\check{C}_x\subseteq C$, find $K\subseteq C$ such that $\dot{q}$ is a $\check{\Por}\frestr K$-name for a real, $|K|<\theta$ and $\check{C}_x\subseteq K$. Hence, $\dot{q}$ is a $\check{\Por}\frestr K$-name for a condition in $\dot{\check{\Qor}}_x$. On the other hand, by $\theta$-innocuity, there is an $A\in\Iwf_x\frestr B$ containing $K$ and $\dot{q}$ is a $\check{\Por}\frestr A$-name for a condition in $\dot{\check{\Qor}}_x$. The case $\check{C_x}\nsubseteq C$ is simpler because $\dot{q}$ is a $\check{\Por}\frestr C$-name for the trivial condition.
   \end{enumerate}
\end{proof}

We conclude this section with a generalization of \cite[Lemma 1.7]{br}.

\begin{lemma}\label{InnEqv}
   Assume that $\langle L,\bar{\Jwf}\rangle$ is a $\theta$-innocuous extension of $\langle L,\bar{\Iwf}\rangle$. Consider $\Por\frestr\langle L,\bar{\Iwf}\rangle$ and $\check{\Por}\frestr\langle L,\bar{\Jwf}\rangle$ template iterations satisfying the hypotheses of Corollary \ref{EmbCor}, but in (2') always assume that $\check{C}_x=C_x$ and $\dot{\check{\Qor}}_x=\Qnm_x$.
   Then, there exists a dense embedding $F:\check{\Por}\frestr L\to\Por\frestr L$.
\end{lemma}
\begin{proof}
   Proceed like in the proof of \cite[Lemma 1.7]{br}. By recursion on $\mathrm{Dp}^{\bar{\Jwf}}(B)$ for $B\subseteq L$, construct
   $F_B:\check{\Por}\frestr B\to\Por\frestr B$ such that
   \begin{enumerate}[(1)]
      \item $F_B$ is a dense embedding and
      \item $F_{B'}\subseteq F_B$ whenever $B'\subseteq B$.
   \end{enumerate}
   Let $p\in\check{\Por}\frestr B$. If $p=\varnothing$, put $F_B(\varnothing)=\varnothing$, so assume that $p\neq\varnothing$. Let $x:=\max(\dom p)$ and find $\bar{B}\in \Jwf_x\frestr B$ such that $p\frestr L_x\in\check{\Por}\frestr\bar{B}$ and $p(x)$ is a $\check{\Por}\frestr\bar{B}$-name for a condition in $\dot{\check{\Qor}}^{\bar{B}}_x$. Consider the following cases.
   \begin{enumerate}[(i)]
      \item \emph{$x\in L_S$}. By hypothesis, there exists an $\bar{A}\subseteq\bar{B}$ of size $<\theta$ such that $p\frestr L_x\in\check{\Por}\frestr\bar{A}$ and $p(x)$ is a $\check{\Por}\frestr\bar{A}$-name for a condition in $\Sor_x^{V^{\check{\Por}\upharpoonright\bar{A}}}$. By innocuity, there exists a $\bar{C}\in\Iwf_x\frestr B\subseteq\Jwf_x\frestr B$
          containing $\bar{A}$, so $p\frestr L_x\in\check{\Por}\frestr\bar{C}$ and $p(x)$ is a $\check{\Por}\frestr\bar{C}$-name for a condition in $\Sor_x^{V^{\check{\Por}\upharpoonright\bar{C}}}$. As $\mathrm{Dp}^{\bar{\Jwf}}(\bar{C})<\mathrm{Dp}^{\bar{\Jwf}}(B)$, the function $F_{\bar{C}}$ has already been defined. So let
          $F_B(p):=F_{\bar{C}}(p\frestr L_x)\widehat{\ \ }\langle p_0(x)\rangle_x$ where $p_0(x)$ is the $\Por\frestr\bar{C}$-name associated to $p(x)$ with respect to the function $F_{\bar{C}}$. Notice that, because of (2), $F_B(p)$ does not depend on the choice of $\bar{C}$.
      \item \emph{$x\in L_C$ and $C_x\subseteq\bar{B}$}, so $\dot{\check{\Qor}}^{\bar{B}}_x=\Qnm_x$. Proceed like before, but take $\bar{A}$ such that $C_x\subseteq \bar{A}$.
      \item \emph{$x\in L_C$ but $C_x\not\subseteq\bar{B}$}, so $\dot{\check{\Qor}}^{\bar{B}}_x=\mathds{1}$, that is, $p(x)$ is forced to be the trivial condition.
          Proceed as in\footnote{Here, $F_B(p)=F_{\bar{B}}(p\frestr L_x)$ would be ok, but proceeding as in (i) guarantees that $\dom F_B(p)=\dom p$.} (i).
   \end{enumerate}
\end{proof}


\section{Preservation theorems}\label{SecPres}

The main goal of this section is to prove preservation results for template iterations associated to some cardinal invariants. These preservation properties were developed in the context of fsi of ccc posets by Judah and Shelah \cite{jushe}, with improvements by Brendle \cite{brendle}. These are summarized and generalized in \cite{gold} and in \cite[Sect. 6.4 and 6.5]{barju}. The exposition in \cite{mejia,mejia02} is very close to the presentation of this section.

\begin{context}\label{ContextUnbd}
 Fix an increasing sequence $\langle\sqsubset_n\rangle_{n<\omega}$ of 2-place closed relations (in the topological sense) in $\omega^\omega$ such that, for any $n<\omega$ and $g\in\omega^\omega$, $(\sqsubset_n)^g=\left\{f\in\omega^\omega\ /\ f\sqsubset_n g\right\}$ is (closed) nwd (nowhere dense).

 Put $\sqsubset=\bigcup_{n<\omega}\sqsubset_n$. Therefore, for every $g\in\omega^\omega$, $(\sqsubset)^g$ is an $F_\sigma$ meager set.

 For $f,g\in\omega^\omega$, say that \emph{$g$ $\sqsubset$-dominates $f$} if $f\sqsubset g$. $F\subseteq\omega^\omega$ is a \emph{$\sqsubset$-unbounded family} if no function in $\omega^\omega$ $\sqsubset$-dominates all the members of $F$. Associate with this notion the cardinal $\bfrak_\sqsubset$, which is the least size of a $\sqsubset$-unbounded family. Dually, say that $C\subseteq\omega^\omega$ is a \emph{$\sqsubset$-dominating family} if any real in $\omega^\omega$ is $\sqsubset$-dominated by some member of $C$. The cardinal $\dfrak_\sqsubset$ is the least size of a $\sqsubset$-dominating family.

 Given a set $Y$, say that a real $f\in\omega^\omega$ is \emph{$\sqsubset$-unbounded over $Y$} if $f\not\sqsubset g$ for every $g\in Y\cap\omega^\omega$. This is denoted by $f\not\sqsubset Y$.
\end{context}

Although we define Context \ref{ContextUnbd} for $\omega^\omega$, we can use, in general, the same notion by changing the space for the domain or the codomain of $\sqsubset$ to another uncountable Polish space, like $2^\omega$ or other spaces whose members can be coded by reals in $\omega^\omega$.

For all the notions, results and discussions in this section, fix $\theta$ an uncountable regular cardinal.

\begin{definition}[Judah and Shelah {\cite{jushe}}, {\cite[Def. 6.4.4]{barju}}]\label{DefPresProp}
   A forcing notion $\Por$ is \emph{$\theta$-$\sqsubset$-good} if the following property holds\footnote{\cite[Def. 6.4.4]{barju} has a different formulation, which is equivalent to our formulation for $\theta$-cc posets (recall that $\theta$ is uncountable regular). See \cite[Lemma 2]{mejia} for details.}: For any $\Por$-name $\dot{h}$ for a real in $\omega^\omega$, there exists a nonempty $Y\subseteq\omega^\omega$ (in the ground model) of size $<\theta$ such that, for any $f\in\omega^\omega$, if $f\not\sqsubset Y$ then $\Vdash f\not\sqsubset\dot{h}$.

   Say that $\Por$ is \emph{$\sqsubset$-good} if it is $\aleph_1$-$\sqsubset$-good.
\end{definition}

This is a standard property associated to the preservation of $\bfrak_\sqsubset\leq\theta$ and the preservation of $\dfrak_\sqsubset$ large through forcing extensions that have the property. To explain this, first say that $F\subseteq\omega^\omega$ is \emph{$\theta$-$\sqsubset$-unbounded} if, for any $X\subseteq\omega^\omega$ of size $<\theta$, there exists an $f\in F$ which is $\sqsubset$-unbounded over $X$. It is clear that, if $F$ is such a family, then $\bfrak_\sqsubset\leq|F|$ and $\theta\leq\dfrak_\sqsubset$. In practice, $F$ has size $\theta$, so $\bfrak_\sqsubset\leq\theta$. On the other hand, $\theta$-$\sqsubset$-good posets preserve, in any generic extension, $\theta$-$\sqsubset$-unbounded families of the ground model and, if $\lambda\geq\theta$ is a cardinal and $\dfrak_\sqsubset\geq\lambda$ in the ground model, then this inequality is also preserved in any generic extension. It is also known that the property of Definition \ref{DefPresProp} is preserved under fsi of $\theta$-cc posets. Also, for posets $\Por\lessdot\Qor$, if $\Qor$ is $\theta$-$\sqsubset$-good, then so is $\Por$.

We prove in this section that this property, under some conditions, is preserved through template iterations. Before that, we present some examples.

\begin{lemma}[{\cite[Lemma 4]{mejia}}]\label{smallPlus}
   Any poset of size $<\theta$ is $\theta$-$\sqsubset$-good. In particular, $\Cor$ is $\sqsubset$-good.
\end{lemma}

\begin{example}\label{SubsecUnbd}
 \begin{enumerate}[(1)]
  \item \emph{Preserving non-meager sets:} For $f,g\in\omega^\omega$ and $n<\omega$, define $f\eqcirc_n g$ iff $\forall_{k\geq n}(f(k)\neq g(k))$, so $f\eqcirc g$ iff $f$ and $g$ are eventually different, that is, $\forall^\infty_{k<\omega}(f(k)\neq g(k))$. Theorem \ref{Catchar} implies that $\bfrak_\eqcirc=\non(\Mwf)$ and $\dfrak_\eqcirc=\cov(\Mwf)$.

  \item \emph{Preserving unbounded families:} For $f,g\in\omega^\omega$, define $f\leq^*_n g$ iff $\forall_{k\geq n}(f(k)\leq g(k))$, so $f\leq^*g$ iff $\forall^\infty_{k\in\omega}(f(k)\leq g(k))$. Clearly, $\bfrak=\bfrak_{\leq^*}$ and $\dfrak=\dfrak_{\leq^*}$. Miller \cite{miller} proved that $\Eor$ is $\leq^*$-good. $\Bor$ is also $\leq^*$-good because it is $\omega^\omega$-bounding.

  \item \emph{Preserving splitting families:} For $A,B\in[\omega]^\omega$, define $A\propto_n B\sii(B\menos n\subseteq A\textrm{\ or }B\menos n\subseteq\omega\menos A)$, so $A\propto B\sii(B\subseteq^* A\textrm{\ or }B\subseteq^*\omega\menos A)$. Note also that $A\not\propto B$ iff $A$ splits $B$, so $\sfrak=\bfrak_\propto$ and $\rfrak=\dfrak_\propto$. Baumgartner and Dordal \cite{baudor} proved that $\Dor$ is $\propto$-good (see also \cite[Main Lemma 3.8]{brendlebog}).

  \item \emph{Preserving finitely splitting families:} For $a\in[\omega]^\omega$ and an interval partition $\bar{J}=\langle J_n\rangle_{n<\omega}$ of $\omega$, define $a\vartriangleright_n\bar{J}\sii(\forall_{k\geq n}(J_k\nsubseteq a)\textrm{\ or }\forall_{k\geq n}(J_k\nsubseteq\omega\menos a))$, so $a\vartriangleright\bar{J}$ iff $(\forall_{k<\omega}^\infty(J_k\nsubseteq a)\textrm{\ or }\forall^\infty_{k<\omega}(J_k\nsubseteq\omega\menos a))$. $a\not\vartriangleright\bar{J}$ is read \emph{$a$ splits $\bar{J}$}. It is proved in \cite{KaWe-Spl} that $\bfrak_\vartriangleright=\max\{\bfrak,\sfrak\}$ and $\dfrak_\vartriangleright=\min\{\dfrak,\rfrak\}$.

  The author proved in \cite[Lemma 2.20]{mejia02} that any $\leq^*$-good poset is $\vartriangleright$-good. In particular, $\Bor$ and $\Eor$ are $\vartriangleright$-good.

  \item \emph{Preserving null-covering families:} Let $\langle I_n\rangle_{n<\omega}$ be the interval partition of $\omega$ such that $\forall_{n<\omega}(|I_n|=2^{n+1})$ . For $f,g\in2^\omega$ define $f\pitchfork_ng\sii\forall_{k\geq n}(f\frestr I_k\neq g\frestr I_k)$, so $f\pitchfork g\sii \forall^\infty_{k<\omega}(f\frestr I_k\neq g\frestr I_k)$. Clearly, $(\pitchfork)^g$ is a co-null $F_\sigma$ meager set. This relation is related to the covering-uniformity of measure because $\cov(\Nwf)\leq\bfrak_\pitchfork\leq\non(\Mwf)$ and $\cov(\Mwf)\leq\dfrak_\pitchfork\leq\non(\Nwf)$ (see \cite[Lemma 7]{mejia}).

  It is known in \cite[Lemma $1^*$]{brendle} that, given an infinite cardinal $\nu<\theta$, every $\nu$-centered forcing notion is $\theta$-$\pitchfork$-good.

  \item \emph{Preserving ``union of non-null sets is non-null'':} Fix $\Hwf:=\{id^{k+1}\ /\ k<\omega\}$ (where $id^{k+1}(i)=i^{k+1}$) and let $S(\omega,\Hwf):=\bigcup_{h\in\Hwf}S(\omega,h)$. For $x\in\omega^\omega$ and a slalom $\psi\in S(\omega,\Hwf)$, put $x\in^*_n\psi$ iff $\forall_{k\geq n}(x(k)\in\psi(k))$, so $x\in^*\psi$ iff $\forall^\infty_{k<\omega}(x(k)\in\psi(k))$. By Bartoszy\'{n}ski characterization (Theorem \ref{BartChar}) applied to $id$ and to a function $g$ that dominates all the functions in $\Hwf$, $\add(\Nwf)=\bfrak_{\in^*}$ and $\cof(\Nwf)=\dfrak_{\in^*}$.

  Judah and Shelah \cite{jushe} proved that, given an infinite cardinal $\nu<\theta$, every $\nu$-centered forcing notion is $\theta$-$\in^*$-good. Moreover, as a consequence of results of Kamburelis \cite{kamburelis}, any subalgebra\footnote{Here, $\Bor$ is seen as the complete Boolean algebra of Borel sets (in $2^\omega$) modulo the null ideal.} of $\Bor$ is $\in^*$-good.
 \end{enumerate}
\end{example}

\begin{example}[Preserving new reals]\label{SubsecNewreal}
   For $f,g\in\omega^\omega$ define $f=^*_n g$ as $\forall_{k\geq n}(f(k)=g(k))$, so $f=^*g\sii\forall^\infty_{k<\omega}(f(k)=g(k))$. Note that, if $M$ is a model of $\thzfc$ and $c$ is a real, then $c$ is $=^*$-unbounded over $M$ iff $c\notin M$. It is also easy to see that $b_{=^*}=2$ and $d_{=^*}=\cfrak$. Here, we are not interested in the cardinal invariants but in the ``preservation'' of new reals that are added at certain stage of an iteration and that cannot be added at other different stages. Concretely, we use this relation to prove Theorem \ref{newrealint}.
\end{example}

\begin{lemma}\label{Presnewrealccc}
   Any $\theta$-cc poset is $\theta$-$=^*$-good. In particular, ccc posets are $=^*$-good.
\end{lemma}
\begin{proof}
   Let $\dot{h}$ be a $\Por$-name for a real. Find a maximal antichain $A\subseteq\Por$ such that, for $p\in A$, either $p\Vdash``\dot{h}\notin V"$ or there is a real $f_p$ such that $p\Vdash\dot{h}=f_p$. Clearly, $Y:=\left\{f_p\ /\ p\in A\right\}$ (we include only those that exist) has size $<\theta$ and it is a witness of goodness for $\dot{h}$.
\end{proof}

To prove the preservation theorems, we generalize preservation of goodness in limit steps (though this proof is not that different from the fsi case).

\begin{lemma}\label{PresGoodLimdir}
   Let $I$ be a directed partial order, $\langle\Por_i\rangle_{i\in I}$ a directed system of posets and $\Por=\limdir_{i\in I}\Por_i$. If $|I|<\theta$ and $\Por_i$ is $\theta$-$\sqsubset$-good for any $i\in I$, then $\Por$ is $\theta$-$\sqsubset$-good.
\end{lemma}
\begin{proof}
   Let $\dot{h}$ be a $\Por$-name for a real in $\omega^\omega$. For $i\in I$, find a $\Por_i$-name for a real $\dot{h}_i$ and a sequence $\{\dot{p}^i_m\}_{m<\omega}$ of $\Por_i$-names that represents a decreasing sequence of conditions in $\Por/\Por_i$ such that $\Por_i$ forces that $\dot{p}^i_m\Vdash_{\Por/\Por_i}\dot{h}\frestr m=\dot{h}_i\frestr m$. For each $i\in I$ choose $Y_i\subseteq\omega^\omega$ of size $<\theta$ that witnesses goodness of $\Por_i$ for $\dot{h}_i$. As $|I|<\theta$, $Y=\bigcup_{i\in I}Y_i$ has size $<\theta$ by regularity of $\theta$.

   We prove that $Y$ witnesses goodness of $\Por$ for $\dot{h}$. Assume, towards a contradiction, that $f\in\omega^\omega$, $f\not\sqsubset Y$ and that there are $p\in\Por$ and $n<\omega$ such that $p\Vdash_\Por f\sqsubset_n\dot{h}$. Choose $i\in I$ such that $p\in\Por_i$. Let $G$ be $\Por_i$-generic over the ground model $V$ with $p\in G$. Then, by the choice of $Y_i$, $f\not\sqsubset h_i$, in particular, $f\not\sqsubset_n h_i$. As $C:=(\sqsubset_n)_f=\{g\in\omega^\omega\ /\ f\sqsubset_n g\}$ is closed, there is an $m<\omega$ such that $[h_i\frestr m]\cap C=\varnothing$. Thus, $p^i_m\Vdash_{\Por/\Por_i}[\dot{h}\frestr m]\cap C=\varnothing$, that is, $p^i_m\Vdash_{\Por/\Por_i}f\not\sqsubset_n\dot{h}$. On the other hand, $\Vdash_{\Por/\Por_i}f\sqsubset_n\dot{h}$ (because $p\Vdash_{\Por,V} f\sqsubset_n\dot{h}$), a contradiction.
\end{proof}

\begin{theorem}[First preservation theorem for template iterations]\label{PresTemp}
   Let $\Por\frestr\langle L,\bar{\Iwf}\rangle$ be a template iteration such that $\Por\frestr L$ is $\theta$-cc. Assume that $\nu\leq\theta$ is an uncountable cardinal such that
   \begin{enumerate}[(i)]
      \item for all $B\in[L]^{<\nu}$, $\Iwf(B)$ has size $<\nu$,
      \item for all $A\subseteq L$, if $p\in\Por\frestr A$ and $\dot{h}$ is a $\Por\frestr A$-name for a real, then there is a $C\subseteq A$ of size $<\nu$ such that $p\in\Por\frestr C$ and $\dot{h}$ is a $\Por\frestr C$-name, and
      \item for all $x\in L$ and $B\in\hat{\Iwf}_x$, $\Vdash_{\Por\upharpoonright B}\Qnm^B_x$ is $\theta$-$\sqsubset$-good.
   \end{enumerate}
   Then, $\Por\frestr L$ is $\theta$-$\sqsubset$-good. Moreover, if $L'$ is an initial segment of $L$ such that $\forall_{x\in L\menos L'}(L'\in\hat{\Iwf}_x)$, then $\Por\frestr L'$ forces that $\Por\frestr L/\Por\frestr L'$ is $\theta$-$\sqsubset$-good.
\end{theorem}
\begin{proof}
   We prove, by induction on $\mathrm{Dp}(A)$ with $L'\subseteq A\subseteq L$, that $\Por\frestr L'$ forces that $\Por\frestr A/\Por\frestr L'$ is $\theta$-$\sqsubset$-good. We may assume $L'\subsetneq A$. Proceed by cases.
   \begin{enumerate}[(1)]
      \item \emph{$A$ has a maximum $x$ and $A_x=A\cap L_x\in\hat{\Iwf}_x$.} By Lemma \ref{2stepofQuot}, in $V':=V^{\Por\upharpoonright L'}$ (fix this notation for the rest of the proof), $\Por\frestr A/\Por\frestr L'$ is equivalent to $(\Por\frestr A_x/\Por\frestr L')\ast\Qnm^{A_x}_x$, so it is $\theta$-$\sqsubset$-good by (iii) and  induction hypothesis (recall that the two step iteration of $\theta$-cc $\theta$-$\sqsubset$-good posets is $\theta$-$\sqsubset$-good, which is easy to prove).
      \item \emph{$A$ has a maximum $x$ but $A_x\notin\hat{\Iwf}_x$.} Then, $\Por\frestr A=\limdir_{B\in\Awf}\Por\frestr B$ where $\Awf:=\{B\subseteq A\ /\ B\cap L_x\in\Iwf_x\frestr A\textrm{\ and }L'\subseteq B\}$. Let $\dot{h}\in V$ be a $\Por\frestr A$-name for a real. If there exists a $B\in\Awf$ such that $\dot{h}$ is a $\Por\frestr B$-name then, in $V'$, $\Por\frestr B/\Por\frestr L'$ is $\theta$-$\sqsubset$-good (by induction hypothesis) and any witness of goodness of $\Por\frestr B/\Por\frestr L'$ for $\dot{h}$ is also a witness of goodness of $\Por\frestr A/\Por\frestr L'$ because $\Por\frestr B/\Por\frestr L'\lessdot\Por\frestr A/\Por\frestr L'$ by Lemma \ref{QuotEmb}. So assume that $\dot{h}$ is not a $\Por\frestr B$-name for any $B\in\Awf$. By (ii), there is a $C\in[A\menos L']^{<\nu}$ such that $\dot{h}$ is a $\Por\frestr(L'\cup C)$-name with $x\in C$. As $L'\in\hat{\Iwf}_x$ note that
          \begin{multline*}
                \Cwf:=\left\{D\subseteq L'\cup C\ /\ L'\subseteq D\textrm{\ and\ }D\cap L_x\in\Iwf_x\frestr (L'\cup C)\right\}\\
                \subseteq\left\{B\cap(L'\cup C)\ /\ B\in\Awf\right\}\subseteq\left\{L'\cup E\ /\ E\subseteq C\textrm{\ and\ }E\cap L_x\in\Iwf_x\frestr C\right\}.
          \end{multline*}
          and $\Cwf$ is cofinal in the latter set. As $\mu:=|\Iwf(C)|<\nu$ (by (i)), $|\Cwf|\leq\mu$, so enumerate $\Cwf:=\left\{D_\alpha\ /\ \alpha<\mu\right\}$ where each $D_\alpha=B_\alpha\cap(L'\cup C)$ for some $B_\alpha\in\Awf$. Note also that $(L'\cup C)\cap L_x\notin\Cwf$ (if so, there exists a $B\in\Awf$ such that $L'\cup C\subseteq B$ and $\dot{h}$ would be a $\Por\frestr B$-name, which is false), so $\Por\frestr(L'\cup C)=\limdir_{\alpha<\mu}\Por\frestr D_\alpha$ and, by Lemma \ref{dirlimquot}, $\Por\frestr(L'\cup C)/\Por\frestr L'=\limdir_{\alpha<\mu}\Por\frestr D_\alpha/\Por\frestr L'$ in $V'$. By induction hypothesis, as $\Por\frestr D_\alpha/\Por\frestr L'\lessdot\Por\frestr B_\alpha/\Por\frestr L'$, both posets are $\theta$-$\sqsubset$-good for any $\alpha<\mu$. Therefore, by Lemma \ref{PresGoodLimdir}, $\Por\frestr(L'\cup C)/\Por\frestr L'$ is $\theta$-$\sqsubset$-good. Any family of reals that witnesses this goodness for $\dot{h}$ works for the goodness of $\Por\frestr A/\Por\frestr L'$ for $\dot{h}$.
      \item \emph{$A$ does not have a maximum element.} So $\Por\frestr A=\limdir_{B\in\Bwf}\Por\frestr B$ where $\Bwf:=\{B\in\Iwf_x\frestr A\ /\ x\in A\textrm{\ and }L'\subseteq B\}$. Let $\dot{h}$ a $\Por\frestr A$-name for a real. If there is no $B\in\Bwf$ such that $\dot{h}$ is a $\Por\frestr B$-name, find $C\subseteq A\menos L'$ of size $<\nu$ such that $\dot{h}$ is a $\Por\frestr(L'\cup C)$-name and, without loss of generality,  assume that $C$ doesn't have a maximum. Proceed exactly like in the previous case.
   \end{enumerate}
\end{proof}

\begin{remark}\label{ShTempMainThm}
   Shelah's model (\cite{shelah}, see also \cite{br}) for the consistency of $\dfrak<\afrak$ with $\thzfc$ uses a template iteration like in Example \ref{ExpTempItFund} where $L_C=\varnothing$ and $\Sor_x=\Dor$ for every $x\in L_S=L$. To use the isomorphism-of-names argument, the iteration is constructed under the continuum hypothesis so the conditions of Theorem \ref{PresTemp} with $\theta=\aleph_1$ and $\sqsubset=\propto$ are meet and, thus, $\sfrak=\aleph_1$ in the generic extension. Therefore, if $\aleph_1<\mu<\lambda$ are regular cardinals and $\lambda^\omega=\lambda$, there is a model of $\thzfc$ such that $\sfrak=\aleph_1<\bfrak=\dfrak=\mu<\afrak=\cfrak=\lambda$. Moreover, the same model satisfies $\cov(\Nwf)=\aleph_1$, $\add(\Mwf)=\cof(\Nwf)=\mu$ and $\non(\Nwf)=\lambda$.
\end{remark}

We now prove a preservation result of the same property but with different conditions.

\begin{theorem}[Second preservation theorem for template iterations]\label{PresTemp2}
   Let $\Por\frestr\langle L,\bar{\Iwf}\rangle$ be a template iteration such that $\Por\frestr L$ is $\theta$-cc and let $L'$ be an initial segment of $L$ such that $\forall_{x\in L\menos L'}(L'\in\hat{\Iwf}_x)$. Assume, for any $A\subseteq L$ with $L'\subsetneq A$:
   \begin{enumerate}[(i)]
      \item Whenever $A$ has a maximum $x$ and $A_x:=A\cap L_x\notin\hat{\Iwf}_x$, if $\dot{h}$ is a $\Por\frestr A$-name for a real, then there exists an increasing sequence $\langle B_n\rangle_{n<\omega}$ in $\Bwf_A:=\left\{B\subseteq A\ /\ B\cap L_x\in\Iwf_x\frestr A\right.$ and $\left.L'\subseteq B\right\}$ such that $\dot{h}$ is a $\Por\frestr C$-name for a real,
            where $C:=\bigcup_{n<\omega}B_n$, and $\Por\frestr C=\limdir_{n<\omega}\Por\frestr B_n$,
      \item Whenever $A$ does not have a maximum and $\dot{h}$ is a $\Por\frestr A$-name for a real, then
            there exists an increasing sequence $\langle B_n\rangle_{n<\omega}$ in $\Bwf_A:=\left\{B\in\Iwf_x\frestr A\ /\ x\in A\right.$ and $\left.L'\subseteq B\right\}$ like in (i).
      \item for all $x\in L$ and $B\in\hat{\Iwf}_x$, $\Vdash_{\Por\upharpoonright B}\Qnm^B_x$ is $\theta$-$\sqsubset$-good.
   \end{enumerate}
   Then, $\Por\frestr L'$ forces that $\Por\frestr L/\Por\frestr L'$ is $\theta$-$\sqsubset$-good.
\end{theorem}
\begin{proof}
   By induction on $\mathrm{Dp}(A)$ for $A\supseteq L'$, we prove that $\Por\frestr L'$ forces that $\Por\frestr A/\Por\frestr L'$ is $\theta$-$\sqsubset$-good. Proceed by cases when $L'\neq A$.
   \begin{enumerate}[(1)]
      \item \emph{$A$ has a maximum $x$ and $A_x=A\cap L_x\in\hat{\Iwf}_x$.} Exactly like case (1) of the proof of Theorem \ref{PresTemp}.
      \item \emph{$A$ has a maximum $x$ but $A_x\notin\hat{\Iwf}_x$.} If $\dot{h}$ is a $\Por\frestr A$-name for a real, use (i) to find $\langle B_n\rangle_{n<\omega}$ and $C$. Then, by induction hypothesis and Lemmas \ref{PresGoodLimdir} and \ref{dirlimquot}, $\Por\frestr C/\Por\frestr L'$ is $\theta$-$\sqsubset$-good in $V^{\Por\upharpoonright L'}$. Any family of reals that witnesses this goodness for $\dot{h}$ also works for $\Por\frestr A/\Por\frestr L'$.
      \item \emph{$A$ does not have a maximum.} Proceed like in case (2) and use (ii).
   \end{enumerate}
\end{proof}

\begin{remark}\label{RemFsi}
    It is easy to note that any template iteration $\Por\frestr\langle\delta,\Iwf\rangle$ for a fsi (Example \ref{ExmpFSI}), where all the involved posets have the ccc, satisfies the conditions of the previous theorem for any initial segment $L'$ (clearly, any initial segment of $\delta$ is an ordinal), moreover, any $A\subseteq\delta$ that has a maximum $x$ satisfies $A\cap L_x\in\hat{\Iwf}_x$, so condition (i) becomes irrelevant in this case.
\end{remark}

The following theorem shows that, in many template iterations, new reals that are added at a certain stage cannot be added at other stages of the iteration. This is important in our applications, in relation with Lemma \ref{lemmagfrak}, to find the value of $\gfrak$ in some generic extension.

\begin{theorem}[New reals not added at other stages]\label{newrealint}
   Let $\Por\frestr\langle L,\bar{\Iwf}\rangle$ be a template iteration as in Example \ref{ExpTempItFund},
   $x\in L$ such that $\bar{L}_x:=L_x\cup\{x\}\in\hat{\Iwf}_z$ for all $z>x$ in $L$ and let $\dot{f}$ be a $\Por\frestr\bar{L}_x$-name of a real such that
   $\Vdash_{\Por\upharpoonright\bar{L}_x}\dot{f}\notin V^{\Por\upharpoonright L_x}$. Then, $\Por\frestr L$ forces that $\dot{f}\notin V^{\Por\upharpoonright(L\menos\{x\})}$.
\end{theorem}

This result is a direct consequence of Theorem \ref{unbrealint}, which is a more general result about the preservation of $\sqsubset$-unbounded reals. Fix $\sqsubset$ a relation as in Context \ref{ContextUnbd}, $M$ a transitive model of $\thzfc$, $\Por\in M$ and $\Qor$ posets such that $\Por\lessdot_M\Qor$ and let $c\in\omega^\omega$ be a $\sqsubset$-unbounded real over $M$. We are interested in the property ``\emph{$\Qor$ forces $c$ $\sqsubset$-unbounded over $M^{\Por}$}", that is, for every $\Por$-name $\dot{h}\in M$ for a real in $\omega^\omega$, $\Qor$ forces that $c\not\sqsubset\dot{h}$. More generally, we say that \emph{$\Qor$ preserves $\sqsubset$-unbounded reals over $M^\Por$} if it forces $c$ $\sqsubset$-unbounded over $M^\Por$ for all $c$ $\sqsubset$-unbounded over $M$. This property is essential for the applications of matrix iteration constructions in \cite{blsh,BF,mejia} and turns out to be useful in the framework of template iterations, as we illustrate in the remaining results.

\begin{lemma}\label{PresUnbLemma}
  \begin{enumerate}[(a)]
     \item \emph{(\cite[Thm. 7]{mejia})} Let $\Sor$ be a Suslin ccc poset coded in $M$ such that ``$\Sor$ is $\sqsubset$-good" is true in $M$. Then, $\Sor^V$ preserves $\sqsubset$-unbounded reals over $M^{\Sor^M}$.
     \item \emph{(\cite[Lemma 11]{BF})} If $\Por\in M$ is a poset, then $\Por$ preserves\footnote{Here, $\Por^M=\Por$.} $\sqsubset$-unbounded reals over $M^\Por$.
  \end{enumerate}
\end{lemma}
\begin{proof}
  We prove (b) as it is more general than the cited result. Let $c$ be a $\sqsubset$-unbounded real over $M$. Work within $M$. Let $\dot{h}$ be a $\Por$-name for a real in $\omega^\omega$ and fix $p\in\Por$ and $n<\omega$. Choose $\{p_k\}_{k<\omega}$ a decreasing sequence in $\Por$ and $g\in\omega^\omega$ such that $p_0=p$ and $p_k\Vdash\dot{h}\frestr k=g\frestr k$.

  Possibly outside $M$, $c\not\sqsubset g$, so $c\not\sqsubset_n g$, which implies that there is a $k<\omega$ such that $[g\frestr k]\cap(\sqsubset_n)_c=\varnothing$, this because $(\sqsubset_n)_c$ is a closed set. As $p_k\Vdash_M\dot{h}\frestr k=g\frestr k$, we get that $p_k\Vdash c\not\sqsubset_n\dot{h}$ (as $c$ may not be in $M$, this forcing relation may not interpreted in $M$ but it holds in the universe).
\end{proof}

This property of preserving unbounded reals works well under fsi and direct limits, as explained in the following result.

\begin{lemma}\label{PresUnbrealLimit}
  Let $c$ be a $\sqsubset$-unbounded real over $M$.
  \begin{enumerate}[(a)]
   \item Let $\Por_0\in M$ and $\Qor_0$ be posets, $\Pnm_1\in M$ a $\Por_0$-name for a poset and $\Qnm_1$ a $\Qor$-name for a poset such that $\Por_0\lessdot_M\Qor_0$ and $\Qor_0$ forces $\Pnm_1\lessdot_{M^{\Por_0}}\Qnm_1$. If $\Qor_0$ forces $c$ $\sqsubset$-unbounded over $M^{\Por_0}$ and $\Qor_0$ forces that $\Qnm_1$ forces $c$ $\sqsubset$-unbounded over $(M^{\Por_0})^{\Pnm_1}$, then $\Qor_0\ast\Qnm_1$ forces $c$ $\sqsubset$-unbounded over $M^{\Por_0\ast\Pnm_1}$.
   \item Let $I\in M$ be a directed set, $\langle\Por_i\rangle_{i\in I}\in M$ and $\langle\Qor_i\rangle_{i\in I}$ directed systems of posets such that
   \begin{enumerate}[(i)]
      \item for each $i\in I$, $\Por_i\lessdot_M\Qor_i$ and $\Qor_i$ forces $c$ $\sqsubset$-unbounded over $M^{\Por_i}$ and
      \item whenever $i\leq j$, the diagram $\langle\Por_i,\Por_j,\Qor_i,\Qor_j\rangle$ is correct with respect to $M$.
   \end{enumerate}
   Then, $\Qor$ forces $c$ $\sqsubset$-unbounded over $M^\Por$ where $\Por:=\limdir_{i\in I}\Por_i$ and $\Qor:=\limdir_{i\in I}\Qor_i$. Moreover, for any $i\in I$, $\langle\Por_i,\Por,\Qor_i,\Qor\rangle$ is correct with respect to $M$.
  \end{enumerate}
\end{lemma}
\begin{proof}
  (a) is obvious, so we prove (b). By Lemma \ref{dirlimEmb}, it is enough to prove that, if $\dot{h}\in M$ is a $\Por$-name for a real in $\omega^\omega$, then $\Vdash_{\Qor}c\not\sqsubset\dot{h}$. Assume, towards a contradiction, that there are $q\in\Qor$ and $n<\omega$ such that $q\Vdash_{\Qor}c\sqsubset_n\dot{h}$. Choose $i\in I$ such that $q\in\Qor_i$.

   Let $G$ be $\Qor_i$-generic over $V$ with $q\in G$. By assumption, $\Vdash_{\Qor/\Qor_i}c\sqsubset_n\dot{h}$. In $M[G\cap\Por_i]$, find $g\in\omega^\omega$ and a decreasing chain $\{p_k\}_{k<\omega}$ in $\Por/\Por_i$ such that $p_k\Vdash_{\Por/\Por_i}\dot{h}\frestr k=g\frestr k$. In $V[G\cap\Qor_i]$, by (i), $c\not\sqsubset g$, so there is a $k<\omega$ such that $[g\frestr k]\cap(\sqsubset_n)_c=\varnothing$. Then, as $\Por/\Por_i\lessdot_{M[G\cap\Por_i]}\Qor/\Qor_i$ by Lemma \ref{QuotEmb}, $p_k\Vdash_{\Qor/\Qor_i}[\dot{h}\frestr k]\cap(\sqsubset_n)_c=\varnothing$, that is, $p_k\Vdash_{\Qor/\Qor_i} c\not\sqsubset_n\dot{h}$, which is a contradiction.
\end{proof}

\begin{theorem}[Preservation of $\sqsubset$-unbounded reals]\label{unbrealint}
   Let $\Por\frestr\langle L,\bar{\Iwf}\rangle$ be a template iteration. Fix $x\in L$ such that $\bar{L}_x:=L_x\cup\{x\}\in\hat{\Iwf}_z$ for all $z>x$ in $L$ and let $\dot{c}$ be a $\Por\frestr\bar{L}_x$-name for a $\sqsubset$-unbounded real over $V^{\Por\upharpoonright L_x}$. Assume
   \begin{description}
    \item[$(\star)$] for any $y\in L$, $B\in\hat{\Iwf}_y$ with $\bar{L}_x\subseteq B$, if $\Por\frestr B$ forces $\dot{c}$ $\sqsubset$-unbounded over $V^{\Por\upharpoonright(B\menos\{x\})}$, then $\Por\frestr(B\cup\{y\})$ forces $\dot{c}$ $\sqsubset$-unbounded over $V^{\Por\upharpoonright((B\cup\{y\})\menos\{x\})}$.
   \end{description}
   Then, $\Por\frestr L$ forces $\dot{c}$ $\sqsubset$-unbounded over $V^{\Por\upharpoonright(L\menos\{x\})}$
\end{theorem}
\begin{proof}
   By induction on $\mathrm{Dp}(A)$ with $\bar{L}_x\subseteq A\subseteq L$, we prove that $\Por\frestr A$ forces $\dot{c}$ $\sqsubset$-unbounded over $V^{\Por\upharpoonright(A\menos\{x\})}$. We may assume $A\neq\bar{L}_x$. Proceed by cases.
   \begin{enumerate}[(1)]
      \item \emph{$A$ has a maximum $y$ and $A_y=A\cap L_y\in\hat{\Iwf}_y$.} So $x<y$. By induction hypothesis and $(\star)$, $\Por\frestr(A_y\cup\{y\})$ forces $\dot{c}$ $\sqsubset$-unbounded over $V^{\Por\upharpoonright((A_y\cup\{y\})\menos\{x\})}$.
      \item \emph{$A$ has a maximum $y$ but $A_y\notin\hat{\Iwf}_y$.} Clearly, $x<y$. If $\Bwf=\{B\subseteq A\ /\ B\cap L_y\in\Iwf_y\frestr A\textrm{\ and\ }\bar{L}_x\subseteq B\}$, then $\Por\frestr A=\limdir_{B\in\Bwf}\Por\frestr B$ and $\Por\frestr A\menos\{x\}=\limdir_{B\in\Bwf}\Por\frestr(B\menos\{x\})$. Moreover, if $B\subseteq B'$ are in $\Bwf$, $\langle\Por\frestr B\menos\{x\},\Por\frestr B'\menos\{x\},\Por\frestr B,\Por\frestr B'\rangle$ is correct. By induction hypothesis, $\Por\frestr B$ forces $\dot{c}$ $\sqsubset$-unbounded over $V^{\Por\upharpoonright(B\menos\{x\})}$ for all $B\in\Bwf$. Therefore, by Lemma \ref{PresUnbrealLimit}, $\Por\frestr A$ forces  $\dot{c}$ $\sqsubset$-unbounded over $V^{\Por\upharpoonright(A\menos\{x\})}$
      \item \emph{$A$ does not have a maximum.} Proceed like in the previous case.
   \end{enumerate}
\end{proof}

\begin{corollary}\label{PresUnbRealExpTemp}
   Let $\Por\frestr\langle L,\bar{\Iwf}\rangle$ be a template iteration as in Example \ref{ExpTempItFund} such that, for every $x\in L_S$ and $B\in\hat{\Iwf}_x$, $\Por\frestr B$ forces that $\Qnm^B_x$ is $\sqsubset$-good. Let $x\in L$ such that $\bar{L}_x:=L_x\cup\{x\}\in\hat{\Iwf}_z$ for all $z>x$ in $L$ and let $\dot{c}$ be a $\Por\frestr\bar{L}_x$-name for a $\sqsubset$-unbounded real over $V^{\Por\upharpoonright L_x}$. Then, $\Por\frestr L$ forces that $\dot{c}$ is $\sqsubset$-unbounded over $V^{\Por\upharpoonright(L\menos\{x\})}$.
\end{corollary}
\begin{proof}
   It is enough to prove $(\star)$ of Theorem \ref{unbrealint}. Assume that $y\in L$, $B\in\hat{\Iwf}_y$ with $\bar{L}_x\subseteq B$ and that $\Por\frestr B$ forces $\dot{c}$ $\sqsubset$-unbounded over $V^{\Por\upharpoonright(B\menos\{x\})}$. It is a direct consequence of Lemma \ref{PresUnbLemma} that $\Por\frestr(B\cup\{y\})$ forces $\dot{c}$ $\sqsubset$-unbounded over $V^{\Por\upharpoonright((B\cup\{y\})\menos\{x\})}$: if $y\in L_S$ use (a) of the Lemma or, if $y\in L_C$, use (b) when $C_y\subseteq B\menos\{x\}$ (the other case is easier because it implies $V^{\Por\upharpoonright((B\cup\{y\})\menos\{x\})}=V^{\Por\upharpoonright(B\menos\{x\})}$).
\end{proof}

\begin{proof}[Proof of Theorem \ref{newrealint}]
   Apply Lemma \ref{Presnewrealccc} and Corollary \ref{PresUnbRealExpTemp} for the relation defined in Example \ref{SubsecNewreal}.
\end{proof}

The following result shows conditions to force a cardinal invariant of the form $\dfrak_\sqsubset$ to be quite big.

\begin{theorem}\label{dfrakbig}
   Let $\theta$ be an uncountable regular cardinal and $\Por\frestr\langle L,\bar{\Iwf}\rangle$ a template iteration as in Example \ref{ExpTempItFund}. Assume:
   \begin{enumerate}[(i)]
    \item If $\dot{x}$ is a $\Por\frestr L$-name for a real, then it is a $\Por\frestr A$-name for some $A\subseteq L$ of size $<\theta$.
    \item For every $x\in L_S$ and $B\in\hat{\Iwf}_x$, $\Por\frestr B$ forces that $\Qnm^B_x$ is $\sqsubset$-good.
    \item $W\subseteq L$ is a cofinal subset of size $\lambda\geq\theta$ such that, for all $z\in W$, $L_z\in\Iwf_z$ and there is a $\Por\frestr(L_z\cup\{z\})$-name $\dot{c}_z$ for a $\sqsubset$-unbounded real over $V^{\Por\upharpoonright L_z}$.
   \end{enumerate}
   Then, $\Por\frestr L$ forces $\dfrak_\sqsubset\geq\lambda$ (as an inequality of ordinal numbers because $\lambda$ may not be a cardinal in the extension\footnote{This does not happen in our applications, though.}, in which case $\dfrak_\sqsubset$ is bigger than or equal to $|\lambda|^+$).
\end{theorem}
\begin{proof}
   Let $\dot{X}=\{\dot{x}_\xi\ / \xi<\delta\}$ be a $\Por\frestr L$-name of a set of reals with $\delta<\lambda$ an ordinal\footnote{Which could be even forced to be equal to $|\lambda|$ by some condition in $\Por\frestr L$.}. By (i), for each $\xi<\delta$ there is a $K_\xi\subseteq L$ of size $<\theta$ such that $\dot{x}_\xi$ is a $\Por\frestr K_\xi$-name. $K=\bigcup_{\xi<\delta}K_\xi$ has size $<\lambda$ and, clearly, $\dot{X}$ is a $\Por\frestr K$-name. Choose $z\in W\menos K$, so $\Por\frestr(L_z\cup\{z\})$ forces $\dot{c}_z$ $\sqsubset$-unbounded over $V^{\Por\upharpoonright L_z}$. Hence, by Corollary \ref{PresUnbRealExpTemp}, $\Por\frestr L$ forces that $\dot{c}_z$ is $\sqsubset$-unbounded over $V^{\Por\upharpoonright L\menos\{z\}}$, in particular, $\dot{X}$ is not a $\sqsubset$-dominating family.
\end{proof}

\begin{remark}
  In the applications of this paper, $\sqsubset$-unbounded reals will actually be given by Cohen reals. Note that, according to Context \ref{ContextUnbd}, any Cohen real over a model $M$ is $\sqsubset$-unbounded over $M$ (actually, this is the only reason why we want $(\sqsubset_n)^g$ to be nwd for any real $g$ and $n<\omega$). We also use Cohen reals to produce $\theta$-$\sqsubset$-unbounded families for $\theta$ uncountable regular.
\end{remark}

\begin{lemma}\label{AddCohen}
   Let $\langle V_\alpha\rangle_{\alpha\leq\theta}$ be an increasing sequence of transitive models of $\thzfc$ such that
   \begin{enumerate}[(i)]
      \item there is a Cohen real $c_\alpha\in \omega^\omega\cap V_{\alpha+1}$ over $V_\alpha$ for any $\alpha<\theta$,
      \item $\langle\omega^\omega\cap V_\alpha\rangle_{\alpha<\theta},\{c_\alpha\}_{\alpha<\theta}\in V_\theta$ and
      \item $\omega^\omega\cap V_\theta=\bigcup_{\alpha<\theta}\omega^\omega\cap V_\alpha$.
   \end{enumerate}
   Then, in $V_\theta$, $\{c_\alpha\ /\ \alpha<\theta\}$ is a $\theta$-$\sqsubset$-unbounded family.
\end{lemma}


\section{The groupwise-density number and fsi}\label{SecApplg}

With the fsi techniques of \cite{brendle}, the author constructed in \cite[Sect. 3]{mejia} and
\cite[Thm. 4.1-4.4]{mejia02} models with large continuum where the cardinal invariants defined in the introduction, with the exception of $\gfrak$ and $\afrak$, can take many different values. As the iterations used there can be defined as template iterations (see Example \ref{ExmpFSI}) and the preservation Theorem \ref{PresTemp2} applies, we can also get a value of $\gfrak$. We show how to obtain it in this section.

The following result is useful to determine a value of $\gfrak$ in generic extensions.

\begin{lemma}[Blass {\cite[Thm. 2]{Bl}},  see also {\cite[Lemma 1.17]{Br-Suslin}}]\label{lemmagfrak}
   Let $\theta$ be an uncountable regular cardinal, $\langle V_\alpha\rangle_{\alpha\leq\theta}$ an increasing sequence of transitive models of $\thzfc$ such that
   \begin{enumerate}[(i)]
      \item $[\omega]^\omega\cap(V_{\alpha+1}\menos V_{\alpha})\neq\varnothing$,
      \item $\langle[\omega]^\omega\cap V_\alpha\rangle_{\alpha<\theta}\in V_\theta$ and
      \item $[\omega]^\omega\cap V_\theta=\bigcup_{\alpha<\theta}[\omega]^\omega\cap V_\alpha$.
   \end{enumerate}
   Then, in $V_\theta$, $\gfrak\leq\theta$.
\end{lemma}

For the results in this section, fix uncountable regular cardinals $\mu_1\leq\mu_2\leq\mu_3\leq\nu$ and a cardinal $\lambda\geq\nu$. In what follows of this paper, $\mu\theta$ always denotes ordinal multiplication of cardinal numbers $\mu$ and $\theta$.

\begin{theorem}\label{gfrakcofmu3}
    If $\lambda^{<\mu_3}=\lambda$ then for each item (a),(b) and (c) there is a ccc poset that forces the corresponding item, $\add(\Nwf)=\mu_1$, $\cov(\Nwf)=\mu_2$, $\pfrak=\sfrak=\gfrak=\mu_3$ and $\non(\Nwf)=\rfrak=\cfrak=\lambda$.
    \begin{enumerate}[(a)]
       \item $\non(\Mwf)=\mu_3$ and $\cov(\Mwf)=\lambda$.
       \item $\add(\Mwf)=\cof(\Mwf)=\nu$.
       \item $\bfrak=\mu_3$, $\non(\Mwf)=\cov(\Mwf)=\nu$ and $\dfrak=\lambda$.
    \end{enumerate}
\end{theorem}
\begin{proof}
\begin{enumerate}[(a)]
\item Consider $\langle\lambda,\bar{\Iwf}\rangle$ the template corresponding to a fsi of length $\lambda$ (see Example \ref{ExmpTemp}). For each $\alpha<\lambda$ enumerate $[\alpha]^{<\mu_3}:=\{C_{\alpha,\beta}\}_{\beta<\lambda}$. Fix a bijection $g:\lambda\to\lambda^3$ such that $g^{-1}(\alpha,\beta,\gamma)\geq\alpha,\beta,\gamma$ for any $\alpha,\beta,\gamma<\lambda$. Consider a template iteration $\Por\frestr\langle\lambda,\bar{\Iwf}\rangle$ as in Example \ref{ExpTempItFund} such that $L_S=\left\{\xi<\lambda\ /\ \exists_{\delta}(\xi=4\delta)\right\}$, $\Sor_\xi=\Cor$ for $\xi\in L_S$ and, for each $\xi\in L_C$, if $\xi=4\delta_\xi+r_\xi$ with $0<r_\xi<4$ and $g(\delta_\xi)=(\alpha,\beta,\gamma)$, then
      \begin{itemize}
      \item $C_\xi:=C_{\alpha,\beta}$.
      \item $\{\Locnm_{\alpha,\beta,\eta}\}_{\eta<\lambda}$ is an enumeration of the $\Por\frestr C_{\alpha,\beta}$-names for \emph{all} the subposets of $\Loc^{V^{\Por\upharpoonright C_{\alpha,\beta}}}$ of size $<\mu_1$.
      \item $\{\Bnm_{\alpha,\beta,\eta}\}_{\eta<\lambda}$ is an enumeration of the $\Por\frestr C_{\alpha,\beta}$-names for \emph{all} the subalgebras of $\Bor^{V^{\Por\upharpoonright C_{\alpha,\beta}}}$ of size $<\mu_2$.
      \item $\{\dot{\Fwf}_{\alpha,\beta,\eta}\}_{\eta<\lambda}$ is an enumeration of the $\Por\frestr C_{\alpha,\beta}$-names for \emph{all} the filter bases of size $<\mu_3$.
      \item If $r_\xi=1$, then $\Qnm_\xi=\Locnm_{\alpha,\beta,\gamma}$.
      \item If $r_\xi=2$, then $\Qnm_\xi=\Bnm_{\alpha,\beta,\gamma}$.
      \item If $r_\xi=3$, then $\Qnm_\xi=\Mor_{\dot{\Fwf}_{\alpha,\beta,\gamma}}$.
   \end{itemize}
   By Lemma \ref{CondSupp}, $\Por\frestr\lambda$ is ccc. We prove that $\Por\frestr\lambda$ forces the following.
   \begin{itemize}
     \item $\cov(\Nwf)=\mu_2$. To force $\geq$, let $\{\dot{N}_\eta\}_{\eta<\mu}$ be a sequence of $\Por\frestr L$-names of Borel null sets with $\mu<\mu_2$. Then, there is an $\alpha<\lambda$ such that all the $\dot{N}_\eta$ ($\eta<\mu$) are $\Por\frestr\alpha$-names (i.e., their Borel codes), so, by Lemma \ref{CondSupp}(c) (applied to $\theta=\mu_3$), we can find a $\beta<\lambda$ such that these are $\Por\frestr C_{\alpha,\beta}$-names. In $V^{\Por\upharpoonright C_{\alpha,\beta}}$, find a model $M$ of a large finite fragment of $\thzfc$ such that $\{N_\eta\ /\ \eta<\mu\}\subseteq M$ and $|M|\leq\mu$. Now, back in $V$, find $\gamma<\lambda$ such that $\Bnm_{\alpha,\beta,\gamma}$ is a $\Por\frestr C_{\alpha,\beta}$-name for $\Bor^{\dot{M}}$. Thus, with $\xi=4g^{-1}(\alpha,\beta,\gamma)+2$, $\Por\frestr(C_{\alpha,\beta}\cup\{\xi\})$ adds a random real over $\dot{M}$, in particular, this real is not in $\bigcup_{\eta<\mu}\dot{N}_\eta$.

     To force the converse inequality, note that $\Por\frestr\mu_2$ adds a $\mu_2$-$\pitchfork$-unbounded family of Cohen reals of size $\mu_2$ by Lemma \ref{AddCohen}. Also, this unbounded family is preserved in the $\Por\frestr\lambda$ extension because $\Por\frestr\lambda/\Por\frestr\mu_2$ is forced by $\Por\frestr\mu_2$ to be $\mu_2$-$\pitchfork$-good by Example \ref{SubsecUnbd}, Lemma  \ref{smallPlus} and Theorem \ref{PresTemp2} (see also Remark \ref{RemFsi}). Therefore, $\Por\frestr\lambda$ forces $\cov(\Nwf)\leq\bfrak_\pitchfork\leq\mu_2$.

     \item $\add(\Nwf)=\mu_1$. A similar argument as before, but use that $\Por\frestr\lambda/\Por\frestr\mu_1$ is $\mu_1$-$\in^*$-good for $\leq$ and, for the converse, use the small subposets of $\Loc$ to force that any family of reals of size $<\mu_1$ is localized by a single slalom in $S(\omega,id)$.

     \item $\pfrak=\non(\Mwf)=\mu_3$. To force $\mu_3\leq\pfrak$ use the small Mathias-Prikry posets and argue like the proof of $\cov(\Nwf)\geq\mu_2$. To force $\non(\Mwf)\leq\mu_3$, use that $\Por\frestr\lambda/\Por\frestr\mu_3$ is $\mu_3$-$\eqcirc$-good.

     \item $\cov(\Mwf)=\cfrak=\lambda$. As $|\Por\frestr\lambda|\leq\lambda$ and $\lambda^\omega=\lambda$, then $\cfrak\leq\lambda$ is clearly forced. To force $\lambda\leq\cov(\Mwf)$, apply Theorem \ref{dfrakbig} to $\theta=\mu_3$, $W=L_S$ and to the relation $\eqcirc$.

     \item $\gfrak=\mu_3$. As $\pfrak\leq\gfrak$, we only need to force $\gfrak\leq\mu_3$. Consider a partition $\{A_\eta\}_{\eta<\mu_3}$ of $\lambda$ such that each $A_\eta$ has size $\lambda$ an intersects $L_S$. For each $\eta<\mu_3$, put $E_\eta=\bigcup_{\xi<\eta}A_\xi$. Now, if $G$ is $\Por\frestr\lambda$-generic over $V$, let $W_{\mu_3}=V[G]$ and $W_\eta=V[G\cap\Por\frestr E_\eta]$ for each $\eta<\mu_3$. It is enough to see that the conditions of Lemma \ref{lemmagfrak} hold for the sequence $\{W_\eta\}_{\eta\leq\mu_3}$. Conditions (ii) and (iii) follow from the fact that, in $V$, $\Por\frestr L=\limdir_{\eta<\mu_3}\Por\frestr E_\eta$, this because of Lemma \ref{CondSupp}(b). To see (i), in $V$, let $\xi\in A_\eta\cap L_S=(E_{\eta+1}\menos E_\eta)\cap L_S$, so $\Por\frestr((E_\eta\cap\xi)\cup\{\xi\})$ adds a Cohen real over $V^{\Por\upharpoonright(E_\eta\cap\xi)}$ and, by Theorem \ref{newrealint}, this new real does not belong to $V^{\Por\upharpoonright E_\eta}=W_\eta$.
   \end{itemize}

   \item Let $\langle L=\lambda\nu,\bar{\Iwf}\rangle$ be the template corresponding to a fsi of length $\lambda\nu$. Fix a bijection $h:\lambda\to\lambda\times\lambda\times3$ and, for each $\alpha<\nu$, $\alpha\neq0$, enumerate $[\lambda\alpha]^{<\mu_3}:=\{C_{\alpha,\beta}\}_{\beta<\lambda}$. Perform a template iteration $\Por\frestr\langle L,\bar{\Iwf}\rangle$ as in Example \ref{ExpTempItFund} such that $L_S=\lambda\cup\{\lambda\alpha\ /\ 0<\alpha<\nu,\}$, $\Sor_\xi=\Dor$ for each $\xi\in L_S$ and, for each $\xi\in L_C$, if $\xi=\lambda\alpha+1+\eta$ for some $0<\alpha<\nu$ and $\eta<\lambda$ with $h(\eta)=(\beta,\gamma,r)$, then
    \begin{itemize}
      \item $C_\xi:=C_{\alpha,\beta}$
      \item $\{\Locnm_{\alpha,\beta,\eta}\}_{\eta<\lambda}$ is an enumeration of the $\Por\frestr C_{\alpha,\beta}$-names for \emph{all} the subposets of $\Loc^{V^{\Por\upharpoonright C_{\alpha,\beta}}}$ of size $<\mu_1$.
      \item $\{\Bnm_{\alpha,\beta,\eta}\}_{\eta<\lambda}$ is an enumeration of the $\Por\frestr C_{\alpha,\beta}$-names for \emph{all} the subalgebras of $\Bor^{V^{\Por\upharpoonright C_{\alpha,\beta}}}$ of size $<\mu_2$.
      \item $\{\dot{\Fwf}_{\alpha,\beta,\eta}\}_{\eta<\lambda}$ is an enumeration of the $\Por\frestr C_{\alpha,\beta}$-names for \emph{all} the filter bases of size $<\mu_3$.
      \item If $r=0$, then $\Qnm_\xi=\Locnm_{\alpha,\beta,\gamma}$.
      \item If $r=1$, then $\Qnm_\xi=\Bnm_{\alpha,\beta,\gamma}$.
      \item if $r=2$, then $\Qnm_\xi=\Mor_{\dot{\Fwf}_{\alpha,\beta,\gamma}}$.
    \end{itemize}
    Like in (a), $\Por\frestr L$ is ccc and forces $\add(\Nwf)=\mu_1$, $\cov(\Nwf)=\mu_2$, $\pfrak=\gfrak=\mu_3$ and $\cfrak\leq\lambda$. We show that $\Por\frestr L$ forces the following.
    \begin{itemize}
      \item $\sfrak\leq\mu_3$. $\Por\frestr\mu_3$ adds a $\mu_3$-$\propto$-unbounded family that is preserved in $\Por\frestr L$ because $\Por\frestr L/\Por\frestr\mu_3$ is $\mu_3$-$\propto$-good. So, in the final extension, $\sfrak=\bfrak_\propto\leq\mu_3$.

      \item $\lambda\leq\non(\Nwf)$. By Theorem \ref{dfrakbig} applied to $\theta=\mu_3$, $W=L_S$ (recall that Hechler forcing adds Cohen reals) and to $\pitchfork$.

      \item $\lambda\leq\rfrak$. Same argument as before, but with the relation $\propto$.

      \item $\add(\Mwf)=\cof(\Mwf)=\nu$. From Lemma \ref{CondSupp}(b), $\Por\frestr L=\limdir_{\alpha<\nu}\Por\frestr\lambda\alpha$ so, by Lemma \ref{AddCohen}, it adds a $\nu$-$\eqcirc$-unbounded family of Cohen reals of size $\nu$, which makes $\non(\Mwf)\leq\nu\leq\cov(\Mwf)$. We are left to prove $\nu\leq\bfrak$ and $\dfrak\leq\nu$. Indeed, $\Por\frestr\lambda(\alpha+1)$ adds a dominating real $\dot{d}_\alpha$ over $V^{\Por\upharpoonright\alpha}$ for any $\alpha<\nu$, so $\Por\frestr L$ forces that $\{\dot{d}_\alpha\ /\ \alpha<\nu\}$ is a dominating family and that any family of reals of size $<\nu$ can be dominated by some $\dot{d}_\alpha$.
    \end{itemize}

    \item Perform the same template iteration of (b), but only change $\Sor_\xi=\Eor$ for $\xi\in L_S$. With the same techniques as before, $\Por\frestr L$ forces the desired statements. Just notice that, to force $\bfrak,\sfrak\leq\mu_3$ and $\lambda\leq\dfrak,\rfrak$, we should use $\mu_3$-$\vartriangleright$-goodness.
   \end{enumerate}
\end{proof}

The same type of argument as in the previous proof leads to the following results.

\begin{theorem}\label{gfrakcofmu2}
   Assume $\lambda^{<\mu_2}=\lambda$. It is consistent with $\thzfc$ that $\add(\Nwf)=\mu_1$, $\pfrak=\bfrak=\sfrak=\gfrak=\mu_2$, $\cov(\Nwf)=\non(\Mwf)=\cov(\Mwf)=\non(\Nwf)=\nu$ and $\dfrak=\rfrak=\cfrak=\lambda$.
\end{theorem}
\begin{proof}
   As in the proof of Theorem \ref{gfrakcofmu3}(b), perform a fsi (as a template iteration) of length $\lambda\nu$ alternating between: $\Bor$ and $\Cor$ for coordinates in $L_S$, subposets of $\Loc$ of size $<\mu_1$ and $\Mor_\Fwf$ with a filter base $\Fwf$ of size $<\mu_2$ for coordinates in $L_C$ ($L_S$ and $L_C$ are the same as in the proof of Theorem \ref{gfrakcofmu3}(b)).
\end{proof}

\begin{theorem}\label{gfrakcofmu1}
   Assume $\lambda^{<\mu_1}=\lambda$. For each item (a) and (b) there is a ccc poset that forces the corresponding item,
   $\pfrak=\gfrak=\mu_1$, $\cov(\Nwf)=\add(\Mwf)=\cof(\Mwf)=\non(\Nwf)=\nu$ and $\cfrak=\lambda$.
   \begin{enumerate}[(a)]
      \item $\add(\Nwf)=\mu_1$ and $\cof(\Nwf)=\lambda$.
      \item $\add(\Nwf)=\cof(\Nwf)=\nu$.
   \end{enumerate}
\end{theorem}
\begin{proof}
  For each case, perform a fsi (as a template iteration) of length $\lambda\nu$, where $L_S$ and $L_C$ are the same as in the proof of Theorem \ref{gfrakcofmu3}(b), alternating between:
  \begin{enumerate}[(a)]
     \item $\Bor$ and $\Dor$ for coordinates in $L_S$, subposets of $\Loc$ of size $<\mu_1$ and Mathias-Prikry type posets with filter bases of size $<\mu_1$ for coordinates in $L_C$.
     \item $\Loc$ for coordinates in $L_S$ and Mathias-Prikry type posets with filter bases of size $<\mu_1$ for coordinates in $L_C$.
  \end{enumerate}
\end{proof}

In this last theorem, we do not know how to obtain values for $\sfrak$, $\rfrak$ and $\ufrak$. For example, we would like to obtain $\sfrak\leq\mu_1$ and $\lambda\leq\rfrak$ in (a), but the preservation properties related to $\sfrak$ (and $\rfrak$) that we know so far do not work for $\Bor$ and $\Dor$ at the same time, i.e., $\Dor$ is $\propto$-good but $\Bor$ is not (see the paragraph after Lemma 2.18 in \cite{mejia02}) and, although $\Bor$ is $\vartriangleright$-good, $\Dor$ is not because it adds dominating reals.


\section{Proof of the main result}\label{SecAppl}

\begin{theorem}[Main result]\label{AppSplitting}
   Let $\kappa$ be a measurable cardinal, $\theta<\kappa<\mu<\lambda$ all regular uncountable cardinals such that $\theta^{<\theta}=\theta$ and $\lambda^\kappa=\lambda$. Then, there exists a ccc poset forcing that $\sfrak=\theta<\bfrak=\mu<\afrak=\cfrak=\lambda$. Moreover, this poset forces $\cov(\Nwf)\leq\pfrak=\gfrak=\theta$, $\add(\Mwf)=\cof(\Mwf)=\mu$ and $\non(\Nwf)=\rfrak=\lambda$.
\end{theorem}

We don't get the equality $\cov(\Nwf)=\theta$ in the model constructed for this result, but in Theorem \ref{AppMany} we explain how to modify the construction to force, additionally, $\theta\leq\add(\Nwf)$.

Fix $\Dwf$ a non-principal $\kappa$-complete ultrafilter on $\kappa$.

\begin{definition}[Appropriate template iteration]\label{DefApprTempItSplitting}
A template iteration $\Por\frestr\langle L,\Iwf\rangle$ is \emph{appropriate} (for the proof of Theorem \ref{AppSplitting}) if the following conditions hold.
\begin{enumerate}[(I)]
   \item $\lambda\mu\subseteq L$ is cofinal in $L$, $|L|=\lambda$ and $0=\min(L)$.
   \item Every $x\in L$ has an immediate successor and, for $\xi\in\lambda\mu$, $\xi+1$ is the immediate successor of $\xi$.
   \item If $\gamma\in\lambda\mu$ is a limit ordinal of cofinality $\neq\kappa$, then $\gamma=\sup_L\left\{\alpha\in\lambda\mu\ /\ \alpha<\gamma\right\}$.
   \item $L$ is partitioned into three disjoint sets $L_H$, $L_F$ and $L_T$.
   \item $L_H\cap\lambda\mu$ has size $\lambda$ and it is unbounded in $\lambda\mu$.
   \item For each $\alpha\in\lambda\mu$, $L_\alpha\in\Iwf_\alpha$.
   \item If $X\in[L]^{<\theta}$ then $|\Iwf(X)|<\theta$.
   \item For $x\in L_H$ and $B\in\hat{\Iwf}_x$, $\Qnm_x^B$ is a $\Por\frestr B$-name for $\Dor^{V^{\Por\upharpoonright B}}$.
   \item For $x\in L_F$ there is a fixed $C_x\in\hat{\Iwf}_x$ of size $<\theta$ and $\dot{\Fwf}_x$ a $\Por\frestr C_x$-name for a filter base on $\omega$ of size $<\theta$ such that, for every $B\in\hat{\Iwf}_x$,
       \[\Qnm_x^B=\left\{\begin{array}{ll}
          \Mor_{\dot{\Fwf}_x} & \textrm{if $C_x\subseteq B$,}\\
          \dot{\mathds{1}} & \textrm{otherwise.}
       \end{array}\right.\]
   \item For $x\in L_T$ and $B\in\hat{\Iwf}_x$, $\Qnm_x^B$ is the trivial poset.
   \item Given $\dot{\Fwf}$ a $\Por\frestr L$-name for a filter base on $\omega$ of size $<\theta$, there exists an $x\in L_F$ such that $\Vdash_{\Por\upharpoonright L}\dot{\Fwf}=\dot{\Fwf}_x$.
\end{enumerate}
\end{definition}

Notice that an appropriate template iteration $\Por\frestr\langle L,\Iwf\rangle$ satisfies the hypothesis of Lemma \ref{CondSupp}, so it has the Knaster condition and, whenever $A\subseteq L$, $p\in\Por\frestr A$ and $\dot{h}$ is a $\Por\frestr A$-name for a real, there is a $K\subseteq A$ of size $<\theta$ such that $p\in\Por\frestr K$ and $\dot{h}$ is a $\Por\frestr K$-name. We first prove that any appropriate template iteration forces the statements of Theorem \ref{AppSplitting} with the possible exception of $\afrak=\lambda$. Moreover, it forces:

\begin{itemize}
  \item $\add(\Mwf)=\cof(\Mwf)$. Let $D$ be a cofinal subset of $L_H\cap\lambda\mu$ of size $\mu$. For $\alpha\in D$, $\Por\frestr L_{\alpha+1}$ adds a dominating real $d_\alpha$, as well as a Cohen real $c_\alpha$, over $V^{\Por\upharpoonright L_\alpha}$. As $\Por\frestr L=\limdir_{\alpha\in D}\Por\frestr L_\alpha$, these reals form a dominating family in $V^{\Por\upharpoonright L}$ and any set of reals of size $<\mu$ is bounded by some $d_\alpha$, so $\bfrak=\dfrak=\mu$. On the other hand, $\non(\Mwf)\leq\mu\leq\cov(\Mwf)$ because, by Lemma \ref{AddCohen}, $\{c_\alpha\ /\ \alpha\in D\}$ is a $\mu$-$\eqcirc$-unbounded family.

  \item $\cfrak\leq\lambda$. Because $|L|\leq\lambda$.

  \item $\non(\Nwf)=\rfrak=\lambda$. By Theorem \ref{dfrakbig} applied to $W=L_H\cap\lambda\mu$ and to the relations $\propto$ and $\pitchfork$.

  \item $\theta\leq\pfrak$. Let $\dot{\Fwf}$ be a $\Por\frestr L$-name for a filter base of size $<\theta$. By (XI), there is an $x\in L_F$ such that $\Por\frestr L$ forces $\dot{\Fwf}=\dot{\Fwf}_x$. On the other hand, as $\Por\frestr(C_x\cup\{x\})\simeq\Por\frestr C_x\ast\Mor_{\dot{\Fwf}_x}$, it adds a pseudo-intersection of $\dot{\Fwf}_x$.

  \item $\sfrak\leq\theta$. By (V), find $\alpha<\lambda\mu$ minimal such that $|L_H\cap\alpha|=\theta$. Notice that $\cf(\alpha)=\theta$ and, by (III), $\Por\frestr L_\alpha=\limdir_{\varepsilon\in L_H\cap\alpha}\Por\frestr L_\varepsilon$. Then, by Lemma \ref{AddCohen}, $\Por\frestr L_\alpha$ adds a $\theta$-$\propto$-unbounded family of size $\theta$ that is preserved in $\Por\frestr L$ because $\Por\frestr L/\Por\frestr L_\alpha$ is $\theta$-$\propto$-good by Theorem \ref{PresTemp} and (VII).

  \item $\cov(\Nwf)\leq\theta$. Same argument as before, but with the relation $\pitchfork$. 

  \item $\gfrak\leq\theta$. By (V), find an increasing sequence $\{E_\alpha\}_{\alpha<\theta}$ of subsets of $L$ such that its union is $L$ and, for each $\alpha<\theta$, $L_H\cap\lambda\mu\cap(E_{\alpha+1}\menos E_\alpha)\neq\varnothing$. Let $G$ be $\Por\frestr L$-generic over $V$ and put $W_\theta=V[G]$ and $W_\alpha=V[G\cap\Por\frestr E_\alpha]$ for any $\alpha<\theta$. It is enough to show that $\{W_\alpha\}_{\alpha\leq\theta}$ satisfies the conditions of Lemma \ref{lemmagfrak}. Conditions (ii) and (iii) hold because, in $V$, $\Por\frestr L=\limdir_{\alpha<\theta}\Por\frestr E_\alpha$. To see (i), choose a $\beta\in L_H\cap\lambda\mu\cap(E_{\alpha+1}\menos E_\alpha)$ and note that $\Por\frestr((E_\alpha\cap L_\beta)\cup\{\beta\})$ adds a Cohen real over $V^{\Por\upharpoonright(E_\alpha\cap L_\beta)}$ so, by Theorem \ref{newrealint} and (VI), that real does not belong to $V^{\Por\upharpoonright E_\alpha}=W_\alpha$.
\end{itemize}

Therefore, to prove Theorem \ref{AppSplitting}, it is enough to construct an appropriate template iteration that forces $\afrak\geq\lambda$. For this, by recursion, we construct a chain of appropriate template iterations of length $\lambda$ such that ultrapowers are taken in the successor steps (so we can use Lemma \ref{DestrMad} to destroy mad families of size strictly between $\kappa$ and $\lambda$). Before proceeding with this construction, we explain how to deal with the successor and limit steps in a general way.

Fix an appropriate template iteration $\Por\frestr\langle L,\Iwf\rangle$. Consider, from the context of Lemma \ref{UtrapowTemp}, the templates $\Iwf^*$ and $\Iwf^\dagger$ associated to the ultrapower $L^*$ of the linear order $L$. We show how to construct, in a canonical way, an appropriate template iteration $\Por^\dagger\frestr\langle L^*,\Iwf^\dagger\rangle$ that is forcing equivalent to the ultrapower of $\Por\frestr L$.

As $\cf(\lambda\mu)=\mu>\kappa$, it is easy to note that $\lambda\mu$ is still cofinal in $L^*$. By standard arguments with ultrapowers, conditions (I)-(III) of Definition \ref{DefApprTempItSplitting}
are satisfied by $L^*$. Let $L^*_H:=L^\kappa_H/\Dwf$, $L^*_F$ and $L^*_T$ defined likewise. (IV)-(VII) for $\langle L^*,\Iwf^*\rangle$ and $\langle L^*,\Iwf^\dagger\rangle$ are clear, the last one by Lemma \ref{RestUltrapowTemp}. Notice that $L^*_H\cap L=L_H$, $L^*_F\cap L=L_F$ and $L^*_T\cap L=L_T$.

\begin{lemma}\label{ApprSplUltrapow}
   There is a template iteration $\Por^*\frestr\langle L^*,\Iwf^*\rangle$ such that (VIII)-(X) hold and, for any $\bar{A}=[\{A_\alpha\}_{\alpha<\kappa}]\subseteq L^*$, there is an (onto) dense embedding $F_{\bar{A}}:\prod_{\alpha<\kappa}\Por\frestr A_\alpha/\Dwf\to\Por^*\frestr\bar{A}$ such that, for any
   $\bar{D}=[\{D_\alpha\}_{\alpha<\kappa}]\subseteq\bar{A}$, $F_{\bar{D}}\subseteq F_{\bar{A}}$.
\end{lemma}
\begin{proof}
   To define the desired template iteration $\Por^*\frestr\langle L^*,\Iwf^*\rangle$, it is enough to show how $C_{\bar{x}}$ and $\dot{\Fwf}_{\bar{x}}$ are defined for (IX). We put $C_{\bar{x}}:=\bar{C}_{\bar{x}}=[\{C_{x_\alpha}\}_{\alpha<\kappa}]$ which is clearly in $\hat{\Iwf}^*_{\bar{x}}$. By induction on $\mathrm{Dp}^{\bar{\Iwf}^*}(\bar{A})$ for $\bar{A}$ of the form $[\{A_\alpha\}_{\alpha<\kappa}]\subseteq L^*$ we
   \begin{enumerate}[(i)]
    \item define $\dot{\Fwf}_{\bar{x}}$ for those $\bar{x}\in\bar{A}\cap L^*_F$ such that $\bar{C}_{\bar{x}}\subseteq\bar{A}$ (by Lemma \ref{UpsilonTemp}, $\mathrm{Dp}^{\bar{\Iwf}^*}(\bar{C}_{\bar{x}})<\mathrm{Dp}^{\bar{\Iwf}^*}(\bar{A})$),
    \item construct $F_{\bar{A}}$ and
    \item prove $F_{\bar{D}}\subseteq F_{\bar{A}}$ for all $\bar{D}=[\{D_\alpha\}_{\alpha<\kappa}]\subseteq\bar{A}$.
   \end{enumerate}
   It is possible that some $\dot{\Fwf}_{\bar{x}}$ were defined at previous stages of the induction, but, as we illustrate in the following construction, this name only depends on $\bar{C}_{\bar{x}}$, so there are no inconsistencies.

   If $\bar{x}\in\bar{A}\cap L^*_F$ and $\bar{C}_{\bar{x}}\subseteq\bar{A}$, define the $\Por^*\frestr \bar{C}_{\bar{x}}$-name $\dot{\Fwf}^*_{\bar{x}}:=\langle\dot{\Fwf}_{x_\alpha}\rangle_{\alpha<\kappa}/\Dwf$ in the following way. By ccc-ness, for $\Dwf$-many $\alpha$ there is a cardinal $\nu_\alpha<\theta$ such that $\Por\frestr C_{x_\alpha}$ forces $|\dot{\Fwf}_{x_\alpha}|\leq\nu_\alpha$. As $\theta<\kappa$, there is a cardinal $\nu<\theta$ such that $\nu_\alpha=\nu$ for $\Dwf$-many $\alpha$. For those $\alpha$, let $\dot{\Fwf}_{x_\alpha}:=\left\{\dot{U}_{\alpha,\xi}\ /\ \xi<\nu\right\}$ and put, for $\xi<\nu$, $\dot{U}^*_\xi:=\langle\dot{U}_{\alpha,\xi}\rangle_{\alpha<\kappa}/\Dwf$, which is a $\prod_{\alpha<\kappa}\Por\frestr C_{x_\alpha}/\Dwf$-name for a real. But, by induction hypothesis, this ultraproduct is equivalent to $\Por^*\frestr\bar{C}_{\bar{x}}$ by the function $F_{\bar{C}_{\bar{x}}}$, so let $\dot{\Fwf}^*_{\bar{x}}$ be a $\Por^*\frestr\bar{C}_{\bar{x}}$-name for $\{\dot{U}^*_\xi\ /\ \xi<\nu\}$. By Lemma \ref{Sigma1-1andUltrapow} and $\kappa$-completeness of $\Dwf$, $\Por^*\frestr\bar{C}_{\bar{x}}$ forces that $\dot{\Fwf}^*_{\bar{x}}$ is a filter base (i.e., formed by infinite subsets of $\omega$ and closed under finite intersections).

   Now, we construct $F_{\bar{A}}$. Let $\bar{p}\in\prod_{\alpha<\kappa}\Por\frestr A_\alpha/\Dwf$, that is, $p_\alpha\in\Por\frestr A_\alpha$ for $\Dwf$-many $\alpha$. Let $x_\alpha:=\max(\dom(p_\alpha))$, so there exists a $B_\alpha\in\Iwf_{x_\alpha}\frestr A$ such that $p_\alpha\frestr L_{x_\alpha}\in\Por\frestr B_\alpha$ and $p_\alpha(x_\alpha)$ is a $\Por\frestr B_\alpha$-name for a condition in $\Qnm_{x_\alpha}^{B_\alpha}$. Let $\bar{r}:=\langle p_\alpha\frestr L_{x_\alpha}\rangle_{\alpha<\kappa}/\Dwf$ and $p(\bar{x}):=\langle p_\alpha(x_\alpha)\rangle_{\alpha<\kappa}/\Dwf$ which is a $\Por^*\frestr\bar{B}$-name for a real (by induction hypothesis) where $\bar{B}:=[\{B_\alpha\}_{\alpha<\kappa}]\in\Iwf^*_x\frestr\bar{A}$. By Lemma \ref{Sigma1-1andUltrapow}, considering cases on (VIII), (IX) and (X), $p(\bar{x})$ is actually a $\Por^*\frestr\bar{B}$-name for a condition in $\Qnm_{\bar{x}}^{*\bar{B}}$, so define $F_{\bar{A}}(\bar{p})=F_{\bar{B}}(\bar{r})\widehat{\ }\langle p(\bar{x})\rangle_{\bar{x}}$ (by induction hypothesis (iii), this definition does not depend on $\bar{B}$). (iii) follows easily from this construction.
\end{proof}

A template iteration $\Por^\dagger\frestr\langle L^*,\bar{\Iwf}^\dagger\rangle$ can be defined in a similar way so that $\Por^\dagger\frestr\bar{A}$ is forcing equivalent to $\prod_{\alpha<\kappa}\Por\frestr A_\alpha/\Dwf$ for any $\bar{A}=[\{A_\alpha\}_{\alpha<\kappa}]\subseteq L^*$. Notice that $\langle L^*,\bar{\Iwf}^\dagger\rangle$ is a $\theta$-innocuous extension of $\langle L^*,\bar{\Iwf}^*\rangle$ (Lemma \ref{UtrapowTemp}) so, by Lemma \ref{InnEqv}, \emph{$\Por^\dagger\frestr\bar{A}$ is forcing equivalent to $\Por^*\frestr\bar{A}$.}

\begin{lemma}\label{ApprSplUltrapow2}
    $\Por^*\frestr\langle L^*,\bar{\Iwf}^*\rangle$ and $\Por^\dagger\frestr\langle L^*,\bar{\Iwf}^\dagger\rangle$ are appropriate template iterations. Moreover, $\Por\frestr A$ is forcing equivalent to $\Por^*\frestr A$ and $\Por^\dagger\frestr A$ for any $A\subseteq L$.
\end{lemma}
\begin{proof}
   We prove condition (XI) for both iterations. As every set in $\Iwf^\dagger_{\bar{x}}$ is contained in some set in $\Iwf^*_{\bar{x}}$ for any $\bar{x}\in L^*$, it is enough to consider only the case for $\bar{\Iwf}^*$. Let $\dot{\bar{\Fwf}}$ be a $\Por^*\frestr L^*$-name for a filter base on $\omega$ of size $<\theta$. By ccc-ness, find $\nu<\theta$ such that $\dot{\bar{\Fwf}}$ is forced to have size $\leq\nu$ and let $\dot{\bar{\Fwf}}=\left\{\dot{\bar{U}}_\epsilon\ /\ \epsilon<\nu\right\}$. Each $\dot{\bar{U}}_\epsilon$ is of the form $\langle \dot{U}_{\alpha,\epsilon}\rangle_{\alpha<\kappa}/\Dwf$ where each $\dot{U}_{\alpha,\epsilon}$ is a $\Por\frestr L$-name for an infinite subset of $\omega$.
   As $\nu<\theta$, by Lemma \ref{Sigma1-1andUltrapow}, $\dot{\Fwf}_\alpha:=\left\{\dot{U}_{\alpha,\epsilon}\ /\ \epsilon<\nu\right\}$ is a $\Por\frestr L$-name for a filter base for $\Dwf$-many $\alpha$, so, by (XI), there exists an $x_\alpha\in L_F$ such that $\Vdash_{\Por\upharpoonright L}\dot{\Fwf}_\alpha=\dot{\Fwf}_{x_\alpha}$. Then, $\Vdash_{\Por^*\upharpoonright L^*}\dot{\bar{\Fwf}}=\dot{\Fwf}^*_{\bar{x}}$.

  Note that $\bar{C}_x=C_x$ and $\dot{\Fwf}^*_x=\dot{\Fwf}_x$ for any $x\in L$ by the construction in the proof of Lemma \ref{ApprSplUltrapow}. Therefore, for $A\subseteq L$, $\Por\frestr A\simeq\Por^*\frestr A\simeq\Por^\dagger\frestr A$ follows from Lemma \ref{InnEqv} because, for $x\in A$, $\Iwf^*_x\frestr A=\Iwf^\dagger_x\frestr A$
   and $\langle A,\bar{\Iwf}^*\frestr A\rangle$ is a strongly $\theta$-innocuous extension of $\langle A,\bar{\Iwf}\frestr A\rangle$.
\end{proof}

Now we turn to our approach of the limit step. Let $\delta$ be a limit ordinal and consider a chain $\{\langle L^\alpha,\bar{\Iwf}^\alpha\rangle\}_{\alpha<\delta}$ of templates and appropriate template iterations
$\Por^\alpha\frestr\langle L^\alpha,\bar{\Iwf}^\alpha\rangle$ with the following properties for all $\alpha<\beta<\delta$.

\begin{enumerate}[(1)]
   \item $\langle L^\beta,\bar{\Iwf}^\beta\rangle$ is a strongly $\theta$-innocuous extension of
         $\langle L^\alpha,\bar{\Iwf}^\alpha\rangle$.
   \item For $x\in L^\alpha$, its immediate successor in $L^\alpha$ is the same as in $L^\beta$.
   \item $L^\alpha_H=L^\beta_H\cap L^\alpha$ and $L_F^\alpha\subseteq L_F^\beta$.
   \item If $x\in L^\alpha_F$ then $C^\alpha_x=C^\beta_x$ and $\Vdash_{\Por^\beta\upharpoonright C^\beta_x}\dot{\Fwf}^\alpha_x=\dot{\Fwf}^\beta_x$.
\end{enumerate}

Note that Corollary \ref{EmbCor} implies that $\Por^\alpha\frestr X\lessdot\Por^\beta\frestr X$ for any $X\subseteq L_\alpha$.

Consider $L^\delta$ and the templates $\bar{\Iwf}$ and $\bar{\Jwf}$ as in the context of Lemma \ref{ChainTemp}. Let $L^\delta_H=\bigcup_{\alpha<\delta}L^\alpha_H$, $L^\delta_F$ a set disjoint from $L^\delta_H$ that contains $\bigcup_{\alpha<\delta}L^\alpha_F$ and $L_T^\delta=L^\delta\menos(L^\delta_H\cup L^\delta_F)$. Properties (I)-(V) are straightforward for $L^\delta$, moreover, properties (1)-(3) hold for any $\alpha<\delta$ by replacing $\beta$ by $\delta$ and for both templates $\bar{\Iwf}$ and $\bar{\Jwf}$. (VII) also holds for both templates because of Lemma \ref{RestChainTemp}. Nevertheless, (VI) holds for $\bar{\Jwf}$ but it need not hold for $\bar{\Iwf}$.

We show how to define template iterations $\Por_0^\delta\frestr\langle L^\delta,\bar{\Iwf}\rangle$ and
$\Por_1^\delta\frestr\langle L^\delta,\bar{\Jwf}\rangle$ such that they are close to be appropriate and have nice agreement with the template iterations $\Por^\alpha\frestr\langle L^\alpha,\bar{\Iwf}^\alpha\rangle$ for $\alpha<\delta$. We just need to be specific about (IX) in order to define the iterations. In the case of $\bar{\Iwf}$, for $x\in L^\delta_F$, if there is some $\alpha<\delta$ such that $x\in L^\alpha_F$ put $C^\delta_x:=C^\alpha_x$ and $\dot{\Fwf}^\delta_x=\dot{\Fwf}^\alpha_x$. Otherwise, choose $C^\delta_x$ and $\dot{\Fwf}^\delta_x$ freely.

To define $\Por_0^\delta\frestr\langle L^\delta,\bar{\Iwf}\rangle$ it is necessary to proceed by induction and guarantee that, for each $\alpha<\delta$, $\Por^\alpha\frestr X\lessdot\Por^\delta_0\frestr X$ for any $X\subseteq L^\alpha$, which is justified by Corollary \ref{EmbCor}. Notice that (4) holds by replacing $\beta$ by $\delta$.

$\Por_1^\delta\frestr\langle L^\delta,\bar{\Jwf}\rangle$ is defined in the same way by just ensuring to make the same choices of $C^\delta_x$ and $\dot{\Fwf}^\delta_x$ as for $\bar{\Iwf}$. The conclusions of the previous paragraph hold in the same way. However, it is not always the case that property (XI) holds, moreover, it will depend on the particular ``free'' choices of $C^\delta_x$ and $\dot{\Fwf}^\delta_x$. However, there is one case in which (XI) holds for both template iterations.

\begin{lemma}\label{ApprSplChain}
   Assume $\cf(\delta)\geq\theta$ and that $L^\delta_F=\bigcup_{\alpha<\delta}L^\alpha_F$. Then, both template iterations $\Por_0^\delta\frestr\langle L^\delta,\bar{\Iwf}\rangle$ and
   $\Por_1^\delta\frestr\langle L^\delta,\bar{\Jwf}\rangle$ are forcing equivalent and satisfy (XI). Moreover, $\Por_0^\delta\frestr L^\delta=\limdir_{\alpha<\delta}\Por^\alpha\frestr L^\alpha$ and $\Por_1^\delta\frestr\langle L^\delta,\bar{\Jwf}\rangle$ is appropriate.
\end{lemma}
\begin{proof}
   By Lemmas \ref{ChainTemp} and \ref{InnEqv}, both template iterations are equivalent.

   We claim that $\Por_0^\delta\frestr A=\limdir_{\alpha<\delta}\Por^\alpha\frestr(A\cap L^\alpha)$ for any $A\subseteq L^\delta$. Proceed by induction on $\mathrm{Dp}^{\bar{\Iwf}}(A)$. Let $p\in\Por^\delta_0\frestr A$ and $x=\max(\dom(p))$, so there exists a $B\in\Iwf_x\frestr A$ such that $p\frestr L^\delta_x\in\Por^\delta_0\frestr B$ and $p(x)$ is a $\Por^\delta_0\frestr B$-name for a real in $\Qnm^{\delta,B}_{0,x}$. By induction hypothesis and ccc-ness, find $\alpha<\delta$ such that $x\in L^\alpha$, $B\in\Iwf^\alpha_x\frestr(A\cap L^\alpha)$, $p\frestr L^\delta_x=p\frestr L^\alpha_x\in\Por^\alpha\frestr B$ and $p(x)$ is a $\Por^\alpha\frestr B$-name for a real. In the case that $x\in L^\delta_H$, then $x\in L^\alpha_H$ and so $p(x)$ is a $\Por^\alpha\frestr B$-name for a condition in Hechler forcing; in the case that $x\in L^\delta_F$, by increasing $\alpha$ if necessary, $x\in L^\alpha_F$, so clearly $p(x)$ is a $\Por^\alpha\frestr B$-name for a condition in $\Mor_{\dot{\Fwf}^\alpha_x}$ if $C^\alpha_x\subseteq A$ (by increasing $B$ such that $C^\alpha_x\subseteq B$), otherwise, $p(x)$ is a name for the trivial condition; and if $x\in L^\delta_T$, then $x\in L^\alpha_T$ and so $p(x)$ is clearly a name for the trivial condition. Therefore, in any case, $p\in\Por^\alpha\frestr(A\cap L^\alpha)$.

   It remains to prove (XI) for the iteration along $\bar{\Iwf}$. Let $\dot{\Fwf}$ be a $\Por_0^\delta\frestr L^\delta$-name for a filter base on $\omega$ of size $<\theta$. As $\cf(\delta)\geq\theta$ and $\Por_0^\delta\frestr L^\delta=\limdir_{\alpha<\delta}\Por^\alpha\frestr L^\alpha$, there is an $\alpha<\delta$ such that $\dot{\Fwf}$ is a $\Por^\alpha\frestr L^\alpha$-name, so there exists an $x\in L^\alpha_F\subseteq L^\delta_F$ such that $\dot{\Fwf}$ is forced to be equal to $\dot{\Fwf}^\alpha_x=\dot{\Fwf}^\delta_x$.
\end{proof}

We use the indexed template $\bar{\Iwf}$ to prove the preceding result but, for the construction of the model of the main result, $\bar{\Jwf}$ is the one used for the limit step.

\begin{proof}[Proof of Theorem \ref{AppSplitting}]
   Fix a bijective enumeration $\lambda\menos\{0\}:=\left\{\tau_{\alpha,\beta}\ /\ \alpha,\beta<\lambda\right\}$, a bijection $g:\lambda\to\lambda\times\theta$ and an increasing enumeration $\langle\delta_\alpha\rangle_{\alpha<\lambda}$ of $0$ and all the limit ordinals below $\lambda$ that have cofinality $<\theta$. For an ordered pair $z=(x,y)$, denote $(z)_0:=x$ and $(z)_1:=y$.

   By recursion on $\gamma\leq\lambda$, define a chain of templates $\{\langle L^\gamma,\bar{\Iwf}^\gamma\rangle\}_{\gamma\leq\lambda}$ such that they satisfy conditions (I)-(VII) and (1)-(3). We also require, for $\gamma<\delta_\alpha$, that $\left\{\lambda\xi+\tau_{\alpha,\beta}\ /\ 0<\xi<\mu\right.$  and $\left.\beta<\lambda\right\}\subseteq L^\gamma_T$
   and, when $\delta_\alpha\leq\gamma$, $\left\{\lambda\xi+\tau_{\alpha,\beta}\ /\ 0<\xi<\mu\right.$ and $\left.\beta<\lambda\right\}\subseteq L^\gamma_F$.

   Let $L^0:=\lambda\mu$ and let $\bar{\Iwf}^0$ be the template corresponding to a fsi of length $\lambda\mu$ (Example \ref{ExmpTemp}(2)). Put $L^0_H:=\lambda\cup\left\{\lambda\xi\ /\ 0<\xi<\mu\right\}$, $L^0_F:=\left\{\lambda\xi+\tau_{0,\beta}\ /\ 0<\xi<\mu\right.$ and $\left.\beta<\lambda\right\}$ and $L^0_T:=\left\{\lambda\xi+\tau_{\alpha,\beta}\ /\ 0<\xi<\mu\right.$, $\alpha,\beta<\lambda$ and $\left.\alpha\neq0\right\}$. Clearly, conditions (I)-(VII) hold for $\langle L^0,\bar{\Iwf}^0\rangle$.

   Given $\langle L^\gamma,\bar{\Iwf}^\gamma\rangle$, let $\langle L^{\gamma+1},\bar{\Iwf}^{\gamma+1}\rangle:=\langle (L^\gamma)^*,(\bar{\Iwf}^\gamma)^\dagger\rangle$
   as in the previous discussion of ultrapowers. Clearly, (I)-(VII) hold and, moreover,
   $L^{\gamma+1}_H\cap L^\gamma=L^\gamma_H$, likewise for $L^{\gamma+1}_F$ and $L^{\gamma+1}_T$, so (2) and (3) hold. For (1), recall that $\langle L^{\gamma+1},\bar{\Iwf}^{\gamma+1}\rangle$ is a strongly $\theta$-innocuous extension of $\langle L^\gamma,\bar{\Iwf}^\gamma\rangle$, so we needed to prove that it is also a strongly $\theta$-innocuous extension of $\langle L^\beta,\bar{\Iwf}^\beta\rangle$ for each
   $\beta<\gamma$. The non-trivial part is to see that, for $x\in L^\beta$, $\Iwf^{\gamma+1}_x\frestr L^\beta\subseteq\Iwf^{\gamma+1}_x$. If $A\in\Iwf^{\gamma+1}_x\frestr L^\beta$ then, as $L^\beta\subseteq L^\gamma$, there exists an $\bar{H}\in(\Iwf^\gamma_x)^*$ such that
   $A=\bar{H}\cap L^\beta$. Thus, $A=[\{H_\alpha\cap L^\beta\}_{\alpha<\kappa}]\in(\Iwf^\gamma_x)^*$ because $\Iwf^\gamma_x\frestr L^\beta\subseteq \Iwf^\gamma_x$.

   If $\delta$ is a limit ordinal, define $L^\delta:=\bigcup_{\beta<\delta}L^\beta$ and $\bar{\Iwf}^\delta:=\bar{\Jwf}$ according to the previous discussion about chains of templates. Hence, it is only needed to define
   $L^\delta_F$. If $\cf(\delta)\geq\theta$ put $L^\delta_F:=\bigcup_{\beta<\delta}L^\beta_F$, otherwise put $L^\delta_F:=\bigcup_{\beta<\delta}L^\beta_F\cup\left\{\lambda\xi+\tau_{\alpha,\beta}\ /\ 0<\xi<\mu\right.$ and $\left.\beta<\lambda\right\}$ where $\alpha<\lambda$ is such that $\delta=\delta_\alpha$. Clearly, (I)-(VII) and (1)-(3) are satisfied.

   By recursion on $\gamma\leq\lambda$, define appropriate template iterations $\Por^\gamma\frestr\langle L^\gamma,\bar{\Iwf}^\gamma\rangle$ such that (4) is satisfied for them.

   For each $\xi<\mu$, $\xi\neq0$, enumerate $[\lambda\xi]^{<\theta}:=\left\{C^0_{\xi,\alpha}\ /\ \alpha<\lambda\right\}$. Define the iteration $\Por^0\frestr\langle L^0,\bar{\Iwf}^0\rangle$ as follows (implicitly, by induction on $\mathrm{Dp}^{\bar{\Iwf}^0}$).

   \begin{itemize}
     \item For each $0<\xi<\mu$ and $\alpha<\lambda$, let $\langle\dot{\Fwf}^0_{\xi,\alpha,\eta}\rangle_{\eta<\theta}$ be an enumeration of \underline{all} the $\Por^0\frestr C^0_{\xi,\alpha}$-names of filter bases on $\omega$ of size $<\theta$. This can be done because $\theta^{<\theta}=\theta$ and $|\Por^0\frestr X|\leq\theta$ when $|X|<\theta$. Indeed, as $|\Iwf^0(X)|<\theta$ and $\cfrak\leq\theta$, by induction on $\mathrm{Dp}^{\bar{\Iwf}^0}(A)$ for $A\subseteq X$ and Theorem \ref{TempIt}(d), it is not difficult to see that $|\Por^0\frestr A|\leq\theta$.
     \item For $x\in L^0$ and $B\in\hat{\Iwf}_x$, $\Qnm^{0,B}_x$ is defined as indicated in (VIII)-(X). For (IX), if $x\in L^0_F$ then $x=\lambda\xi+\tau_{0,\beta}$ for some $0<\xi<\mu$ and $\beta<\lambda$, so put
         $C^0_x:=C^0_{\xi,(g(\beta))_0}$ and $\dot{\Fwf}^0_x:=\dot{\Fwf}^0_{\xi,g(\beta)}$.
   \end{itemize}

   To see that $\Por^0\frestr\langle L^0,\bar{\Iwf}^0\rangle$ is appropriate, it remains to prove (XI). Indeed, if $\dot{\Fwf}$ is a $\Por^0\frestr L^0$-name for a filter of size $<\theta$ then, by Lemma \ref{CondSupp}(c), find $C\in[L^0]^{<\theta}$ such that $\dot{\Fwf}$ is a $\Por^0\frestr C$-name. Clearly, there exist $0<\xi<\mu$ and $\alpha<\lambda$ such that $C=C^0_{\xi,\alpha}$, so
   $\dot{\Fwf}$ is forced to be equal to $\dot{\Fwf}^0_{\xi,\alpha,\eta}=\dot{\Fwf}^0_{\lambda\xi+\tau_{0,g^{-1}(\alpha,\eta)}}$ for some $\eta<\theta$.

   The iteration $\Por^{\gamma+1}\frestr\langle L^{\gamma+1},\bar{\Iwf}^{\gamma+1}\rangle$ is defined from
   $\Por^\gamma\frestr\langle L^\gamma,\bar{\Iwf}^\gamma\rangle$ as explained in the previous discussion of ultrapowers. From the proofs of Lemmas \ref{ApprSplUltrapow} and \ref{ApprSplUltrapow2}, (4) is satisfied.

   For $\delta\leq\lambda$ limit consider two cases. When $\cf(\delta)\geq\theta$, define
   $\Por^\delta\frestr\langle L^\delta,\bar{\Iwf}^\delta\rangle$ as in Lemma \ref{ApprSplChain}. So assume that $\cf(\delta)<\theta$, that is, $\delta=\delta_\epsilon$ for some $\epsilon<\lambda$. For each $\xi<\mu$, $\xi\neq0$, enumerate $[L^\delta_{\lambda\xi}]^{<\theta}:=\left\{C^\delta_{\xi,\alpha}\ /\ \alpha<\lambda\right\}$. As it was done for $\Por^0\frestr L^0$, define the iteration corresponding to $\delta$ as follows.

   \begin{itemize}
      \item For each $0<\xi<\mu$ and $\alpha<\lambda$, let $\langle\dot{\Fwf}^\delta_{\xi,\alpha,\eta}\rangle_{\eta<\theta}$ be an enumeration of \underline{all} the $\Por^\delta\frestr C^\delta_{\xi,\alpha}$-names of filter bases on $\omega$ of size $<\theta$.
     \item For $x\in L^\delta$ and $B\in\hat{\Iwf}_x$, $\Qnm^{0,B}_x$ is defined as indicated in (VIII)-(X). For (IX), if $x\in\bigcup_{\beta<\delta}L^\beta_F$, let $C^\gamma_x:=C^\beta_x$ and $\dot{\Fwf}^\gamma_x:=\dot{\Fwf}^\beta_x$ for some $\beta<\delta$ such that $x\in L^\beta_F$ (this does not depend on the chosen $\beta$ by (4)); if $x\in L^\delta_F\menos\bigcup_{\beta<\delta}L^\beta_F$ then $x=\lambda\xi+\tau_{\epsilon,\beta}$ for some $0<\xi<\mu$ and $\beta<\lambda$, so put $C^\delta_x:=C^\delta_{\xi,(g(\beta))_0}$ and
         $\dot{\Fwf}^\delta_x:=\dot{\Fwf}^\delta_{\xi,g(\beta)}$.
   \end{itemize}

   According to the previous discussion with chains of templates, it remains to show condition (XI) for
   $\Por^\delta\frestr L^\delta$, but this follows from the same argument as in the case of $\Por^0\frestr L^0$.

   From the discussion following Definition \ref{DefApprTempItSplitting}, it remains to prove that $\Por^\lambda\frestr L^\lambda$ forces $\afrak\geq\lambda$. Indeed, let $\dot{\Awf}$ be a $\Por^\lambda\frestr L^\lambda$-name for an a.d. family of size $\nu<\lambda$ with $\nu\geq\kappa$ (we do not need to consider a.d. families of size $<\kappa$ because $\bfrak$ is forced to be equal to $\mu>\kappa$ and $\bfrak\leq\afrak$ is true in $\thzfc$). By Lemma \ref{ApprSplChain}, $\Por^\lambda\frestr L^\lambda=\limdir_{\alpha<\lambda}\Por^\alpha\frestr L^\alpha$, so there exists an $\alpha<\lambda$ such that $\dot{\Awf}$ is a $\Por^\alpha\frestr L^\alpha$-name. As $\Por^{\alpha+1}\frestr L^{\alpha+1}$ is forcing equivalent to the ultrapower of $\Por^\alpha\frestr L^\alpha$ (Lemma \ref{ApprSplUltrapow}), by Lemma \ref{DestrMad} this ultrapower forces that $\dot{\Awf}$ is not mad, and so does $\Por^\lambda\frestr L^\lambda$.
\end{proof}

A small modification of our construction leads to the following result.

\begin{theorem}\label{AppMany}
   With the same hypotheses of Theorem \ref{AppSplitting}, there is a ccc poset that forces $\add(\Nwf)=\cov(\Nwf)=\pfrak=\sfrak=\gfrak=\theta<\add(\Mwf)=\cof(\Mwf)=\mu<\non(\Nwf)
   =\afrak=\rfrak=\cfrak=\lambda$.
\end{theorem}
\begin{proof}
  We modify our definition of an appropriate iteration to include subposets of $\Loc$ of size $<\theta$. Consider a new set $L_A$ in the partition of $L$ where $L_A$ has similar properties as $L_F$, that is, like in (IX), for $x\in L_A$ there is a fixed $C_x\in\hat{\Iwf}_x$ of size $<\theta$ and a $\Por\frestr C_x$-name $\Locnm_x$ for a subposet of $\Loc^{V^{\Por\upharpoonright C_x}}$ of size $<\theta$ such that, for $B\in\hat{\Iwf}_x$, $\Qnm^B_x=\Locnm_x$ if $C_x\subseteq B$, or $\Qnm^B_x$ is the trivial poset otherwise. Also, we consider a property like (XI):
   \begin{enumerate}[(I)]
   \setcounter{enumi}{11}
      \item For any  $\Por\frestr L$-name $\Rnm$ of a subposet of $\Loc^{V^{\Por\upharpoonright L}}$ of size $<\theta$, there exists an $x\in L_A$ such that $\Por\frestr L$ forces $\Rnm=\Locnm_x$.
   \end{enumerate}
   Such an appropriate iteration that satisfies (XII) forces, additionally, $\theta\leq\add(\Nwf)$, so it forces the statements of this theorem with the possible exception of $\afrak=\lambda$. A chain of appropriate iterations that satisfy (XII) can be constructed in the same way as in the proof of Theorem \ref{AppSplitting} to guarantee that there is an appropriate iteration forcing $\afrak=\lambda$.
\end{proof}

\begin{remark}\label{RemSh2}
   Shelah's model discussed in Remark \ref{ShTempMainThm} satisfies $\gfrak=\aleph_1$ by the same argument as for appropriate iterations.
\end{remark}


\section{Questions}\label{SecQ}

\begin{question}
   Can we solve Problem \ref{MainProb}(3) with respect to $\thzfc$, that is, if $\aleph_1<\theta<\mu<\lambda$ are uncountable regular cardinals, is it consistent with $\thzfc$ that $\sfrak=\theta<\mu=\bfrak<\afrak=\lambda$?
\end{question}

As Theorem \ref{AppSplitting} was an extension of Shelah's argument for the consistency of $\dfrak<\afrak$ modulo a measurable, we can try to generalize the isomorphism-of-names argument to our context in order to obtain a proof without the measurable. However, since we need to include Mathias-Prikry type posets with filter bases of size $<\theta$ in many coordinates of the template (as done in Section \ref{SecAppl}), the iteration may not be uniform enough to do an isomorphism-of-names argument.

\begin{question}
   Let $\kappa$ be a measurable cardinal and $\theta_0<\theta_1<\theta<\kappa<\mu<\lambda$ uncountable regular cardinals. Assuming $\mathrm{GCH}$, is there a ccc poset that forces $\add(\Nwf)=\theta_0$, $\cov(\Nwf)=\theta_1$, $\sfrak=\pfrak=\gfrak=\theta$, $\bfrak=\dfrak=\mu$ and $\afrak=\rfrak=\cfrak=\lambda$ are true?
\end{question}

The natural attempt to construct such a model would be to include subposets of $\Loc$ of size $<\theta_0$ and subposets of $\Bor$ of size $<\theta_1$ in the appropriate iterations of Theorem \ref{AppSplitting} (like it was done for $\Loc$ in Theorem \ref{AppMany}). The only problem is that Theorem \ref{PresTemp} does not work anymore to prove that certain quotients of the constructed posets are $\theta_0$-$\in^*$-good and $\theta_1$-$\pitchfork$-good, so $\add(\Nwf)\leq\theta_0$ and $\cov(\Nwf)\leq\theta_1$ are not guaranteed to hold in the final forcing extension. The reason of this is that (c) of Lemma \ref{CondSupp} is not satisfied for $\theta_0$ and $\theta_1$. A way to get the necessary goodness would be to prove, by induction on $\alpha\leq\lambda$, that the hypotheses of Theorem \ref{PresTemp2} hold for the $\alpha$-th template iteration in the chain. This can be done for the basic and the successor steps, but the limit step for $\delta$ with $\cf(\delta)<\theta$ is problematic.

Shelah's proof of the consistency of $\dfrak<\afrak$ and $\ufrak<\afrak$ modulo a measurable involves an easier construction that does not appeal to templates but to iterations (forcing) equivalent to a fsi (\cite{shelah}, see also \cite{brendle3}). In this way, the construction of the chain of iterations in Section \ref{SecAppl} can be simplified, but we do not know whether preservation results, similar to those in Section \ref{SecPres}, can be obtained without looking at the template structure.

\section*{Acknowledgments}

The author is very thankful to professor J. Brendle for all the guidance and support provided during the research that precedes this paper. He kindly taught the author Shelah's theory of iterated forcing along a template and offered a lot of his time for discussions that concluded in the results that are presented in this text. The author is also grateful to him for letting include his proof of Lemma \ref{randomCorrPres} and for his help with proof reading and grammar corrections. In this latter aspect, the author is also thankful to Vera Fischer. The author wants to thank the anonymous referee for his valuable comments that improved the presentation of this paper.


\end{document}